\documentclass[twoside,11pt]{article}
\usepackage{jmlr2e}
\usepackage{amsmath, amssymb, graphicx}
\usepackage[font=small,labelfont=md,textfont=it]{caption}
\usepackage{floatrow,float}
\usepackage[titletoc, title]{appendix}
\usepackage[hyperpageref]{backref}
\usepackage{longtable}
\usepackage{pgfplots}
\usepackage{diagbox}
\usepackage{booktabs,makecell,multirow}
\usepackage[colorlinks,linkcolor=blue,citecolor=blue]{hyperref}
\usepackage[capitalise, nosort]{cleveref}
\usepackage{cases,color}
\usepackage{algorithm}
\usepackage{algorithmic}
\usepackage{amsmath,bm}
\usepackage{tcolorbox}
\usepackage{array}

\crefname{equation}{}{}
\crefname{lem}{Lemma}{Lemmas}
\crefname{thm}{Theorem}{Theorems}

\newcommand{\dd}{\,{\rm d}}
\newcommand{\bs}{\boldsymbol}
\newcommand{\dual}[1]{\left\langle {#1} \right\rangle}
\newcommand{\proxi}[0]{ {\bf prox}}

\newcommand{\dom}[0]{ {\bf dom\,}}
\newcommand{\argmin}[0]{ {\mathrm {argmin}\,}}

\newcommand{\nm}[1]{\left\lVert {#1} \right\rVert}
\newcommand{\snm}[1]{\left\lvert {#1} \right\rvert}
\newcommand{\ssnm}[1]
{
	\left\vert\kern-0.25ex
	\left\vert\kern-0.25ex
	\left\vert
	{#1}
	\right\vert\kern-0.25ex
	\right\vert\kern-0.25ex
	\right\vert
}

\makeatletter
\def\spher@harm#1{%
	\vbox{\hbox{%
			\offinterlineskip
			\valign{&\hb@xt@2\p@{\hss$##$\hss}\vskip.2ex\cr#1\crcr}%
		}\vskip-.36ex}%
}
\def\gshone{\spher@harm{.}}
\def\gshtwo{\spher@harm{.&.}}
\def\gshthree{\spher@harm{.&.&.}}
\let\gsh\spher@harm
\makeatother

\newtheorem{coro}{Corollary}[section]

\newtheorem{lem}{Lemma}[section]

\newtheorem{thm}{Theorem}[section]

\newcounter{mnote}
\let\oldmarginpar\marginpar
\renewcommand\marginpar[1]{\-\oldmarginpar[\raggedleft\footnotesize #1]%
	{\raggedright\footnotesize #1}}

\makeatletter\def\@captype{table}\makeatother

\ShortHeadings{Unified Convergence Analysis via Strong Lyapunov Functions}{Chen and Luo}
\firstpageno{1}

\begin{document}

\title{A Unified Convergence Analysis of First Order Convex Optimization Methods via Strong Lyapunov Functions}

\author{\name Long Chen \email chenlong@math.uci.edu\\
       \addr Department of Mathematics \\
University of California at Irvine\\
Irvine, CA 92697, USA
       \AND
       \name Hao Luo \email luohao@math.pku.edu.cn \\
       \addr School of Mathematical Sciences\\
      Peking University\\
Beijing, 100871, China}

%\editor{Kevin Murphy and Bernhard Sch{\"o}lkopf}
\editor{}

\maketitle

\begin{abstract}%   <- trailing '%' for backward compatibility of .sty file
We present a unified convergence analysis for first order convex optimization methods using the concept of strong Lyapunov conditions. Combining this with suitable time scaling factors, we are able to handle both convex and strong convex cases, and establish contraction properties of Lyapunov functions for many existing ordinary differential equation models. Then we derive prevailing first order optimization algorithms, such as proximal gradient methods, heavy ball methods (also known as momentum methods), Nesterov accelerated gradient methods, and accelerated proximal gradient methods from numerical discretizations of corresponding dynamical systems. We also apply strong Lyapunov conditions to the discrete level and provide a more systematical analysis framework. Another contribution is  a novel second order dynamical system called Hessian-driven Nesterov accelerated gradient flow which can be used to design and analyze accelerated first order methods for smooth and non-smooth convex optimizations. 
%as possible as we can.
\end{abstract}


%\tableofcontents

\section{Introduction}
We consider first order methods for solving the unconstrained convex minimization problem
\begin{equation}\label{eq:min}
	\min_{x\in V}\,f(x).
\end{equation}
Above and throughout $V$ is a Hilbert space and $V^*$ is its dual space. 
%with inner product $(\cdot,\cdot)$ and $\dual{\cdot,\cdot}$ denotes the dual pairing between its dual space $V^*$ and $V$. 
First order optimization methods for solving \eqref{eq:min} regains the popularity in the application of large scale machine learning~\citep{bottou_optimization_2018}.

Denoted by $x^*$ a global minimizer of $f$ which is unique when $f$ is strictly convex.  Instead of solving the Euler equation $\nabla f(x^*) = 0$, we consider  continuous optimization methods which start from some
ordinary differential equation (ODE)
%
%\LH{Comment: I prefer the general vector form \cref{eq:introode} since it appeals natural to include the second-order systems mentioned below. For simplicity, no need to give the underlying space of $\bm x$.}
\begin{equation}\label{eq:introode}
	\bm x'(t) = \mathcal G(\bm x(t)),\quad t>0.
\end{equation}
Here in general, $\bm x$ is a vector-valued function of time variable $t$ and $\mathcal G$ is a vector field, which can be the negative gradient $-\nabla f$ or any reasonable alternate. We assume $\bm x^*$ (containing $x^*$ as a component when $\bs x$ is a vector)
%\mnote{Then $\bs x^*$ is not well defined.}\mnote{Assume $\mathcal G(\bm x^*) = 0$} 
is an equilibrium point of the autonomous dynamical system \cref{eq:introode}, i.e. $\mathcal G(\bm x^*) = 0$, and ideally this shall imply $\nabla f(x^*)=0$.

%Here $x: \mathbb R_+ \to V$ is a $V$-valued function of time variable $t$ and $\mathcal G: V\to V^*$ is a vector field. We assume $x^*$ is an equilibrium point of the dynamical system defined by \cref{eq:introode}, i.e. $\mathcal G(x^*) = 0$. 

A simple example is $\mathcal G = - \nabla f$, with which the ODE \cref{eq:introode} becomes the well known gradient flow $x' = -\nabla f(x)$. For this standard model, the explicit (forward) Euler scheme leads to the gradient descent method for solving \cref{eq:min}, and the implicit (backward) Euler scheme corresponds to the proximal point algorithm~\citep{guler_convergence_1991,rockafellar_monotone_1976}. When extended to the composite case $f=h+g$ with smooth $h$ and nonsmooth $g$, the semi-implicit discretization, also known as the forward-backward scheme, recovers the proximal gradient method~\citep{parikh_proximal_2014}.
%
%There is a long history using ordinary differential equations for developing optimization methods. \LC{Briefly mention some important references. }
%In the simplest example, $\mathcal G(x) = - \nabla f(x)$ and \cref{eq:introode} becomes the well known gradient flow $x' = -\nabla f(x)$. Explicit (forward) Euler scheme gives the gradient descent method for solving the gradient flow and implicit (backward) Euler scheme corresponds to the proximal method. 

Besides the gradient flow, many more (second order) dynamic systems, such as the heavy ball model~\citep{polyak_methods_1964}, the asymptotic vanishing damping (AVD) system~\citep{Su;Boyd;Candes:2016differential}, 
the dynamic inertial of Newton system~\citep{alvarez_second-order_2002}, and the ODE based  variational method~\citep{wibisono_variational_2016,Wilson:2021} etc, have been developed to explain the acceleration mechanism and design new first order optimization methods as well. In~\citet{luo_chen_differential_2019}, we have proposed the so-called Nesterov accelerated gradient flow and provided an explanation on the acceleration phenomena by using the so-called $A$-stability of ODE solvers. All the models mentioned here admit the unified first order form \cref{eq:introode} with different vector fields.
% $\mathcal G$.
%where $x$ should be understood as a vector function, 
%aiming to explain the accelerated first order methods. 
%First we find a transformation to shrink the spectral radius and then a Gauss-Seidel iteration is used to discretize the ODE system   

The long time decay property of the continuous problem \cref{eq:introode} is very important and gives insights on the rate of convergence of the corresponding optimization methods~\citep{Su;Boyd;Candes:2016differential}.
%those continuous ODE models,  will given insight on the rate of convergence of optimization methods~\citep{Su;Boyd;Candes:2016differential}. 
Appropriate discretizations of the above ODE systems will lead to accelerated first order methods such as the heavy ball method~\citep{polyak_methods_1964}, Nesterov's accelerated gradient method~\citep{Nesterov1983}, and the accelerated proximal gradient method~\citep{beck_fast_2009,tseng_on_accelerated_Seattle_2008} etc. The analysis of discrete algorithms, however, is not a straightforward translation from the continuous level. A standard work flow is to design a Lyapunov function and establish the decay of that Lyapunov function; see~\citet{shi_understanding_2018,Wilson:2021,Siegel:2019} among many others. This procedure often involves tricky algebraic manipulation and tedious calculations. Indeed in~\citet[Section 6]{Su;Boyd;Candes:2016differential}, the authors conclude that {\em ``a general theory mapping properties of ODEs into corresponding properties for discrete updates would be a welcome advance."}

%\LH{Comment: As we used the general vector form \cref{eq:introode}, the following part is rewritten w.r.t. $\bm x$. But I still kept the original one after {\rm `\%'}. We can discuss more about this to keep the proper one.}

%%%%%%%%%%%%%%%%%THE NEW ONE%%%%%%%%%%%%%%%%%%
In this paper we will propose such a theory using a new concept: {\em strong Lyapunov condition}. Recall that, in order to study the stability of an equilibrium of a dynamical system, e.g. \cref{eq:introode}, Lyapunov introduced the so-called Lyapunov function $\mathcal L(\bm x)$ \citep[see][]{Khalil2001}, which is nonnegative and satisfies $\mathcal L(\bm x^*) = 0$ and the Lyapunov condition: 
%\mnote{positive definite? Or just positive?}
\begin{equation}\label{eq:introLy-cond}
	-\nabla \mathcal L(\bm x) \cdot\mathcal G(\bm x)
	\text{ is locally positive near the equilibrium point }\bm x^*.
\end{equation}
%Here and in the following, when no confusion arises, we always use $\nabla \mathcal L(x) \cdot\mathcal G(x)$ to denote the inner product $\inner{\nabla \mathcal L(x), \mathcal G(x)}_*$ as $\nabla \mathcal L(x)$ can be identified as an element of $V$ by Riesz representation.
%\LH{Use dual pairing ?}
%
%The Lyapunov condition \cref{eq:introLy-cond} implies $\mathcal G$ is a descent direction for minimizing $\mathcal L(x)$. 
Then the (local) decay property of $\mathcal L(\bm x(t))$ along the trajectory $\bm x(t)$ of the autonomous system \cref{eq:introode} can be derived immediately
$$
\frac{\dd}{\dd t}  \mathcal L(\bm x(t))
= \nabla \mathcal L(\bm x) \cdot \bm x'(t) = \nabla \mathcal L(\bm x)\cdot \mathcal G(\bm x) < 0. 
$$
Therefore $\mathcal L(\bm x(t)) \to 0$ as $t \to \infty$ from which we may conclude $x(t) \to x^*$ or $f(x(t))\to f(x^*)$ as $t\to \infty$. However, this can only imply the convergence not the rate of convergence, i.e., how fast $\mathcal L(\bm x(t))$ approaches to zero. 
%
%When $f$ is coercive, the later one yields the weak convergence of the trajectory $x(t)$.
%With certain assumptions, we shall show starting from an arbitary initial guess $x_0$, the trajectory $x(t)$ of \cref{eq:introode} will converge to $x^*$ as $t\to \infty$

To establish the convergence rate of $\mathcal L(\bm x(t))$, we introduce the following {\it strong Lyapunov condition}: $\mathcal L(\bm x)$ is a Lyapunov function and there exist constant $q\geqslant 1$, strictly positive function $c(\bm x)$ and function $p(\bm x)$ such that
\begin{equation}\label{eq:introLyp-cond}
	-\nabla \mathcal L(\bm x)\cdot \mathcal G(\bm x)\geqslant
	c(\bm x)\mathcal L^q(\bm x)+p^2(\bm x)
	%	  \quad \forall\,\bm x\in \bm W.
\end{equation}
holds true near $\bm x^*$.
From this, we can easily derive the exponential decay $\mathcal L(\bm x(t)) =O(e^{-ct})$ for $q=1$ and the algebraic decay $\mathcal L(\bm x(t)) =O(t^{1/(1-q)})$ for $q>1$. We emphasize that the condition \cref{eq:introLyp-cond} is not only restricted to the strongly convex case. It can be established for convex case; see \cref{eq:diff-L-gf} for the gradient flow and \cref{eq:LG-AVD} for the AVD system.

We apply our framework to design and analyze first-order optimization methods, especially the accelerated gradient methods, for smooth and non-smooth convex optimization problems. We believe our unified convergence analysis is more transparent and systematic. 
Specifically, once the dynamical system \cref{eq:introode} is discretized in time by \begin{equation}\label{eq:bm-xk1}
	\bm x_{k+1}-\bm x_k = \alpha_k\widetilde{\mathcal G}(\bm x_k,\bm x_{k+1}),
\end{equation}
where $\widetilde{\mathcal G}(\bm x_k,\bm x_{k+1})$ is an approximation of ${\mathcal G}(\bm x_{k+1})$, a sequence of points $\{\bm x_k\}$ is produced. Given some strong Lyapunov function $\mathcal L(\bm x)$ that possess fast decay in the continuous level, a discrete Lyapunov function $\mathcal L_k = \mathcal L(\bm x_k)$ appear naturally. 
%We are interested in the convergence and convergence rate of $\mathcal L_k \to 0$ as $k\to \infty$, which may yield $x_k \to x^*$ or $f(\bm x_k)\to f(\bm x^*)$ as $k\to \infty$.
%
%Although the strong Lyapunov condition will imply the fast decay of $\mathcal L(\bm x(t))$, i
Due to the discretization error, the discrete dynamic system  \cref{eq:bm-xk1} may not be faithful to the continuous one \cref{eq:introode}. Whence, it is nontrivial to say that the scheme \cref{eq:bm-xk1} preserves the decay property in the discrete level. Fortunately, the strong Lyapunov condition \cref{eq:introLyp-cond} works for $\mathcal L_k$ and we will use it to guide the designing and analysis of optimization algorithms. A paradigm of our analysis is summarized in the following three steps.
% \LH{Comment: We may lose the designing of $\mathcal L$. For most existing dynamical systems, we have the corresponding $\mathcal L$.}
\begin{itemize}
	\item First expand the difference
	$$
	\mathcal L(\bm x_{k+1}) - \mathcal L(\bm x_k) = (\nabla \mathcal L(\bm x_{k+1}),\bm x_{k+1} - \bm x_k) - \mathcal R_1,
	$$
	where $\mathcal R_1\geqslant 0$ is the Bregman divergence of $\mathcal L$. The negative remainder $- \mathcal R_1$ is introduced due to the convexity of $\mathcal L$ which can be built-in when designing $\mathcal L$. 
	
%For a vector $\bm x$, e.g. $\bm x = (x, v, \gamma)$, we shall apply the above formulae for each component and try to evaluate the gradient at $k+1$. 
%But usually \dual{\nabla \mathcal L(x_k), x_{k+1} - x_k} + D_{\mathcal L}(x_{k+1}, x_k) 

	\item Then compare the right hand side of the discretization \cref{eq:bm-xk1} 
	%$\bm x_{k+1} - \bm x_k$ 
	with $\alpha_k \mathcal G(\bm x_{k+1})$ and obtain
	\begin{equation}\label{eq:step2}
		\mathcal L(\bm x_{k+1}) - \mathcal L(\bm x_k) \leqslant \alpha_k(\nabla \mathcal L(\bm x_{k+1}), \mathcal G(\bm x_{k+1})) - \mathcal R_1 + \mathcal R_2,
	\end{equation}
	where the positive term $\mathcal R_2$ comes from the lagging of discretization, i.e., $\widetilde{\mathcal G}(\bm x_k,\bm x_{k+1})-\mathcal G(\bm x_{k+1})$, which is generally nonzero for using partial information from $\bm x_{k}$.
	\item Finally applying strong Lyapunov property \cref{eq:introLyp-cond} at $\bm x_{k+1}$ to \cref{eq:step2} will bring more negative term $-p^2(\bm x_{k+1})$, which together with $-\mathcal R_1$ cancels the lagging effect $\mathcal R_2$ and thus implies
	$$
	\mathcal L_{k+1}-\mathcal L_k\leqslant -\alpha_k \mathcal L_{k+1}^q,
	$$
	from which linear or sub-linear decay rate of the sequence $\{\mathcal L_k\}$ can be derived. 
\end{itemize}

Here, we mention a most related work~\citet{Wilson:2021}. They derived dynamical models from the Bregman--Lagrangian and showed an equivalence between the technique of estimate sequences devised by~\citet{Nesterov:2013Introductory} and a family of Lyapunov functions in both continuous and discrete time. Note that their attentions were only paid to smooth objectives and they treated convex case and strongly convex case separately. In this work, however, we handle both convex and strongly convex cases simultaneously by introducing a time scaling factor and unify the verification of the contraction of Lyapunov function via the tool of strong Lyapunov condition which is also generalizable to non-smooth cases. 

The rest of this paper is outlined as follows. \cref{sec:bd-convex} is responsible for a brief review of  preliminary inequalities involving convex functions, and \cref{sec:strong-Lyapunov-functoin} introduces the strong Lyapunov condition and also provides some key estimates. As a revisit of the gradient descent method and the proximal point algorithm, \cref{sec:GD-Euler,sec:Scale-GD-Prox} shall apply our Lyapunov framework to the gradient flow and the scaled gradient flow, respectively. After that, \cref{sec:HB-Momentum,sec:NAG} focus on some typical second-order dynamical systems and give the corresponding convergence rate analysis via strong Lyapunov functions. Finally, \cref{sec:conclude} ends this paper with some concluding remarks.

\section{Bounds on Convex Functions}
\label{sec:bd-convex}
This section gives a quick review of several bounds on convex functions. Throughout, we consider both smooth convex functions over the entire space $V$ and extended-value function $f:V\to{\mathbb R}\cup\{+\infty\}$.  For the latter, the effective domain of $f$ is denoted by $\dom f:=\{x\in V:f(x)<\infty\}$.
%and list the related important inequalities. But for the sake of a little bit more generality, we start from some preliminary definitions for extended-value functions $f:V\to{\mathbb R}\cup\{+\infty\}$, 
%
%Since we consider unconstrained optimization problem, in most places, we assume that the domain is the whole space $V$, i.e., we consider $f:V\to \mathbb R$.
\subsection{Convex functions}
%We begin with definitions of convex functions. 
A continuous function $f$ is called {\em convex} if
\begin{equation}\label{eq:conv}
	f(\alpha x + (1-\alpha) y) \leqslant \alpha f(x) + (1-\alpha) f(y)\quad\forall\,\,x, y\in \dom f,
\end{equation}
for all $\alpha\in [0,1]$, and it is called {\it strictly convex} if the above inequality holds strictly
$$f(\alpha x + (1-\alpha) y) < \alpha f(x) + (1-\alpha) f(y)\quad  \forall\,x, y\in \dom f \text{ and } x\neq y,$$
for all $\alpha\in (0,1)$.
A convex function is called {\it $\mu$-strongly convex} with parameter $\mu > 0$ if 
\begin{equation}\label{eq:def:strong-conv}
	f(\alpha x + (1-\alpha) y) \leqslant \alpha f(x) + (1-\alpha) f(y) -\frac{\mu}{2}\alpha(1-\alpha)\| x- y\|^2\quad\forall\,\,x, y\in \dom f,
\end{equation}
for all $ \alpha\in [0,1]$. 
% A direct computation shows that \cref{eq:def:strong-conv} is equivalent to  $f(x)-\frac{\mu}{2}\nm{x}^2$ is convex. 

The function $f$ is {\em coercive} if $f(x)\to\infty$ when $\nm{x}\to\infty$. When $f$ is $\mu$-strongly convex with $\mu>0$, then it is not hard to see $f$ is coercive. But convexity itself cannot imply the coercivity, e.g. $f(x) = e^{-x}$. The following results are classical, and proofs, which are skipped for the sake of brevity, can be found in~\citet[Proposition 1.2]{ekeland_convex_1987} or~\citet[Theorem.~8.2.2]{Ciarlet89}.

\begin{thm}
	\label{thm:wellposedS3}
	If $f$ is convex and coercive, then the problem \cref{eq:min} admits at least one solution $x^*\in V$, which is unique if we assume further $f$ is strictly convex.
\end{thm}

\subsection{Convex function classes}
Let $\mathcal C^1$ consist of all continuous differentiable functions on $V$. Denote by $\mathcal C_L^{1,1}$ the set of all $\mathcal C^1$ functions, the gradient of which is Lipschitz continuous with constant $0<L<\infty$:
%\mnote{Requiring $L<\infty$ agrees with conventional definition.}
\begin{equation}\label{eq:Lip}
	\|\nabla f(x) - \nabla f(y)\|_{*} \leqslant L\| x - y \|\quad \forall\,x, y\in V,
\end{equation}
where, for $g\in V^*$, the dual norm is
\[
\| g \|_{*} := \sup_{\substack{v \in V \\ \|v\|=1}} \langle g, v \rangle  = \sup_{v\in V\setminus\{0\}} \frac{\langle g,v\rangle}{\nm{v}}. 
\] 

We now introduce several function classes of convex functions. 
For $\mu > 0$, we use
$\mathcal S^0_\mu$ to denote the set of all $\mu$-strongly convex functions, and $\mathcal S^0_0$ for convex functions, where the superscript $0$ indicates the function is only continuous and may not be differentiable. Also, any $f\in\mathcal S_\mu^0$ is assumed to be closed and proper ($\dom f\neq \emptyset$). Moreover, for all $\mu\geqslant 0$ we set $\mathcal S_\mu^1 := \mathcal S_\mu^0\cap\mathcal C^1$. 
%For $\mu > 0$, we use
%$\mathcal S^1_\mu$ to denote the set of all $\mu$-strongly convex functions in $\mathcal C^1$, and $\mathcal S^1_0$ for convex functions in $\mathcal C^1$. 
%
For constants $0\leqslant \mu \leqslant L < \infty$, we introduce the function class
$$
\mathcal S^{1}_{\mu, L} := \{ f\in \mathcal S_0^1 : \mu  \| x - y \|^2 \leqslant \langle \nabla f(x) - \nabla f(y), x - y\rangle \leqslant L\| x - y \|^2\, \forall\,x, y\in V\}.
$$
Set $\mathcal S_{\mu,L}^{1,1} = \mathcal S_\mu^1\cap\mathcal C_L^{1,1}$. It can be shown that $\mathcal S_{\mu,L}^{1,1} = \mathcal S^{1}_{\mu, L}$; see~\citet{lessard_analysis_2016}. 

%The inclusion $\mathcal S^{1,1}_{\mu, L}  \subseteq \mathcal S^{1}_{\mu, L} $ is trivial by the definition of the dual norm. 
%\mnote{But we haven't proved it here.}
%We choose to work with $\mathcal S^{1}_{\mu, L}$ since verification for functions in this class is much easier. 

%\LH{Comment: Since $\mathcal S_\mu^1 := \mathcal S_\mu^0\cap\mathcal C^1$, for $f\in \mathcal S_{\mu,L}^{1,1}$, it is not much evident to have
%	\[
%	\langle \nabla f(x) - \nabla f(y), x - y\rangle\geqslant 
%		\mu  \| x - y \|^2 
%	\]
%which can be obtained from \cref{eq:Dmu}. That is, the inclusion $\mathcal S^{1,1}_{\mu, L}  \subseteq \mathcal S^{1}_{\mu, L} $ is nontrivial.}

\subsection{Bregman divergence and various bounds}
For $f\in\mathcal C^1$, define 
\begin{equation}\label{eq:Df}
	D_f(y,x) := f(y) - f_l(y; x) = f(y) - f(x) - \langle \nabla f(x), y - x \rangle,
\end{equation}
where $f_l(y; x):= f(x) + \langle \nabla f(x), y - x \rangle$  is its linear Taylor expansion at $x$.
%Geometrically $f_l(y; x)$ defines a supporting hyper plane \mnote{ should be $f_l(y; x) - x_{n+1} = 0$ }at point $x$ for the graph of $f$ when $f$ is convex. 
If $f$ is convex, then for fixed $x\in V$, $D_f(\cdot,x)$ is also convex and thus $D_f(y,x)\geqslant 0$. When $f$ is strictly convex, $D_f(y,x)=0$ iff $x=y$, and $D_f(y,x)$ is called the {\it Bregman divergence} associated with $f$, which is in general not symmetric.
%, i.e., $D_f(y,x)\neq D_f(x,y),$ when $x\neq y.$

We then introduce its symmetrization, the symmetrized Bregman divergence,  
\begin{equation}\label{eq:M}
	2 M_{\nabla f}(x,y) :=  
	D_f(y,x)+D_f(x,y)=
	\langle \nabla f(x) - \nabla f(y), x - y\rangle.
\end{equation}
By the fundamental theorem of calculus 
\begin{equation}\label{eq:Df-int}
	\begin{aligned}
		D_f(y,x) &=  \dual{ \int_0^1  \nabla f(x+ \xi(y-x)) - \nabla f(x) \dd  \xi, \,y - x }  \\
		&=\int_0^1 2M_{\nabla f}(x_\xi, x)\frac{\dd \xi}{\xi},
		\quad x_{\xi}: = x+ \xi(y-x).
	\end{aligned}
\end{equation}
Based on \eqref{eq:Df-int}, we can shift the bound for $D_f(y,x)$ to $M_{\nabla f}(x,y)$ and vice verse. The following bounds can be found in~\citet[Chapter 2]{Nesterov:2013Introductory}. 
%where $x_{\xi}: = x+ \xi(y-x)$. 
%
%
%and rewrite the identity \cref{eq:Df-int} as
%\begin{equation}\label{eq:M2D}
%D_f(y,x) = \int_0^1 2M_{\nabla f}(x_\xi, x)\frac{\dd \xi}{\xi},
%\quad x_{\xi}: = x+ \xi(y-x).
%\end{equation}

\begin{lem}\label{lem:Bregmandiv}
	For $f\in\mathcal C_{L}^{1,1}$, we have the upper bound
	\begin{align}
		\label{eq:DL} \max\{D_f(y,x), M_{\nabla f}(x,y)\}&\leqslant  \frac{L}{2}\| x - y \|^2.
		% \quad \forall\,x, y\in V.
	\end{align}
	For $f\in\mathcal S_{\mu}^1$ with $\mu\geqslant 0$, we have the lower bound
	\begin{align}
		\label{eq:Dmu}  \min\{D_f(y,x), M_{\nabla f}(x,y)\} &\geqslant \frac{\mu}{2}\| x - y \|^2.
		% \quad \forall\,x, y\in V.
	\end{align}
	For $f\in\mathcal S^{1}_{0,L}$, we have the lower bound
	\begin{align}
		\label{eq:philowerL} \min \{D_f(y,x), M_{\nabla f}(x,y)\} {}&\geqslant\frac{1}{2L}\| \nabla f(y) - \nabla f(x)\|_*^2.
		%  \qquad \forall\,x,y \in V.
	\end{align}
	For $f\in \mathcal S^{1}_{\mu}$ with $\mu>0$, we have the upper bound
	\begin{align}
		\label{phiuppermu} 
		\max\{D_f(y,x), M_{\nabla f}(x,y)\} &\leqslant \frac{1}{2\mu}\| \nabla f(y) - \nabla f(x)\|_*^2.
		%  \qquad \forall\,x,y \in V.
	\end{align}
	All the above inequalities hold for all $x,y\in \dom f$.
\end{lem}

\subsection{Bounds involving a global minimum}
We list specific examples of inequalities when one variable is $x^*$ for which $\nabla f(x^*) = 0$. Then $D_f(x, x^*) = f(x) - f(x^*)$ is the so-called optimality gap and $2M_{\nabla f}(x,x^*) = \dual{\nabla f(x), x - x^*}$.

\begin{coro} 
	For $f\in \mathcal S^1_{0,L}$, we have
	\begin{align}
		\label{eq:Df2df} \frac{1}{2L}\| \nabla f(x)\|_*^2 \leqslant f(x) - f(x^*) & \leqslant \frac{L}{2}\| x - x^*\|^2,\\ 
		\label{eq:Mf2df} \frac{1}{L}\| \nabla f(x)\|_*^2 \leqslant \dual{\nabla f(x), x - x^*} & \leqslant L \| x -x^*\|^2 .
	\end{align}
	For $f\in \mathcal S^1_{\mu}$ with $\mu>0$, we have
	\begin{align}
		\label{eq:optgapmu}
		\frac{\mu}{2}\| x -x^*\|^2 & \leqslant f(x) - f(x^*)  \leqslant \frac{1}{2\mu}\|\nabla f(x)\|_*^2, \\ 
		\label{eq:Mfmu} \mu\| x - x^*\|^2 & \leqslant  \dual{\nabla f(x), x - x^*}  \leqslant \frac{1}{\mu} \| \nabla f(x)\|_*^2, \\
		\label{eq:Mxstar}\dual{\nabla f(x), x - x^*} &\geqslant f(x) -f(x^*) + \frac{\mu}{2}\| x - x^* \|^2.
		%\label{eq:optgapnormmu} f(x) - f(x^*) + \frac{\mu}{2}\| x - x^*\|^2 & \leqslant \dual{\nabla f(x), x - x^*}.
	\end{align}
	For $f\in \mathcal S^{1}_{\mu, L}$ with $\mu\geqslant 0$, we have
	\begin{equation}\label{eq:refineMxstar}
		\dual{\nabla f(x), x - x^*} \geqslant \frac{\mu L}{\mu + L} \| x-x^*\|^2 + \frac{1}{\mu + L}\| \nabla f(x) \|_*^2.
	\end{equation}
	All the above inequalities hold for all $x,y\in \dom f$.
\end{coro}
Inequalities \eqref{eq:Df2df}-\eqref{eq:Mfmu} are direct consequence of Lemma \ref{lem:Bregmandiv} and  \eqref{eq:Mxstar} is the definition of $\mu$-convex. The refined lower bound \eqref{eq:refineMxstar}  of $M_{\nabla f}$ can be found 
in~\citet[Theorem 2.1.12]{Nesterov:2013Introductory}. 

To the end, we extend \cref{eq:optgapmu,eq:Mxstar} to the nonsmooth case. Recall that the sub-gradient $\partial f$ of a proper and convex function $f$ is a set-valued function and can be defined as follows
\begin{equation}\label{eq:sub}
\partial f(x):=\left\{
p\in V^*:\,f(y)-f(x)\geqslant \dual{p,y-x}\quad\forall\,y\in V
\right\}.
\end{equation} 
Any $p\in\partial f(x)$ will be also called a {\it sub-gradient} of $f$ at $x$. 
\begin{coro} \label{coro:bd-non}
	For $f\in \mathcal S^0_{\mu}$ with $\mu>0$, we have
	\begin{align*}
%		\label{eq:optgapmu-non}
		\frac{\mu}{2}\| x -x^*\|^2 & \leqslant f(x) - f(x^*)  \leqslant \frac{1}{2\mu}\|p\|_*^2, \\ 
%		\label{eq:Mfmu-non}
		 \mu\| x - x^*\|^2 & \leqslant  \dual{p, x - x^*}  \leqslant \frac{1}{\mu} \| p\|_*^2, \\
%		\label{eq:Mxstar-non}
		\dual{p, x - x^*} &\geqslant f(x) -f(x^*) + \frac{\mu}{2}\| x - x^* \|^2,
	\end{align*}
	where $p\in\partial f(x)$ and $x\in \dom f$.
\end{coro}
\section{Strong Lyapunov Functions}
\label{sec:strong-Lyapunov-functoin}
%\LH{Comment: Our optimization problem \cref{eq:min} takes place over $V$. But we also treat \cref{eq:ode-G} on  $V$. This is incompatible with the vector case $\bm x$.}
In this section we consider the autonomous dynamical system
\begin{equation}\label{eq:ode-G}
	x'(t) = \mathcal G(x(t)),\quad t>0,
\end{equation}
where $x: \mathbb R_+ \to \mathcal V$ and $\mathcal G: \mathcal V\to \mathcal V^*$ is a vector field. Here the Hilbert space $\mathcal V$ may not be the space $V$ for the original optimization \cref{eq:min}. We mainly consider smooth $\mathcal G$, with which the well-posedness of \cref{eq:ode-G} is usually evident by standard ODE theory, under mild condition on $\mathcal G$ (Lipschitz continuity).  Let $x^*$ be an equilibrium point of the dynamic system \cref{eq:ode-G}, i.e. $\mathcal G(x^*) = 0$. We are interested in the convergence of the trajectory $x(t)$ to $x^*$ as $t\to \infty$. 
%, in strong sense $\nm{x(t)-x^*}\to 0$.
% or weak sense $f(x(t))-f(x^*)\to0$.

\subsection{Strong Lyapunov condition and decay property}
Originally the Lyapunov function is constructed to study the stability of an equilibrium point. 
%We call $\mathcal L: \mathcal V\to\mathbb R_+$ a locally Lyapunov 
%function of the flow field $\mathcal G$ if $\mathcal L(x^*) = 0$ and it satisfies the locally Lyapunov condition: 
%\begin{equation}\label{eq:Ly-cond}
%	-\nabla \mathcal L(x) \cdot\mathcal G(x)
%	\text{ is locally positive near }x^*.
%\end{equation}
%Then along the trajectory $x(t)$ of  \cref{eq:ode-G}, when $x(t)$ is near $x^*$,
%$$
%\frac{\dd}{\dd t}  \mathcal L(x(t))
%= \nabla \mathcal L(x) \cdot x'(t) = \nabla \mathcal L(x)\cdot \mathcal G(x) < 0. 
%$$
%Therefore $\mathcal L(x(t)) \to 0$ as $t \to \infty$.
To obtain the convergence rate, we need a stronger condition than merely $-\nabla \mathcal L(x)\cdot \mathcal G(x)$ is locally positive. 
If there exist a compact subset $\mathcal W\subseteq \mathcal V$, a positive function $c(x)> 0, \forall x\in \mathcal W$, a constant $q\geqslant 1$, and a function $p(x):\mathcal V\to\mathbb R$ such that $\mathcal L(x)\geqslant 0$ 
\begin{equation}\label{eq:A}
	-\nabla \mathcal L(x)\cdot \mathcal G(x)\geqslant
	c(x)\mathcal L^q(x)+p^2(x)\quad \forall\,x\in \mathcal W.
\end{equation}
%When $W=V$, we call $\mathcal L(x)$ a global Lyapunov function. 
%If $c(x)>0$, for all $x\in W$, 
then we call $\mathcal L$ a locally ($\mathcal W\subset \mathcal V$) or globally ($\mathcal W=\mathcal V$) strong Lyapunov function. We use $\mathcal A(c,q,p,\mathcal W)$ to denote the strong Lyapunov condition \cref{eq:A} and simplify it as $\mathcal A(c,q,p)$ when $\mathcal W = \mathcal V$. 
%When $c = 0$
%
%\LH{In what follows, when no confusion arises, we shall abbreviate $\mathcal L(x(t))$ as $\mathcal L(t)$.} \mnote{ Then $\nabla \mathcal L(t)$ is misleading.}
\begin{thm}\label{thm:strongLya}
	Assume that $\mathcal L(x)$ satisfies $\mathcal A(c,q,p,\mathcal W)$. If the trajectory $x(t)$ to \cref{eq:ode-G} satisfies that $\{x(t):t\geqslant t_0\}\subset \mathcal W$ for some $t_0\geqslant 0$, then for all $t\geqslant t_0$,
	\begin{equation}\label{eq:ineq-L-q}
		\mathcal L(x(t))\leqslant 
		\left\{
		\begin{aligned}
			&\; \mathcal L(x(t_0))\exp \left (-\int_{t_0}^{t}c(x(s))\dd s \right )&&\text{ if } q=1,\\
			&\;	\left((q-1)\int_{t_0}^{t}c(x(s))\dd s+\mathcal L(x(t_0))^{1-q}\right)^{1/(1-q)}&&\text{ if } q>1.
		\end{aligned}
		\right.
	\end{equation}
\end{thm}
\begin{proof}
%	We first prove \cref{eq:ineq-L-q}. 
	By the assumption $\mathcal A(c,q,p,W)$, for all $t\geqslant t_0$,
	\begin{equation}\label{eq:diff-L-q}
		\begin{split}
			\frac{\dd }{\dd t}\mathcal L(x(t)) 
			={}&\nabla \mathcal L(x(t))\cdot x'(t)
			=\nabla \mathcal L(x(t))\cdot \mathcal G(x(t))\\
			\leqslant{}& -c(x(t))\mathcal L^q(x(t))-p^2(x(t))\\
			\leqslant{}& -c(x(t))\mathcal L^q(x(t)).
		\end{split}
	\end{equation}
	The case $q=1$ is trivial from \cref{eq:diff-L-q}. Assume $q>1$. Then we have
	\[
	\frac{\mathrm d }{\mathrm dt}\mathcal L^{1-q}
	= (1-q)\frac{\mathcal L'}{\mathcal L^q}
	\geqslant c(x(t))(q-1),
	\]
	and it follows that
	\[
	\mathcal L^{1-q}-\mathcal L(0)^{1-q}
	\geqslant (q-1)\int_{t_0}^{t}c(x(s))\dd s, \quad t\geqslant t_0,
	\]
	which proves \cref{eq:ineq-L-q}.
\end{proof}

\subsection{Generalization to non-smooth convex optimization}
\label{sec:inclu}
Generally, the field $\mathcal G$ can be a set-value mapping, which may occur when $f$ is convex but non-smooth, which yields the differential inclusion
\begin{equation}\label{eq:di}
	x'(t) \in \mathcal G(x(t)),\quad t>0.
\end{equation}
To emphasize the dependence of sub-gradient $\partial f$, we modify the notation $\mathcal G(x)$ to $\mathcal G( x,\partial f( x))$ and use $\mathcal G( x, d(x))$ for one particular direction $d\in \partial f(x)$. Then \cref{eq:di} can be also written as $x' = \mathcal G( x, d(x))$.

Similarly a Lyapunov function $\mathcal L(x)$ may not be smooth and $\partial \mathcal L(x, \partial f)$ is used to emphasize the dependence of sub-gradient of $f$. For one particular direction $d\in \partial f(x)$,  $\partial \mathcal L(x, d)$ is a single-valued vector function. 

The strong Lyapunov condition can be generalized to the non-smooth case as follows. We call $\mathcal L: \mathcal V\to\mathbb R^+$ a locally Lyapunov 
function of the flow field $\mathcal G$ if $\mathcal L(x^*) = 0$ and there exist a nonnegative function $c(x)\geqslant 0$, a constant $q\geqslant 1$, a compact subset $\mathcal W\subset \mathcal V$, a function 
$p(x): \mathcal V\to\mathbb R$, and $d(x)\in \partial f(x)$ such that $\mathcal L(x)\geqslant 0$ for all $x\in \mathcal W$ and 
\begin{equation}\label{eq:strLnonsmooth}
	-\partial \mathcal L(x, d)\cdot \mathcal G(x, d)\geqslant
	c(x)\mathcal L^q(x)+p^2(x),\quad \forall\,x\in \mathcal W.
\end{equation}
If $c(x)>0$, for all $x\in \mathcal W$, then we call $\mathcal L$ locally ($\mathcal W\subset \mathcal V$) or globally ($\mathcal W=\mathcal V$) strong Lyapunov function. We still use $\mathcal A(c,q,p,\mathcal W)$ to denote the strong Lyapunov condition \cref{eq:A} and use $\mathcal A(c,q,p)$ when $\mathcal W = \mathcal V$. 

Note that when verifying the strong Lyapunov property \eqref{eq:strLnonsmooth}, for non-smooth functions, we only need to find one sub-gradient in $\partial f$.

\subsection{Difference of Lyapunov functions}
We then move to the discrete case. The following identities are obvious by the definition of $D_{\mathcal L}(\cdot,\cdot)$. When $\mathcal L$ is convex, various bounds on $D_{\mathcal L}(\cdot,\cdot)$ can be used to bound the difference of Lyapunov functions. 

\begin{lem}
	Assume $\mathcal L$ is differentiable. Then for any two points $x_k, x_{k+1}\in \mathcal V$
	\begin{equation}\label{eq:differenceLya}
		\mathcal L(x_{k+1}) - \mathcal L(x_k) = 
		\begin{cases}
			\dual{\nabla \mathcal L(x_k), x_{k+1} - x_k} + D_{\mathcal L}(x_{k+1}, x_k),\\
			\dual{\nabla \mathcal L(x_{k+1}), x_{k+1} - x_k} - D_{\mathcal L}(x_k, x_{k+1}).		
		\end{cases}
	\end{equation}
\end{lem}

The two points $x_k$ and $x_{k+1}$ will be connected by some numerical discretization of \cref{eq:ode-G}. For example, for the implicit Euler method, $x_{k+1} - x_k = \alpha \mathcal G(x_{k+1})$. Then the strong Lyapunov property can be applied to $\dual{\nabla \mathcal L(x_{k+1}),  \mathcal G(x_{k+1})}$ which will bring more negative terms on the upper bound of $\mathcal L(x_{k+1}) - \mathcal L(x_k)$ and convergence can be further derived. 

On the other hand, if we use the explicit Euler method $x_{k+1} - x_k = \alpha \mathcal G(x_k)$, the vector field is evaluated at the current point $x_k$, there will be a positive term $D_{\mathcal L}(x_{k+1}, x_k) \approx \alpha^2 \| x_{k+1} - x_k\|^2$ on the upper bound. We then use the strong Lyapunov function at $x_k$ to bring negative terms which scales like $\mathcal O(\alpha)$. Then choosing step size $\alpha$ small enough, we can cancel the positive $\mathcal O(\alpha^2)$ term.

By \cref{coro:bd-non}, for $f\in\mathcal S_\mu^0$, we can use the definition of the convexity to control the difference: for any  $d_{k+1}\in \partial f(x_{k+1})$
\begin{equation}\label{eq:nonsmoothconvex}
	f(x_{k+1}) - f(x_k) \leqslant \langle d_{k+1}, x_{k+1} - x_k \rangle - \frac{\mu}{2}\| x_{k+1} - x_k \|^2.
\end{equation}
Besides the gradient at two end points $\{x_k, x_{k+1}\}$, we may also use another intermediate point.
\begin{lem}\label{lm:3pts}
	For $f\in S_{\mu, L}^1$ and arbitrary three points $\{x_k, y, x_{k+1}\}$, we have
	\begin{align*}
		f(x_{k+1}) - f(x_k) \leqslant {}& \langle \nabla f(y), x_{k+1} - x_k \rangle + \frac{L}{2}\| x_{k+1} - y\|^2 \\
		&\quad- \max \left\{ \frac{\mu}{2}\| y-x_k\|^2, \frac{1}{2L}\| \nabla f(y) - \nabla f(x_k)\|_*^2 \right \}.
	\end{align*}
\end{lem}
\begin{proof}
	We split the difference $f(x_{k+1}) - f(x_k) = f(x_{k+1}) - f(y) + f(y) - f(x_k)$. For the first term, we use \eqref{eq:DL}
	$$
	f(x_{k+1}) - f(y) \leqslant \langle \nabla f(y), x_{k+1} - y \rangle + \frac{L}{2}\| x_{k+1} - y\|^2
	$$
	and for the second term, we use either \eqref{eq:Dmu} or \eqref{eq:philowerL}
	$$
	f(y) - f(x_k) \leqslant \langle \nabla f(y), y - x_k \rangle -  \max \left \{ \frac{\mu}{2}\| y-x_k\|^2, \frac{1}{2L}\| \nabla f(y) - \nabla f(x_k)\|_*^2 \right \}.
	$$
	Summing these two inequalities completes the proof of this lemma.
\end{proof}

\subsection{Decay rate of discrete Lyapunov functions}
Start from an initial guess $x_0$, for $k=0,1,\ldots,$ a generic one step method for \cref{eq:ode-G} can be written as $x_{k+1} = E(\alpha_k, x_k)$, where $\alpha_k$ is the time step size. A discrete Lyapunov sequence is naturally defined as $\{\mathcal L_k = \mathcal L(x_k), k = 0, 1, 2, \ldots \}$. 

Although the strong Lyapunov property ensures the decay of $\mathcal L(x(t))$ in the continuous level, it is nontrivial to design a numerical scheme that preserves the decay property in the discrete level, i.e., the decay of the sequence $\{\mathcal L_k \}$. 

To establish the convergence rate, the key is to have a discrete version of \cref{thm:strongLya} which will yield the convergence (or boundness) of $\{\mathcal L_k \}$, the discrete analogue of $\mathcal L(x(t))$.
%
%As summarized previously, we have to treat the two cases: \cref{eq:ode-ineq-q>1-dis-,eq:ode-ineq-q>1-dis}. Therefore, to prove the decay rate, the key is to establish a discrete version of \cref{thm:strongLya} which will yield the (local) convergence or boundedness of $\{\mathcal L_k \}$, the discrete analogue of $\mathcal L(x(t))$.
%
We present the following decay rate of positive sequences satisfying certain inequalities. 
%For the general case, we refer to~\citet{chen_luo_dynamic_2020}.

\begin{thm}\label{thm:ode-ineq-q>1-dis}
	Let $\{A_k:k\geqslant 0\}$ be a positive sequence. 
	\begin{enumerate}
		\item If
		%	\begin{equation}\label{eq:ode-ineq-q>1-dis-}
		$$			A_{k+1}-A_k\leqslant -\alpha_k A_{k}-p_k^2,\quad k\geqslant 0,$$
		%		\end{equation}
		holds for some nonnegative sequence $\{\alpha_k:k\geqslant 0\}\subset[0,1)$,
		then
		%		\begin{equation}\label{eq:ineq-q>1-dis-}
		\begin{equation}\label{eq:est-Ak-1}
			A_k\leqslant \rho_k A_{0} \text{ and }\quad \sum_{i=0}^{\infty} \frac{p_i^2}{\rho_{i}}\leqslant CA_{0},
		\end{equation}
		%		\end{equation}
		where $$\rho_0=1,\quad\rho_k = \prod_{i=0}^{k-1}(1-\alpha_i)\in (0,1],\quad k\geqslant 1.$$
		
		\item If
		%	\begin{equation}\label{eq:ode-ineq-q>1-dis}
		$$			A_{k+1}-A_k\leqslant -\alpha_k A_{k+1}-p_k^2,\quad k\geqslant 0,$$
		%		\end{equation}
		holds for some nonnegative sequence $\{\alpha_k:k\geqslant 0\}$, then \cref{eq:est-Ak-1} holds true with
		%%		\begin{equation}\label{eq:ineq-q>1-dis-}
		%$$A_k\leqslant \rho_k A_{0}, \quad \sum_{i=0}^{\infty} \frac{p_i^2}{\rho_{i}}\leqslant CA_{0},$$
		%%		\end{equation}
		%		where 
		$$\rho_0=1,\quad\rho_k = \prod_{i=0}^{k-1}\frac{1}{1+\alpha_i}\in (0,1],\quad k\geqslant 1.$$
		
		\item If 
		\begin{equation}\label{eq:ode-ineq-q>3-dis-}
			A_{k+1}-A_k\leqslant -\alpha A_{k}^2,\quad k\geqslant 0,
		\end{equation}
		holds for some $\alpha > 0$, then
		\begin{equation}\label{eq:ineq-q>3-dis-}
			A_k \leqslant \frac{A_0}{1+ \alpha A_0 k}.
		\end{equation}
		
		\item If 
		\begin{equation}\label{eq:ode-ineq-q>4-dis-}
			A_{k+1}-A_k\leqslant -\alpha A_{k+1}^2,\quad k\geqslant 0,
		\end{equation}
		holds for some $\alpha > 0$, then
		\begin{equation}\label{eq:ineq-q>4-dis-}
			A_k \leqslant  (1+\delta)\frac{A_0}{1+ \alpha A_0 k}, \text{ with } \delta = \frac{\alpha A_0}{1+ \alpha A_0}.
		\end{equation}
		%\LH{Comment: As $\delta<1$, just bound this above by $1+\delta\leqslant 2$?}
	\end{enumerate}
\end{thm}
\begin{proof}
	The cases (1) and (2) are straightforward and thus skipped. We now consider case (3). 
	Inequality \eqref{eq:ode-ineq-q>3-dis-} implies $A_{k+1}\leqslant A_k$.
	Consider the reciprocal of the sequence
	\begin{align*}
		\frac{1}{A_{k+1}} -\frac{1}{A_{k}} = \frac{A_k - A_{k+1}}{A_{k+1}A_k}\geqslant \alpha \frac{A_{k}}{A_{k+1}} \geqslant \alpha.
	\end{align*}
	Then sum to get the estimate \eqref{eq:ineq-q>3-dis-}.
	
	For a sequence satisfying \eqref{eq:ode-ineq-q>4-dis-}, we still have $A_{k+1}\leqslant A_k$ and 
	\begin{equation}\label{eq:Ak}
		\frac{1}{A_{k+1}} -\frac{1}{A_{k}} = \frac{A_k - A_{k+1}}{A_{k+1}A_k}\geqslant \alpha \frac{A_{k+1}}{A_{k}}.
	\end{equation}
	Obviously \eqref{eq:ineq-q>4-dis-} holds for $k=0$. Assume it holds for $k\geqslant 1$. If $A_{k+1}\geqslant A_k/(1+\delta)$, then \eqref{eq:Ak} implies 
	\begin{align*}
		\frac{1}{A_{k+1}} \geqslant \frac{1}{A_{k}} + \frac{\alpha}{1+\delta} \geqslant \frac{ 1+ \alpha A_0 (k + 1)}{A_0(1+\delta)}.
	\end{align*}
	Namely \eqref{eq:ineq-q>4-dis-} holds for $k+1$. 
	Otherwise $A_{k+1}< A_k/(1+\delta)$, then 
	\begin{align*}
		A_{k+1} < \frac{1}{1+\delta} A_k \leqslant \frac{A_0}{1+ \alpha A_0 k} \leqslant (1+\delta)\frac{A_0}{1+ \alpha A_0 (k+1)}.
	\end{align*}
	The last inequality can be easily verified by the definition of $\delta$. 
\end{proof}

%$$
%The bounds yields that 
%\[
%\min_{0\leqslant i\leqslant k}
%p_i^2\leqslant \frac{CA_0}{\sum_{i=0}^{k}
%	1/\rho_i }.
%\]
%Asymptotically, we have $p_k^2 = o\left(\rho_k/k\right) $.
%For detailed proofs, we refer to \cref{rem:fast-gd} or~\citep[Theorem 3.1]{ChenLuo2019HNAG}.

%\begin{rem}
The convergence rates $\rho_k$ given in \cref{thm:ode-ineq-q>1-dis} depend on the step size $\{\alpha_k:k\geqslant 0\}$. Within the allowed range of $\alpha_k$, the larger is the step size, the better is the decay rate. 
%\end{rem}

\subsection{Decay rate of parameters}
For most accelerated optimization methods, there is a sequence of parameter $\{\gamma_k\}$ defined by
\begin{equation}\label{eq:gamma}
\gamma_{k+1}-\gamma_{k} = \alpha_k(\mu - \gamma_{k+1}),
\end{equation}
which is an implicit Euler discretization of ODE $\gamma'= \mu - \gamma$. The step size $\alpha_k$ is determined by the parameters $L$ and/or $\mu$. Let
	\begin{equation}\label{eq:lambdak}
		\rho_0=1,\quad\rho_k = 
		\prod_{i=0}^{k-1}\frac{1}{1+\alpha_i},\quad k\geqslant 1.
	\end{equation}
%the step size, we can estimate the decay rate of $\rho_k$. 
	By definition \cref{eq:lambdak}, $\{\rho_k \}$ satisfies the relation
	$$
	\rho_{k+1} - \rho_k = - \alpha_k \rho_{k+1}, \quad \rho_{k+1} = \frac{1}{1+\alpha_k}\rho_k, 
	$$
	which implies $\rho_k$ is monotone decreasing. 
	The formula \eqref{eq:gamma} of $\gamma_k$ yields
	\[
	\frac{1}{1+\alpha_k} = \frac{\gamma_{k+1}}{\gamma_k+\mu \alpha_k}
	\leqslant \frac{\gamma_{k+1}}{\gamma_k},
	\]
	and it follows that
	\begin{equation}\label{eq:lowalpharho}
		\rho_k \leqslant 
		\frac{\gamma_k}{\gamma_0}.
%		=
%		\frac{L\alpha_k^2}{\gamma_0(1+B\alpha_k)}.
	\end{equation}
Namely the decay rate of $\gamma_k$ will give an upper bound of the convergence rate $\rho_k$. We will present estimates when $\alpha_k$ and $\gamma_k$ are related. 
We first present an identity on $\gamma_k$ in terms of the ratio $\alpha_k/\gamma_k$. 
%\begin{coro}\label{coro:decay}
%Assume $f$ is in $\mathcal S_{\mu}^{0}$ with $\mu\geqslant 0$. Let $\{ x_k \}$ be the sequence generated by Algorithm \ref{algo:Prox} with step size $\{ t_k \}$. Then for any $\gamma_0 > 0$
%\begin{equation}\label{eq:Lkdecay}
%\mathcal L_k \leqslant \frac{\mu \mathcal L_0}{\mu + \gamma_0\left [\prod_{i=0}^{k-1} (1+t_i\mu) - 1\right ]},
%\end{equation}
%where $\mathcal L_k = f(x_k)-f(x^*)+\frac{\gamma_k}{2}\nm{x_k-x^*}^2$ with 
%\begin{equation}\label{eq:gammak}
%\gamma_k = \frac{ \mu \gamma_0 \prod_{i=0}^{k-1} (1+t_i\mu)}{\mu + \gamma_0\left [\prod_{i=0}^{k-1} (1+t_i\mu) - 1\right ]}.
%\end{equation}
%In particular, for $\mu = 0$, 
%$$
%\mathcal L_k \leqslant \frac{\mathcal L_0}{1 + \gamma_0 \sum_{i=0}^{k-1} t_i}, \quad \gamma_k = \frac{\gamma_0}{1+ \gamma_0 \sum_{i=0}^{k-1} t_i}. 
%$$
%\end{coro}
\begin{lem}\label{lem:decay}
	Given $\mu\geqslant 0,\,\gamma_0>0$ and some positive real sequence $\{\alpha_k\}_{k=0}^\infty$, define $\{\gamma_k\}_{k=0}^\infty$ by \eqref{eq:gamma}. Then $\gamma_k>0$ and we have
	%\begin{equation}\label{eq:Lkdecay}
	%\mathcal L_k \leqslant \frac{\mu \mathcal L_0}{\mu + \gamma_0\left [\prod_{i=0}^{k-1} (1+t_i\mu) - 1\right ]},
	%\end{equation}
	%where $\mathcal L_k = f(x_k)-f(x^*)+\frac{\gamma_k}{2}\nm{x_k-x^*}^2$ with 
	\begin{equation}\label{eq:gammak}
		\gamma_k = \frac{ \gamma_0 \prod_{i=0}^{k-1} (1+t_i\mu)}{1 + \gamma_0\left [\prod_{i=0}^{k-1} (1+t_i\mu) - 1\right ]/\mu},
	\end{equation}
	where $t_k = \alpha_k/\gamma_k$ and for $\mu=0$ we made the convention
	\begin{equation}\label{eq:conven}
		\frac{1}{\mu}\left [\prod_{i=0}^{k-1} (1+t_i\mu) - 1\right ]:=\sum_{i=0}^{k-1} t_i,
	\end{equation}
	which is compatible with the right hand side as $\mu\to 0$.
\end{lem}
\begin{proof}
	Consider the difference
	$$
	\frac{1}{\gamma_{k+1}} -\frac{1}{\gamma_{k}} = \frac{\gamma_k - \gamma_{k+1}}{\gamma_{k+1}\gamma_k} = \frac{\alpha_k(\gamma_{k+1} - \mu)}{\gamma_{k+1}\gamma_k} = \frac{t_k(\gamma_{k+1} - \mu)}{\gamma_{k+1}},
	$$
	which implies the recursion, for $\mu > 0$, 
	\begin{equation}\label{eq:gammatk}
		\frac{1}{\gamma_{k+1}} - \frac{1}{\mu}  = \frac{1}{1+t_k\mu}\left (\frac{1}{\gamma_{k}} - \frac{1}{\mu} \right ).
	\end{equation}
	Starting from $\gamma_0>0$ and then solving \eqref{eq:gammatk}, we get the formula \cref{eq:gammak}.
	%\cref{eq:gammak} of $\gamma_k$.
\end{proof}

The identity also implies if $\gamma_0 > \mu$, then $\{\gamma_k\}$ is decreasing and converges to $\mu$ from above. If $\gamma_0 < \mu$, then $\{\gamma_k\}$ is increasing and converges to $\mu$ from below. And if $\gamma_0 = \mu$, then $\gamma_k = \mu$ for all $k\geqslant 1$. For all cases, we have
$$
\min\{\gamma_0,\mu\} \leq \lambda_k \leq \max\{\gamma_0,\mu\}, \quad k=0,1,2,\ldots .
$$
%\begin{lem}
% Let $L\alpha_k^2 = \gamma_k$ and $\gamma_{k+1} = \gamma_k/(1+\alpha_k)$. Chose $\alpha_0 > 0$ s.t. $\alpha_1^2 = \alpha_0^2/(1+\alpha_0) \leqslant 9$. Then for all $k\geqslant 1$
%\begin{equation}\label{eq:boundalpha}
% \alpha_k \leqslant \frac{3}{k}, \quad \gamma_k \leqslant \frac{9L}{k^2}.
%\end{equation}
% Consequently, for $k\geqslant 1$, 
% \begin{equation}
%\rho_k : = \prod_{i=0}^{k-1}\frac{1}{1+\alpha_i} \leqslant \frac{9}{\gamma_0} \frac{L}{k^2}.
%\end{equation}
%\end{lem}
%\begin{proof}
% The relation of $\alpha_k$ and $\gamma_k$ implies the recursion $\alpha_{k+1}^2 = \alpha_k^2/(1+\alpha_k).$ By the choice of $\alpha_0$, \eqref{eq:boundalpha} holds for $k=1$. Assume it holds for $k$. Then by induction
%$$
%\alpha_{k+1}^2 = \frac{\alpha_k^2}{1+\alpha_k} \leqslant \frac{9}{k^2}\frac{1}{1+\frac{3}{k}} \leqslant \frac{9}{(k+1)^2} \quad \text{ for } k\geqslant 1.
%$$ 
%The contraction rate $1/(1+\alpha_k) = \alpha_{k+1}^2/\alpha_k^2$ then follows from the bound of $\alpha_k$. 
%\end{proof}

We then consider the ratio $\alpha_k^2/\gamma_k$ is bounded below, which  leads to accelerated rate. For simplicity, we present the results for the case $\gamma_0 = r L\geqslant \mu$. 
Refined analysis involving optimized $\gamma_0$ can be found in \citet[Lemma 2.2.4]{Nesterov:2018Lectures}.
 %occurs mainly in \cref{sec:NAG}.
\begin{lem}
	%	Assume $\gamma_0,\,L>0,\,\mu\geqslant 0$ and $L\geqslant \mu$.
	Given $L\geqslant \mu\geqslant 0,\,\gamma_0= r L\geqslant \mu$, define $\{(\alpha_k,\gamma_k)\}_{k=0}^\infty$ by 	
	%Assume $(\alpha_k,\gamma_k)$ are determined by, for $k=0,1,2, \ldots$,  
	\begin{equation}\label{eq:gk-ak}
		\left\{
		\begin{aligned}
			\gamma_{k+1}  ={}&	\gamma_k+ \alpha_k (\mu - \gamma_{k+1}),\\
			L\alpha_k^2 = {}&	\gamma_k(1 +B\alpha_k), \quad \alpha_k >0,
		\end{aligned}
		\right.
	\end{equation}
	where $B\geqslant 0$.  
	Then $\gamma_k>0$ and we have the following bound of $\rho_k$.
	\begin{itemize}
		\item If $B=0$, then 
		\begin{equation}\label{eq:decayrho-B0}
			\rho_{k} \leqslant 
			\min\left\{
			\left (\frac{\sqrt{r+1}+1}{\sqrt{r+1}+1 + \sqrt{r} k} \right )^2,\,
			\left(1+\sqrt{\frac{\mu}{L}}\right)^{-k}
			\right\}.
		\end{equation}
%		where $\gamma_{\min}:=\min\{\gamma_0,\mu\},\,\gamma_{\max}:=\max\{\gamma_0,\mu\}$.
		\item 	If $B\geqslant 1/2$, then 
		\begin{equation}\label{eq:decayrho-B>1/2}
			\rho_{k} \leqslant 
			\min\left\{
			\left (\frac{2}{2 + \sqrt{r} \, k} \right )^2,\,
			\left(1+\sqrt{\frac{\mu}{L}}\right)^{-k}
			\right\}.
		\end{equation}
	\end{itemize}
\end{lem}
\begin{proof}
Consider the difference of $1/\sqrt{\rho_k}$ and use \eqref{eq:lowalpharho}, we get
	\[
	\begin{aligned}
		\frac{1}{\sqrt{\rho_{k+1}}} - \frac{1}{\sqrt{\rho_k}}
		={}& \frac{\rho_k - \rho_{k+1}}{\sqrt{\rho_k \rho_{k+1}} (\sqrt{\rho_k} + \sqrt{\rho_{k+1}})}=\frac{\alpha_k}{\sqrt{\rho_k}(1+\sqrt{1+\alpha_k})}
%		\\
		%	\geqslant \frac{\alpha_k}{\sqrt{\rho_k(1+\alpha_k)}}\\
		%	\geqslant
		= \sqrt{r}\frac{\sqrt{1+B\alpha_k}}{1+\sqrt{1+\alpha_k}}.
	\end{aligned}
	\]
	If $B = 0$, then $L\alpha_k^2=\gamma_k\leqslant \max\{\gamma_0,\mu\} = rL$. Therefore, $\alpha_k\leqslant \sqrt{r}$ and thus 
	\[
	\frac{1}{\sqrt{\rho_{k+1}}} - \frac{1}{\sqrt{\rho_k}}=
	\frac{1}{1+\sqrt{1+\alpha_k}}\geqslant 
	\frac{1}{\sqrt{r+1} + 1}.
	\]
	Therefore, we obtain 
	\begin{equation}\label{eq:est-1}
		\rho_k\leqslant  \left (\frac{\sqrt{r+1}+1}{\sqrt{r+1}+1 + \sqrt{r} k} \right )^2.
	\end{equation}
	If $B\geqslant 1/2$, then consider the function 
	\[
	\chi(t): = \frac{\sqrt{1+Bt}}{1+\sqrt{1+t}},\quad\,t\geqslant 0.
	\]
	An elementary calculation proves that $\chi(t)\geqslant \chi(0)=1/2$ for all $t\geqslant 0$. Therefore, we have
	\[
	\frac{1}{\sqrt{\rho_{k+1}}} - \frac{1}{\sqrt{\rho_k}}
	\geqslant \frac{1}{2}\sqrt{r},
	\]
	which implies 
	%	\[
	%	\frac{1}{\sqrt{\rho_k}}\geqslant 
	%	\frac{k}{2}\sqrt{\frac{\gamma_0}{L}}+1.
	%	\]
	%	Therefore, we have
	\begin{equation}\label{eq:bd-mu0-NAG}
		\rho_k \leqslant \left (\frac{2}{2 + \sqrt{r} k} \right )^2.
	\end{equation}		
	Note that both of the sublinear rates \cref{eq:est-1,eq:bd-mu0-NAG} hold for all $\mu\geqslant 0$. 
	
	If $\mu>0$, 
	then it is evident that 
	\begin{equation}\label{eq:alpha2lowerbound}
		\alpha_k^2=\frac{\gamma_k}{L}(1+B\alpha_k)\geqslant
		\frac{\gamma_k}{L} \geqslant \frac{\mu}{L},
	\end{equation}
	so we have that
	\[
	\rho_k\leqslant 
	\left(1+\sqrt{\frac{\mu}{L}}\right)^{-k}.
	\]	
	This finishes the proof of this lemma.
\end{proof}

\section{Gradient Flow and Euler Methods}
\label{sec:GD-Euler}
In this section we will study the gradient flow. The implicit Euler method is equivalent to the proximal method and the explicit Euler method to gradient descent methods. Convergence analyses are derived from the strong Lyapunov property for various Lyapunov functions. 
\subsection{Gradient flow}\label{sec:gradient}
The simplest dynamic system is the gradient flow
\begin{equation}\label{eq:gf}
	x'(t) = -\nabla f(x(t)),
\end{equation}
with initial condition $x(0) = x_0$. Namely $\mathcal G(x) = - \nabla f(x)$. 
A natural Lyapunov function is the so-called optimality gap
\begin{equation}
	\label{eq:Lt-fx}
	\mathcal L(x) = f(x)-f(x^*).
\end{equation}
%For this example, $-\nabla \mathcal L(x)\cdot \mathcal G(x)=\nm{\nabla f(x)}_*^2\geqslant 0$.
%which implies the convergence $f(x)$ to $f(x^*)$. 
\subsubsection{Strongly convex case}
To get the convergence rate, we verify the strong Lyapunov property. When $f\in\mathcal S_{\mu}^{1}$ with $\mu>0$, by \cref{eq:optgapmu}, we get the global strong Lyapunov property $\mathcal A(2\mu, 1, 0)$ or $\mathcal A(\mu, 1, p)$ with $p^2 = \|\nabla f(x)\|_*^2/2$. Consequently, by \cref{thm:strongLya}, this yields the exponential decay $\mathcal O(e^{-t})$ of the optimality gap $f(x(t)) - f(x^*)$ along the trajectory of the gradient flow \cref{eq:gf}. 

We have more candidates for strong Lyapunov functions besides the optimality gap \cref{eq:Lt-fx}. For example, we can use the squared distance 
$	\mathcal L(x) = \frac{1}{2}\| x - x^* \|^2.$
If $f\in\mathcal S_{\mu,L}^{1}$ with $0<\mu\leqslant L<\infty$, then using \eqref{eq:refineMxstar} implies that 
\[
\begin{split}
	-\nabla \mathcal L(x)\cdot \mathcal G(x)  =\dual{\nabla f(x),x-x^*}
	\geqslant{}&\frac{\mu L}{L+\mu} 
	\| x-x^*\|^2 + \frac{1}{ L+\mu}\| \nabla f(x) \|_*^2\\
	={}&\frac{2\mu L}{L+\mu}\mathcal L(x)
	+ \frac{1}{ L+\mu}\| \nabla f(x) \|_*^2.
\end{split}
\]
Hence $\mathcal  L$ satisfies $\mathcal A(2\mu L/(L+\mu),1,p)$ with $p^2(x)=\nm{\nabla f(x)}_*^2/(L+\mu)$.
The extra positive term $p^2$ is useful for the analysis of the gradient descent method.

Another choice is the combination of the previous two:
\begin{equation}\label{eq:Lt-gd}
	\mathcal L(x) =  f(x) - f(x^*) + \frac{\mu}{2}\|x - x^*\|^2.
\end{equation}
If $f\in\mathcal S_{\mu}^{1}$ with $\mu>0$, then by \eqref{eq:Mxstar}, we have
$$
\dual{\nabla f(x), x - x^*} \geqslant f(x) -f(x^*) + \frac{\mu}{2}\| x - x^* \|^2 = \mathcal L(x),
$$
and this bound verifies the strong Lyapunov property $\mathcal A(\mu, 1, p)$ since 
\begin{equation}\label{eq:Lt-cond-gd}
	-\nabla \mathcal L(x)\cdot \mathcal G(x) = \|\nabla f(x)\|_*^2 + \mu(\nabla f(x), x - x^*) \geqslant \mu \mathcal L(x) + \|\nabla f(x)\|^2.
\end{equation}
The above extra positive term $p^2 = \|\nabla f(x)\|_*^2$ is also useful for the analysis of the gradient descent method.

%We start from the  {\it \L ojasiewicz property} (see~\citep{Loja1965}) of a function $f$: there exists 
%$C,r>0$ and $\theta\in [1/2,1]$ such that
%\begin{equation}
%	\label{eq:assum-L}
%	\snm{f(x)-f(x^*)}^{\theta}\leqslant C\nm{\nabla f(x)}_*
%	\quad \forall\,x\in \bar B_r(x^*),
%\end{equation}
%where $\bar B_r(x^*):=\{x\in V:\nm{x-x^*}\leqslant r\}$ is a ball centered at $x^*$ with radius $r$.  
%Then it is straightforward to show
%\begin{equation}\label{eq:diff-L-gf}
%	-\nabla \mathcal L(x)\cdot \mathcal G(x)=
%	%\mathcal L'(t)  = 
%	\nm{\nabla f(x)}_*^2
%	\geqslant C\mathcal L^{2\theta}(x),
%	\quad \forall\,x\in \bar B_r(x^*).
%\end{equation}
%Hence, the strong Lyapunov property holds with $q = 2\theta$.  
%If $x(t)$ is the solution to \cref{eq:gf} and there exists $t_0\geq0$ such that 
%$\{x(t):t\geqslant t_0\}\subset \bar B_r(x^*)$, then by \cref{thm:strongLya}
%for all $t>t_0$, 
%$$
%\mathcal L(x(t))\leqslant  \begin{cases}
%Ce^{-t} & \text{ if } \theta=1/2\\
%C t^{-1/(2\theta -1)} & \text{ if } \theta > 1/2.
%\end{cases}
%$$
%Note that non-convexity is assumed for $f$ but the convergence is local in the sense that $x(t)$ should enter the neighborhood $\bar B_r(x^*)$ eventually. 
%

\subsubsection{Convex case}
When $\mu = 0$, the previous strong Lyapunov properties are degenerated. Hence, coercivity of $f$ is needed. 

Define the sub-level set of $f$ on a given constant value $c$ as
$$
S_c(f) = \{x: f(x)\leqslant c\}.
$$
As $f$ is convex, so is $S_c(f)$. The set $\argmin f$ where $f$ attains its minimum value $f_{\min}$ can be written as $S_{f_{\min}}(f)$.

\begin{lem}
	Let $f$ be convex and coercive. For a given finite value $f_0$, there exists a  constant $R_0$ such that 
	\begin{equation}\label{eq:BL}
		\max_{x^* \in \argmin f} \max_{x\in S_{f_0}}\| x - x^* \|  \leqslant R_0. 
	\end{equation} 
\end{lem}
\begin{proof}
	When $f$ is coercive, $S_{f_0}$ is bounded. Otherwise we can find a sequence $\{x_{n}\}$ with $f\left(x_{n}\right) \leqslant f_0$, but $\left\|x_{n}\right\|>n$, for $n=1,2,\ldots,$ contradicting the coercivity of $f$. As $\argmin f\subseteq S_{f_0}$, it is also bounded and \cref{eq:BL} follows.
\end{proof}
%For any point $x$, the distance $d(x,S)$ will be still denoted by $\nm{x-x^*(x)}_*$ with $x^*(x) = {\rm proj}_S(x)$. 

%\medskip
%
%\noindent coercivity {\it Bounded level set}. Let $f$ be convex and uniformly Lipschitz continuously differentiable. 
%\medskip

For $\mathcal L(x) = f(x)-f(x^*)$, where $x^*$ is an arbitrary but fixed point in the minimum set $\argmin f$, as $-\nabla \mathcal L(x)\cdot \mathcal G(x)=\nm{\nabla f(x)}_*^2\geqslant 0$, we conclude that the trajectory of gradient flow $x(t)$ satisfies $x(t)\in S_{f_0}(f)$ with $f_0 = f(x_0)$. Furthermore, assuming coercivity and using the convexity, we have
\begin{equation}\label{eq:upp0}
	f(x)-f(x^*)\leqslant  \dual{\nabla f(x), x - x^*} \leqslant R_0\nm{\nabla f(x)}_*\quad \forall x\in S_{f_0}(f).
\end{equation}
Then it is straightforward to show
\begin{equation}\label{eq:diff-L-gf}
	-\nabla \mathcal L(x)\cdot \mathcal G(x)=
	%\mathcal L'(t)  = 
	\nm{\nabla f(x)}_*^2
	\geqslant \frac{1}{R_0^2}\mathcal L^2(x)
	\quad \forall\,x\in S_{f_0}(f).
\end{equation}
Hence, the strong Lyapunov property holds with $q = 2$ and $\mathcal W=S_{f_0}(f)$. Applying  \cref{thm:strongLya} implies the sublinear rate $O(1/t)$ of the optimality gap $f(x(t)) - f(x^*)$ along the trajectory of the gradient flow provided the function is coercive and convex.

\subsection{Proximal point algorithm}
Consider the implicit Euler method for the gradient flow
\begin{equation}\label{eq:im}
	x_{k+1}  =  x_k  - \alpha_k \nabla f(x_{k+1}),
\end{equation}
which can be written as 
\begin{equation}\label{eq:im-ge}
	x_{k+1}  =  \proxi_{\alpha_k f}(x_{k}) :=
	\mathop{\argmin}\limits_x	\left\{
	f(x)+\frac{1}{2\alpha_k}\nm{x-x_k}^2
	\right\},
\end{equation}
and is known as the proximal point algorithm (PPA).

\begin{thm}\label{thm:im-gf}
	Assume $f\in\mathcal S_{\mu}^{1}$ with $\mu>0$. The sequence $\{x_k\}$ generated by \cref{eq:im-ge} satisfies 
	\[
	\mathcal L_{k+1}  \leqslant \frac{1}{1 + \mu \alpha_k} \mathcal L_{k},
	\]
	for all $\alpha_k>0$, where 
	\begin{equation}\label{eq:Lk-PPA}
		\mathcal L_k = \mathcal L(x_k) = f(x_k)-f(x^*)+\frac{\mu}{2}\nm{x_k-x^*}^2.
	\end{equation}
\end{thm}
\begin{proof}
	We first use the convexity of $\mathcal L$, then the discretization, and last the strong Lyapunov property \cref{eq:Lt-cond-gd} to get 
	\begin{align*}
		\mathcal L_{k+1} - \mathcal L_{k} 
		&\leqslant(\nabla \mathcal L(x_{k+1}), x_{k+1}-x_k) 
		= \alpha_k(\nabla \mathcal L(x_{k+1}), \mathcal G(x_{k+1})) 
		\leqslant - \mu \alpha_k \mathcal L_{k+1}.
	\end{align*}
	The linear convergence rate then follows.
\end{proof}

\subsection{Gradient descent method}
Next we present analysis for the explicit Euler method for the gradient flow, which is exactly the gradient descent method
\begin{equation}\label{eq:GD}
	x_{k+1}  =  x_k  - \alpha_k \nabla f(x_k).
\end{equation}

\begin{thm}\label{thm:conv-GD}
	Assume $f\in\mathcal S_{\mu,L}^{1,1}$ with $0<\mu\leqslant L<\infty$.  Let $\{x_k\}$ be the sequence generated by  \eqref{eq:GD}. Then for $\alpha_k\leqslant 2/(L+\mu)$, we have
	$$\mathcal L_{k+1} \leqslant(1-\mu\alpha_k) \mathcal L_{k},$$
	where $\mathcal L_k$ is define by \cref{eq:Lk-PPA}. The optimal value $\alpha_k=2/(L+\mu)$ gives 
	$$\mathcal L_{k+1} \leqslant\frac{L-\mu}{L+\mu} \, \mathcal L_{k},$$
	and the quasi-optimal value $\alpha_k = 1/L$ gives
	$$
	\mathcal L_{k+1} \leqslant (1 - \mu/L) \mathcal L_{k}.
	$$
\end{thm}
\begin{proof}
	As $f\in\mathcal S_{\mu,L}^{1}\subset \mathcal S_{\mu}^{1}$, we have verified the strong Lyapunov property $\mathcal A(c,1,p)$ (cf.  \eqref{eq:Lt-cond-gd}). Note that $\mathcal L\in \mathcal S^{1}_{2\mu, L+\mu}$. Using the identity \cref{eq:differenceLya}, the upper bound of $D_{\mathcal L}$, and the strong Lyapunov condition at $x_k$, we have
	\begin{align*}
		\mathcal L_{k+1} - \mathcal L_{k} 
		&\leqslant(\nabla \mathcal L(x_{k}), x_{k+1}-x_k) + \frac{L+\mu}{2}\| x_{k+1}-x_k \|^2 \\
		& = - \alpha_k(\nabla \mathcal L(x_{k}), \nabla f(x_{k}))  + \frac{L+\mu}{2}\| x_{k+1}-x_k \|^2 \\
		&\leqslant- \mu \, \alpha_k \mathcal L_{k}   - \alpha_k \| \nabla f(x_k) \|_*^2 + \frac{L+\mu}{2}\alpha_k^2 \| \nabla f(x_k) \|_*^2.
	\end{align*}
	Then for $\alpha_k \leqslant 2/(L+\mu)$,
	we will have $\mathcal L_{k+1} - \mathcal L_{k} \leqslant - \mu \, \alpha_k \mathcal L_{k}$ and the linear convergence follows.
\end{proof}

One can also choose
\[
\mathcal L(x) = f(x) - f(x^*) \quad \text{or}\quad \mathcal L(x) =\frac{1}{2}\| x - x^*\|^2,
\]
and prove the linear convergence of the gradient descent methods using the strong Lyapunov property. 
The details is left to the interested readers.

\subsection{Proximal gradient method}\label{sec:algo-analy-comp}
\label{sec:prox}
In some applications, the non-smooth convex function $f$ has the composite structure $f = h+g$, where $h\in \mathcal S^{1}_{\mu,L}$ is smooth and $g$ is convex but nonsmooth. We may also assume that $h\in\mathcal S_{0,L}^{1,1}$ and $g\in \mathcal S^0_{\mu}$ as we can split $h+g$ as $(h(x) - \frac{\mu}{2}\|x\|^2) + (g(x) + \frac{\mu}{2}\|x\|^2)$.
One important example is the least absolute shrinkage and selection operator (LASSO) problem~\citep{Tibshirani1996}
\[
\min_x f(x) := \frac{1}{2}\nm{Ax-b}^2 + \rho\nm{x}_{l^1},
\]
which is also known as basis pursuit in the signal processing context~\citep{chen1999}. Application and generalization of LASSO to a variety of problems can be found in~\citet[Table 1]{Tibshirani1996}. 

%We could use the the implicit scheme 
%\begin{equation}\label{eq:prox}
%x_{k+1} = \proxi_{\alpha_k/\gamma_k f}(x_k),
%\end{equation}
%for the scaled gradient flow and obtain a super-linear rate by choosing $\alpha_k \to \infty$. But this is build upon on the efficient computation of the proximal operator. For LASSO problem, it is possible to adapt the efficient semi-smooth Newton's method developed in Li, Sun and Toh~\citep{li_highly_2017}. In general, however, the proximal step \cref{eq:prox} is hard to solve. Instead we consider splitting methods that use an explicit scheme for smooth part $h$ and an implicit scheme for non-smooth part $g$. 

We start from the proximal gradient (PG) method which is also known as operator splitting or forward-backward splitting~\citep{parikh_proximal_2014}. Namely we use an explicit scheme for smooth part $h$ and an implicit scheme for non-smooth part $g$:
\begin{equation}\label{eq:pg}
	\frac{x_{k+1} -x_{k} }{\alpha_k} \in-\nabla h(x_k)- \partial g(x_{k+1}),
\end{equation}
which is an implicit-explicit method for the generalized gradient flow
\[
x'\in\mathcal G\mathcal (x,\partial f(x)) :=-\partial f(x)=-\nabla h(x)-\partial g(x).
\]
It can be written using the proximal operator
\begin{equation}\label{eq:im-ge-comp}
	x_{k+1}  =  \proxi_{\alpha_k g}(x_{k}-\alpha_k\nabla h(x_k)),
\end{equation}
and summarized as the following algorithm.

\begin{algorithm}[H]
	\caption{PG method for minimizing $f=h+g,\,h\in\mathcal S_{\mu,L}^1$ with $0\leqslant \mu\leqslant L<\infty$}
	\label{algo:PG}
	\begin{algorithmic}[1] 
		\REQUIRE  $x_0\in V$.
		\FOR{$k=0,1,\ldots$}		
		\STATE Choose $\alpha_k\in (0, 2/L)$.
		\STATE Compute $y_{k} = x_k - \alpha_k\nabla h(x_k)$.
		\STATE Update $x_{k+1}  =  \proxi_{\alpha_k g}(y_{k})$.
		\ENDFOR		
	\end{algorithmic}
\end{algorithm}

We consider the Lyapunov function \cref{eq:Lt-fx}, i.e., $
	\mathcal L(x) = f(x)-f(x^*).$
Let $d(x)\in\partial f(x)$ be an arbitrary direction in the sub-gradient. Assume $f\in \mathcal S_{\mu}^0$ with $\mu > 0$. Thanks to \cref{coro:bd-non}, we have
\begin{equation}\label{eq:dLG}
	-\partial \mathcal L(x,d(x))\cdot \mathcal G(x,d(x)) = \|d(x)\|_*^2 
	%=\frac{1}{2}\|d(x)\|^2 +\frac{1}{2}\|d(x)\|^2 
	\geqslant \mu \mathcal L(x)+\frac{1}{2}\|d(x)\|_*^2,
\end{equation}
which means $\mathcal L$ satisfies strong Lyapunov property $\mathcal A(\mu,1,p)$ with $p = \nm{d(x)}_*^2/2$. When $\mu=0$, assuming additionally $f$ is coercive, then
\begin{equation}\label{eq:local-Lk-mu-0}
	-\partial \mathcal L(x,d(x))\cdot \mathcal G(x,d(x)) 
	\geqslant \frac{1}{2 R_0^2}\mathcal L^2(\bs x)
	+\frac{1}{2}\nm{d(x)}_*^2,
\end{equation}
where $R_0$ is defined in \cref{eq:BL}. That is when $\mu=0$, $\mathcal L$ satisfies strong Lyapunov property $\mathcal A(1/(2R^2),2,p, S_{f_0})$ with $p = \nm{d(x)}_*^2/2$ for arbitrary direction in the sub-gradient $d(x)\in \partial f(x)$.

\begin{lem}\label{lm:pgdecay}
	Assume $f=h+g$ where $h\in\mathcal S_{\mu,L}^1$ with $0\leqslant \mu\leqslant L<\infty$. Let $\{(x_k, y_k)\}$ be the sequence generated by \cref{algo:PG}. Let 
	\begin{align*}
		q_{k+1} &= (y_{k} -x_{k+1})/\alpha_k\in\partial g(x_{k+1}),\\
		d_{k+1} &= \nabla h(x_{k+1}) + q_{k+1} \in \partial f(x_{k+1}),\\
		d_{k+\frac{1}{2}} &= \nabla h(x_{k}) + q_{k+1} = (x_{k} -x_{k+1})/\alpha_k.
	\end{align*}
	Then for $0\leqslant \alpha_k \leqslant 2/L$,
	$$
	f(x_{k+1}) - f(x_k) \leqslant \alpha_k\left ( \frac{L\alpha_k}{2} - 1\right )\min \left \{\|d_{k+1}\|^2_{*}, \|d_{k+\frac{1}{2}}\|^2_{*}\right \}.
	$$
\end{lem}
\begin{proof}
	Applying \eqref{eq:differenceLya} and \eqref{eq:philowerL} to $h$ and using the convexity of $g$, we have the bound
	\begin{align*}
		f(x_{k+1}) - f(x_k)  &= h(x_{k+1})-h(x_{k}) + g(x_{k+1})-g(x_{k})\\
		%&{}\leqslant \dual{\nabla h_{k+1},x_{k+1}-x_k}-\frac{1}{2L}
		%\nm{\nabla h(x_{k+1})-\nabla h(x_{k})}^2+ \dual{q_{k+1}, x_{k+1}-x_k}\\
		&{}\leqslant \dual{d_{k+1},x_{k+1}-x_k}-\frac{1}{2L}
		\nm{\nabla h(x_{k+1})-\nabla h(x_{k})}_*^2 - \frac{\mu}{2}\| x_{k+1}-x_k\|^2.
	\end{align*}	
	We use the discretization \cref{eq:pg} to replace $x_{k+1} - x_k = -\alpha_k d_{k+\frac{1}{2}}$ and get
	\begin{align*}
		\dual{d_{k+1}, x_{k+1} - x_k} 
		={}& - \alpha_k \dual{d_{k+1}, d_{k+\frac{1}{2}}} 
		={} - \alpha_k \|d_{k+1}\|_*^2 + \alpha_k \dual{d_{k+1}, d_{k+1} - d_{k+\frac{1}{2}}}.
		%		\\
		%		={}& - \alpha_k \|d_{k+\frac{1}{2}}\|^2 - \alpha_k \dual{d_{k+1} - d_{k+\frac{1}{2}}, d_{k+\frac{1}{2}}}.
	\end{align*}
	Note that 
	$
	d_{k+1} - d_{k+\frac{1}{2}} = \nabla h(x+1) - \nabla h(x_k).$
	We can control the cross term as
	\begin{align*}
		\alpha_k \snm{\dual{d_{k+1}, d_{k+1} - d_{k+\frac{1}{2}}} } \leqslant 
		\frac{1}{2L}\nm{\nabla h(x_{k+1}) - \nabla h(x_k)}_*^2
		+\frac{L\alpha_k^2}{2}\nm{d_{k+1}}_*^2.
		%	,\\
		%\alpha_k |\dual{d_{k+1} - d_{k+\frac{1}{2}}, d_{k+\frac{1}{2}}}| \leqslant 
		%	\frac{1}{2L}\nm{\nabla h(x_{k+1}) - \nabla h(x_k)}^2
		%	+\frac{L\alpha_k^2}{2}\nm{d_{k+\frac{1}{2}}}^2.
	\end{align*}
	Adding them together, we get the desired inequality with bound $\|d_{k+1}\|_*$. We can switch to $d_{k+\frac{1}{2}}$ in a similar fashion and obtain a slighter better inequality
	\begin{equation}\label{eq:betterPGestimate}
		f(x_{k+1}) - f(x_k) \leqslant \alpha_k\left ( \frac{(L-\mu)\alpha_k}{2} - 1\right ) \|d_{k+\frac{1}{2}}\|^2_{*},
	\end{equation}
	and the range of the step size can be enlarged to $0<\alpha_k \leqslant 2/(L-\mu)$. 	
\end{proof}

The vector $d_{k+\frac{1}{2}}$ is the so-called gradient mapping~\citep[see][]{Nesterov:2013Introductory,luo_chen_differential_2019}. The gradient $d_{k+1}\in \partial f(x_{k+1})$ is useful in the analysis while $d_{k+\frac{1}{2}}$ is practical in the implementation. 
%Their difference is controllable by $\| \nabla h(x_k) - \nabla h(x_{k+1}) \|_* $. 

\begin{thm}\label{thm:conv-pg-f}
	Assume $f=h+g$ where $h\in\mathcal S_{\mu,L}^1$ with $0\leqslant \mu\leqslant L<\infty$. When $\mu = 0$, assume further that $f$ is coercive. Let the sequence $\{x_k\}$ be generated by \cref{algo:PG} with fixed step size $\alpha_k = 1/L$. Then  we have
	\begin{equation}\label{eq:conv-Lk}
		\mathcal L_{k+1}-\mathcal L_{k}\leqslant 
		\left\{
		\begin{aligned}
			{}&-\frac{\mu}{L }\mathcal L_{k+1}&&\text{ if } \mu>0,\\
			{}&-\frac{1}{2R_0^2} \mathcal L_{k+1}^2&&\text{ if } \mu=0,
		\end{aligned}
		\right.
	\end{equation}
	where $\mathcal L_k = f(x_k)-f(x^*)$. Consequently, for all $ k\geqslant 0$, it holds that 
	\begin{equation}\label{eq:conv-Lk-sgf-comp}
		\mathcal L_k\leqslant 
		\left\{
		\begin{aligned}
			{}&\mathcal L_0(1+\mu/L)^{-k}&&\text{ if } \mu>0,\\
			{}&\left(1+\frac{\mathcal L_0}{1+ 2R_0^2\mathcal L_0}\right)\frac{\mathcal L_{0}}{2R_0^2 + \mathcal L_{0} k}&&\text{ if } \mu=0.
		\end{aligned}
		\right.
	\end{equation}
\end{thm}
\begin{proof}
	First of all, we have the relation	$\mathcal L_{k+1} - \mathcal L_k = f(x_{k+1}) - f(x_k)$ and, by Lemma \ref{lm:pgdecay}, the choice $\alpha_k = 1/L$ implies
	$$
	\mathcal L_{k+1} - \mathcal L_k
	\leqslant {} - \frac{1}{2L}\nm{d_{k+1}}_*^2.
	$$
	The strong Lyapunov property at $x_{k+1}$ reads as
	$$
	\begin{cases}
		\displaystyle \frac{1}{2\mu}\| d_{k+1}\|_*^2 \geqslant f(x_{k+1}) - f(x^*) & \mu > 0,\\
		R_0 \|d_{k+1}\|_* \geqslant f(x_{k+1}) - f(x^*) & \mu = 0,
	\end{cases}
	$$
	which can be proved similarly to \cref{eq:dLG,eq:local-Lk-mu-0}.
	%\eqref{eq:uppmu} and \eqref{eq:upp0}. 
	
	For $\mu > 0$, we then have
	\[
	\begin{split}
		\mathcal L_{k+1} - \mathcal L_k
		\leqslant {} - \frac{1}{2L}\nm{d_{k+1}}_*^2\leqslant -\frac{\mu}{L} \mathcal L_{k+1}.
	\end{split}
	\]
	The desired result \cref{eq:conv-Lk} then follows. 
	
	When $\mu = 0$, we first conclude $\mathcal L_{k+1} - \mathcal L_k\leqslant 0$ for all $k\geqslant 0$. That is $f(x_k) \leqslant f(x_0)$ for all $k\geqslant 0$ so that we can use bound \cref{eq:BL} and the strong Lyapunov property implies
	\[
	\begin{split}
		\mathcal L_{k+1} - \mathcal L_k
		\leqslant {} - \frac{1}{2L}\nm{d_{k+1}}_*^2 \leqslant -\frac{1}{2LR_0^2} \mathcal L_{k+1}^2,
	\end{split}
	\]
	which proves \cref{eq:conv-Lk} for $\mu=0$. The sub-linear rate \eqref{eq:conv-Lk-sgf-comp} follows from Theorem \ref{thm:ode-ineq-q>1-dis}.
\end{proof}

\section{Gradient Flows with Time Scaling}
% and Proximal Method}
\label{sec:Scale-GD-Prox}
In this section we shall introduce a rescaled gradient flow for $f\in\mathcal S_\mu^1$ and deal with the strongly convex case $\mu > 0$ and convex case $\mu = 0$ in a unified way. 
%The implicit Euler method corresponds to the proximal methods and implicit-explicit methods leads to the proximal gradient method for composite convex optimization problems. 
\subsection{Scaled gradient flows}
%In this subsection, we introduce a scaled gradient flow to deal with the strongly convex case $\mu > 0$ and convex case $\mu = 0$ in a unified way. 
Consider the following first-order ODE system
\begin{equation}\label{eq:s-gf}
	\left\{
	\begin{aligned}
		x'={}&-\nabla f(x)/\gamma,\\
		\gamma'={}&\mu-\gamma.
	\end{aligned}
	\right.
\end{equation}
with arbitrary initial value $x(0)=x_0$ and $\gamma(0)=\gamma_0 > 0$.
%Or equivalently, let $\bs x:=(x,\gamma)$ and consider
%\begin{equation}\label{eq:scale-gd}
%\bs x' = \mathcal G(\bs x(t)).
%\end{equation}
Let $\bs x=(x,\gamma)$ and write \cref{eq:s-gf} as
%\[
%\mathcal G(\bs x)
%%= 
%%\begin{pmatrix}
%% \mathcal G^x\\
%%  \mathcal G^{\gamma}
%%\end{pmatrix}
%:=
%\begin{pmatrix}
%	-\nabla f(x)/\gamma\\
%	\mu-\gamma
%\end{pmatrix}.
%\]
%The scaled gradient flow \cref{eq:s-gf} is equivalent to
\[
\bs x'(t) = \mathcal G(\bs x(t)).
\]

We introduce a Lyapunov function
\begin{equation}\label{eq:Lt-rs-gd}
	\mathcal L(\bs x):=f(x)-f(x^*)+\frac{\gamma}{2}\nm{x-x^*}^2.
\end{equation}
%When $\mu = 0$, $x^*$ is an arbitrary but fixed point in the minimum set $\argmin f$. 
Let us verify the strong Lyapunov property as follows
\begin{align}
	-\nabla \mathcal L(\bs x)\cdot \mathcal G(\bs x) = {}&
	\dual{\nabla f(x),x-x^*}+
	\frac{\gamma-\mu}{2}\nm{x-x^*}^2
	+\frac{1}{\gamma}\nm{\nabla f(x)}_*^2\notag\\
	\geqslant{}&f(x)-f(x^*)+\frac{\gamma}{2}\nm{x-x^*}^2+\frac{1}{\gamma}\nm{\nabla f(x)}_*^2\notag\\
	={}&\mathcal L(\bs x)+\frac{1}{\gamma}\nm{\nabla f(x)}_*^2.
	\label{eq:A-re-gd}
\end{align}
Hence $\mathcal L$ is a strong Lyapunov function of $\mathcal G$ with $q= 1,\,c(\bs x)=1$ and $p^2(\bs x)=\nm{\nabla f(x)}_*^2/ \gamma$. By \cref{thm:strongLya}, we have 
\begin{equation}\label{eq:conv-sc-gf}
	\mathcal L(x(t))\leqslant e^{-t}\mathcal L(0),
	\quad\,t\geqslant 0.
\end{equation}
Note that even for $\mu = 0$, we can still achieve the exponential decay, i.e., the linear convergence rate rather than the sub-linear rate. 

%The exponential decay \cref{eq:conv-sc-gf} may 
%be sped or slowed down if we introduce a proper time rescaling
%\[
%t(\tau):=\int_{0}^{\tau}\alpha(s)\,\mathrm ds,
%\]
%where $\alpha:\mathbb R\to\mathbb R_+$ is any positive function. 
%Indeed, we can chose large scaling factor $\alpha$, which is at least bounded below $r\geqslant \tilde{\alpha}>0$, to speed the exponential decay rate
%\[
%e^{-t(\tau)}  = e^{-\int_{0}^{\tau}\alpha(s)\,\mathrm ds}\leqslant e^{-\tilde{\alpha}\tau}.
%\]
%However, if we choose vanishing factor $\alpha(\tau) = 1/(1+\tau)$, then the exponential rate is slowed down and  becomes sublinear rate
%\[
%e^{-t(\tau)}  = e^{-\int_{0}^{\tau}\alpha(s\,\mathrm ds} = \frac{1}{1+\tau}.
%\]
In the continuous level, exponential decay can 
also be viewed as nothing but a time rescaling.
Introducing the 
parameter $\gamma$ governed by the equation 
$\gamma'=\mu-\gamma$ brings the effect of rescaling 
and allows us to handle $\mu\geqslant 0$ in a unified way.

\subsection{Scaled proximal point algorithm}
\label{sec:scale-prox}
%\LC{Merge with previous subsection.}
%As we mentioned before, the implicit Euler method applied to the non-smooth function is known as the {\it proximal method}. 
Convergence analysis of the implicit Euler methods for smooth or non-smooth convex functions are almost identical. Therefore in the following we present the non-smooth case only, i.e., $f\in\mathcal S_\mu^0$ with $\mu\geqslant 0$.
%can be easily transferred to the non-smooth case. 
%We use the scaled gradient flow as an example to illustrate the linear convergence rate. 

Given any time step size $\alpha_k>0$, the implicit Euler method reads as
\begin{equation}\label{implicitEuler-gf-s-non}
	\left\{
	\begin{split}
		\frac{		x_{k+1}-x_{k}}{\alpha_k} \in {}&\mathcal G^x(x_{k+1},\gamma_{k},\partial f(x_{k+1})) := - \frac{1}{\gamma_{k}} \partial f(x_{k+1})\\\frac{		\gamma_{k+1}-\gamma_{k} }{\alpha_k}= {}&
		\mathcal G^\gamma(x_k,\gamma_{k+1}) := \mu - \gamma_{k+1}.
	\end{split}
	\right.
\end{equation}
 Denoted by $t_k=\alpha_k/\gamma_{k},$ the update of $x_{k+1}$ in \eqref{implicitEuler-gf-s-non} can be written using the proximal operator
\begin{equation}\label{eq:scale-prox}
	x_{k+1}  =  \proxi_{t_k f}(x_{k}) :=
	\mathop{\argmin}\limits_x	\left\{
	f(x)+\frac{1}{2t_k}\nm{x-x_k}^2
	\right\}.
\end{equation}
%
%In the scaled proximal point algorithm presented below, the update of parameters $(\alpha_k, \gamma_k)$ are skipped since only the ratio $t_k=\alpha_k/\gamma_{k}$ enters the algorithm. 
%\begin{algorithm}[H]
%	\caption{Scaled PPA for minimizing $f\in\mathcal S_\mu^0$ with $\mu\geqslant 0$}
%	\label{algo:Prox}
%	\begin{algorithmic}[1] 
%		\REQUIRE  $x_0\in V$.
%		\FOR{$k=0,1,\ldots$}		
%		\STATE Choose $t_k > 0$.
%		\STATE Update $x_{k+1}  =  \proxi_{t_k f}(x_{k})$
%		\ENDFOR		
%	\end{algorithmic}
%\end{algorithm}

%and the name proximal method follows. 

We still use the Lyapunov function \cref{eq:Lt-rs-gd} and the strong Lyapunov condition: for any $d(x)\in \partial f(x)$ 
\begin{equation}\label{eq:nonsmoothLya}
	-\partial \mathcal L(\bs x, d(x))\cdot 
	\mathcal G(\bs x, d(x)) 
	\geqslant{}\mathcal L(\bs x)+\frac{1}{\gamma}\nm{d(x)}_*^2.
\end{equation}
can be verified similarly to \eqref{eq:A-re-gd}.
%: for $\bs x = (x, \gamma)$
%\begin{equation*}
%	\mathcal L(\bs x):=f(x)-f(x^*)+\frac{\gamma}{2}\nm{x-x^*}^2.
%\end{equation*}
%and verify $\mathcal A(1,1,p)$ as follows: for any $d(x)\in \partial f(x)$,
%\[
%\begin{split}
%	-\partial \mathcal L(\bs x, d(x))\cdot 
%	\mathcal G(\bs x, d(x)) = {}&
%	\dual{d(x),x-x^*}+
%	\frac{\gamma-\mu}{2}\nm{x-x^*}^2
%	+\frac{1}{\gamma}\nm{d(x)}_*^2\\
%	\geqslant{}&f(x)-f(x^*)+\frac{\gamma}{2}\nm{x-x^*}^2+\frac{1}{\gamma}\nm{d(x)}_*^2\\
%	={}&\mathcal L(\bs x)+\frac{1}{\gamma}\nm{d(x)}_*^2.
%\end{split}
%\]

To study the change of discrete Lyapunov function
\[
\mathcal L_k = f(x_k)-f(x^*)+\frac{\gamma_k}{2}\nm{x_k-x^*}^2,
\]
we shall move along the path $(x_k, \gamma_k) \longrightarrow (x_{k+1}, \gamma_k) \longrightarrow (x_{k+1}, \gamma_{k+1})$. To be able to use the strong Lyapunov property, we will try to evaluate the vector field $\mathcal G$ at $(x_{k+1}, \gamma_{k+1})$.  In \cref{eq:scale-prox}, only the step size $t_k=\alpha_k/\gamma_{k}$ enters the algorithm and $(\alpha_k,\gamma_k)$ can be solved in terms of $t_k$. Therefore in \eqref{eq:Lkdecay} we estimate the rate by $t_k$. 

\begin{thm}\label{thm:im-sgf-non}
	Assume $f$ is in $\mathcal S_{\mu}^{0}$ with $\mu\geqslant 0$. Then for \cref{implicitEuler-gf-s-non} with any $\alpha_k>0$, we have
	\begin{equation*}
		%\label{eq:conv-Lk-sgf}
		\mathcal L_{k+1} \leqslant \frac{1}{1+\alpha_k} \, \mathcal L_{k}.
	\end{equation*}
%	where $\mathcal L_k = f(x_k)-f(x^*)+\frac{\gamma_k}{2}\nm{x_k-x^*}^2$.
	Consequently for any $\gamma_0 > 0$,
	\begin{equation}\label{eq:Lkdecay}
		\mathcal L_k \leqslant \frac{ \mathcal L_0}{1 + \gamma_0\left [\prod_{i=0}^{k-1} (1+t_i\mu) - 1\right ]/\mu},
	\end{equation}
	where $t_k=\alpha_k/\gamma_{k}$ is the rescaled step size and for $\mu=0$ we used the convention \cref{eq:conven}.
\end{thm}
\begin{proof}
	%	The proof is almost identical to that of \cref{thm:implicitEuler-gf-s}.
	First of all, the direction $d_{k+1} = \frac{x_k-x_{k+1}}{t_k}\in\partial f(x_{k+1}).$
	We split the difference as
	\begin{align*}
		\mathcal L_{k+1}- \mathcal L_k={}&
		\mathcal L(x_{k+1},\gamma_{k}) - \mathcal L(x_{k},\gamma_{k})\\
		&+\mathcal L(x_{k+1},\gamma_{k+1}) - \mathcal L(x_{k+1},\gamma_{k}) 
		:={} I_1 + I_2.		
	\end{align*}
	As $\mathcal L$ is linear in $\gamma$ and $\mathcal G^{\gamma}$ is independent of $(x,\gamma)$
	\[
	I_2 = {}\dual{\nabla_{\gamma}\mathcal L( \bs x_{k+1}), \gamma_{k+1} - \gamma_{k}} ={} \alpha_k (\nabla_{\gamma}\mathcal L
	( \bs x_{k+1}), 
	\mathcal G^{\gamma}( \bs x_{k+1},d_{k+1})).
	\]
	As for fixed $\gamma>0$, the function 
	$\mathcal L(\cdot,\gamma)\in \mathcal S_{\gamma+\mu}^{0}$, 
	we obtain that 
	\begin{align*}
		I_{1} \leqslant 	{}&\dual{\partial_x\mathcal L(d_{k+1},\gamma_k),x_{k+1}-x_k}-\frac{\mu+\gamma_k}{2}\nm{x_{k+1}-x_k}^2\\
		={}&\alpha_k\dual{\partial_{x}\mathcal L(d_{k+1},\gamma_{k+1}), \mathcal G^x(d_{k+1},\gamma_{k+1})}
		-\frac{\mu+\gamma_k}{2}\nm{x_{k+1}-x_k}^2	\notag\\
		{}&\quad+\alpha_k\left(\frac{1}{\gamma_{k+1}}-\frac{1}{\gamma_{k}}\right)\nm{d_{k+1}}_*^2\notag\\
		\leqslant{}&\alpha_k\dual{\partial_{x}\mathcal L(d_{k+1},\gamma_{k+1}), \mathcal G^x(d_{k+1},\gamma_{k+1})}
		+		\frac{\alpha_k}{\gamma_{k+1}}
		\nm{d_{k+1}}_*^2.
%		\label{eq:I1-rs-gd}
	\end{align*}
	%	
	%	\[
	%	\begin{aligned}
	%		I_{1} 
	%		\leqslant{}&\alpha_k\dual{d_{k+1}, \mathcal G^x(\bs x_{k+1},d_{k+1})}
	%		+		t_k
	%		\nm{d_{k+1}}_*^2,
	%	\end{aligned}
	%\]
	%	which is similarly derived as \cref{eq:I1-rs-gd}.
	Adding all together and using the strong Lyapunov condition \eqref{eq:nonsmoothLya}, we get 
	\[
	\begin{split}
		\mathcal L_{k+1}- \mathcal L_k
		\leqslant {}&\alpha_k\dual{ \partial \mathcal L(d_{k+1},\gamma_{k+1}), \mathcal G(d_{k+1},\gamma_{k+1})}		+		\frac{\alpha_k}{\gamma_{k+1}}
		\nm{d_{k+1}}_*^2
		\leqslant {}- \alpha_k \mathcal L_{k+1}.
	\end{split}
	\]
	%	This concludes the proof of this theorem.
	
	By recursion, we have $$\mathcal L_k \leqslant \frac{\mathcal L_0}{\prod_{i=0}^{k-1}(1+\alpha_i)} = \mathcal L_0 \prod_{i=0}^{k-1} \frac{\gamma_{i+1}}{(1+t_i\mu) \gamma_i} = \frac{\mathcal L_0}{\prod_{i=0}^{k-1}(1+t_i \mu )} \frac{\gamma_k}{\gamma_0},$$
	and \eqref{eq:Lkdecay} follows from the identity \eqref{eq:gammak} on $\gamma_k$.	
\end{proof}

For explicit methods, sub-linear rate is expected for $\mu = 0$ (see \cref{thm:conv-pg-f}). The implicit method, however, retains the linear rate uniformly for all $\mu \geqslant 0$. The larger is step size $t_k$, the better is the convergence rate. On the other hand, uniform bound $\alpha_k \geqslant \alpha > 0$ implies the exponential increasing of $t_k$ and the proximal operator is harder to evaluate. In the limiting case $t_k = \infty$, it goes back to the original optimization problem.

%%%%%%%%%%%%%%%%%	
\section{Heavy Ball Flow and Momentum Methods}
\label{sec:HB-Momentum}
The well-known heavy ball (HB) flow system~\citep{polyak_methods_1964} reads as follows
\begin{equation}\label{eq:hb}
	x''+\gamma \, x'+\beta\nabla f(x) = 0,\quad t>0,
\end{equation}
where $\gamma,\,\beta>0$ are constant parameters and initial conditions are given by $x(0)=x_0,\,x'(0)=x_1$. Discretization of \eqref{eq:hb} leads to  the so-called heavy ball method (a.k.a the momentum method). In this section we shall study the momentum method using the strong Lyapunov functions for the strongly convex case $f\in\mathcal S_{\mu,L}^1$ with $0<\mu\leqslant L<\infty$.

\subsection{Literature review}
In the early 1960s,~\citet{polyak_methods_1964} considered the HB model \cref{eq:hb} together with its discretization
\begin{equation}\label{eq:hbm}
	x_{k+1} = x_k-s_k\nabla f(x_k) + \alpha_k(x_k-x_{k-1}),
\end{equation}
which covers a large class of iterative solvers for solving linear algebraic systems. 
Applying the asymptotic bound between the matrix norm and the spectral radius,~\citet[Theorem 9]{polyak_methods_1964} established the local convergence result for \cref{eq:hb,eq:hbm} (with stronger smoothness condition $f\in\mathcal C^2$) via spectral analysis and obtained the minimum spectral radius
\[
\rho^*=\rho(\alpha^*,s^*) = \frac{\sqrt{L}-\sqrt{\mu}}{\sqrt{L}+\sqrt{\mu}},
%= \frac{\sqrt{\kappa} -1 }{\sqrt{\kappa} +1 }.
\]
where
\[
\alpha^* =\left(\frac{\sqrt{L}-\sqrt{\mu}}{\sqrt{L}+\sqrt{\mu}}\right)^2 ,\quad s^* = \frac{4}{(\sqrt{L}+\sqrt{\mu})^2}.
\]
%Form this result, we may expect the acceleration of \cref{eq:hbm} with suitable chosen parameters when $x_0$ and $x_1$ are sufficiently close to $x^*$. 

However, it was shown in~\citet{lessard_analysis_2016} that the HB method with parameters optimized for linear ODEs does not guarantee the global convergence for general nonlinear objective function, which justifies the limitation of spectral analysis and the need of Lyapunov analysis. 

Recently,~\citet{ghadimi_global_2015} established 
the first global ergodic convergence rate 
\begin{equation}\label{eq:erg}
	f(\widetilde{x}_k)-f(x^*)\leqslant \left\{
	\begin{aligned}
		&\frac{C}{k+1},&&\mu=0,\\
		&C\left(1-\frac{\mu}{L}\right)^k,&&\mu>0,
	\end{aligned}
	\right.
\end{equation}
for the HB method \cref{eq:hbm}, where $\widetilde x_k = \frac{1}{k+1}\sum_{i=0}^{k}x_i$ denotes the Cesa\'{e}ro average. Later on,
~\citet{sun_non-ergodic_2018} considered $\beta=1,\,\gamma>0$ and the Lyapunov function
\begin{equation}\label{eq:hb-sys-mu=0-Lt}
	\mathcal L(\bs x):=f(x)-f(x^*)+\frac{1}{2}\nm{v}^2,
\end{equation}
which has a nice physical meaning. Indeed, if we treat $x(t)$ as the trajectory of a particle  and understand $f(x)-f(x^*)$ as the potential energy, then $v = x'$ is the velocity and $\frac{1}{2}\nm{v}^2$ is the kinetic energy. Therefore $\mathcal L$ defined in \cref{eq:hb-sys-mu=0-Lt} is the total energy consisting of summation of the potential energy and the kinetic energy. 

Hence, instead of reducing the potential energy only in the gradient flow, the HB flow \cref{eq:hb} reduces the total energy, and an easy calculation yields the global Lyapunov property
%\begin{equation}\label{eq:hb-dL-G}
$-\nabla \mathcal L(\bs x)\cdot \mathcal G(\bs x) = \gamma\nm{v}^2.$
%\end{equation}
Unfortunately, we only have the boundedness
%\begin{equation}\label{eq:db-Lt}
$	\mathcal L(\bs x(t))\leqslant \mathcal L(\bs x_0), 0<t<\infty,$
%\end{equation}
due to the absence of strong Lyapunov property. 
To get the global convergence rate $O(1/t)$, further assumptions (such as coercivity of $f$) are needed which may not be easy to verify. Moreover, ~\citet{sun_non-ergodic_2018} improved the ergodic rate \cref{eq:erg} to non-ergodic sense.

Another choice
$\gamma=2\sqrt{\mu},\,\beta=1$ has been considered in~\citet{Siegel:2019,Wilson:2021}. This is nothing but a rescaling of the time variable $t=\sqrt{\mu}\tau$ for \cref{eq:hb-sys-mu>0} and thus  linear convergence $O(e^{-\sqrt{\mu}\tau})$  can be established under the new time variable $\tau>0$. In~\citet{Siegel:2019,Wilson:2021}, several methods for the HB system \cref{eq:hb-sys-mu>0} with provable accelerated convergence rate $(1-\sqrt{\mu/L})^k$ have been presented. 

We also note that~\citet{shi_understanding_2018} considered $\gamma=2\sqrt{\mu}$ and $\beta = 1+\sqrt{\mu s}$ with $s>0$ and require $f\in\mathcal S_{\mu,L}^{1}\cap\mathcal C^2$. They introduced another Lyapunov function
\[
\mathcal L(\bs x) = f(x(t))-f(x^*) + \frac{1}{4\beta}\nm{x'(t)}^2+\frac{\gamma^2}{4\beta}\nm{ x(t)   + \frac{x'(t)}{\gamma} -x^*}^2,
\]
and obtained the exponential decay 
\[
\mathcal L(x(t))\leqslant e^{-\sqrt{\mu}t/4}\mathcal L(0).
\]
Besides, they also established the linear rate $O(1-\mu/L)^k$ for the corresponding discretization. As we shall show below, our analysis based on strong Lyapunov conditions is simpler.

\subsection{HB model with suitable parameters}
In our recent work \citep{luo_chen_differential_2019}, we considered 
$\gamma=2$ and $\beta=1/\mu$. 
In this case, \cref{eq:hb} is equivalent to a first-order system
\begin{equation}\label{eq:hb-sys-mu>0}
	\left\{
	\begin{aligned}
		x'={}& v-x,\\
		v'={}& x-v-\nabla f(x)/\mu.
	\end{aligned}
	\right.
\end{equation}
Let $\bs x=(x,v)$ and let the above right hand side be $\mathcal G(\bs x)$.
%\[
%\mathcal G(\bs x):=
%\begin{pmatrix}
%	v-x\\
%	x-v-\nabla f(x)/\mu
%\end{pmatrix}.
%\]
The HB system \cref{eq:hb-sys-mu>0} can be abbreviated as $\bs x'(t) = \mathcal G(\bs x(t))$.
%\[
%\bs x'(t) = \mathcal G(\bs x(t)).
%\]
We choose the Lyapunov function 
%\citep[see][]{luo_chen_differential_2019,Siegel:2019,Wilson:2021}
\begin{equation}\label{eq:L-HB}
	\mathcal L(\bs x):=f(x)-f(x^*)+\frac{\mu}{2}\nm{v-x^*}^2.
\end{equation}
Besides, we present the following identity which is trivial but 
very useful for our analysis in both the continuous and discrete level
%\begin{lem}\label{lem:id}
%	For any $u,v,w\in V$, we have
\begin{equation}\label{eq:squares}
	2(u-v,v-w) = \nm{u-w}^2-\nm{u-v}^2-\nm{v-w}^2\quad \forall\,u,v,w\in V.
\end{equation}

%	\end{lem}

We now verify the strong Lyapunov property of \cref{eq:L-HB} as follows. A direct computation leads to
\[
\begin{aligned}
	{}&-\nabla \mathcal L(\bs x)\cdot \mathcal G(\bs x) = 
\dual{\nabla f(x),x-x^*}-\mu(x-v,v-x^*)\notag\\
={}&\dual{\nabla f(x),x-x^*}-\frac{\mu}{2}
\left(\nm{x-x^*}^2-\nm{x-v}^2-\nm{v-x^*}^2\right).
\end{aligned}
\]
Recalling \cref{eq:Mxstar} and our assumption that $f\in\mathcal S_{\mu,L}^1$ with $0<\mu\leqslant L<\infty$, this implies 
\begin{equation}	\label{eq:A-hb}
		-\nabla \mathcal L(\bs x)\cdot \mathcal G(\bs x) 
	\geqslant{} f(x)-f(x^*)+\frac{\mu}{2}\nm{v-x^*}^2 +
	\frac{\mu}{2}\nm{x-v}^2
	={}\mathcal L(\bs x)
	+\frac{\mu}{2}\nm{x-v}^2.
\end{equation}
Consequently $\mathcal L$ is a strong Lyapunov function with $q= 1,\,c(\bs x)=1$ and $p^2(\bs x)=\mu\nm{x-v}^2/2$. By \cref{thm:strongLya}, it follows that $$\mathcal L(x(t))\leqslant 
e^{-t}\mathcal L(0),\quad t\geqslant 0.$$
%Obviously, our choice of parameters and Lyapunov function is much simpler.
\subsection{A semi-implicit discretization}
%for the heavy ball flow}
Given $\alpha_k>0$, we consider a semi-implicit scheme for solving the HB system \cref{eq:hb-sys-mu>0}:
%We consider the Gauss-Seidel method
\begin{equation}\label{eq:hbGS}
	\left\{
	\begin{split}
		x_{k+1} ={}&x_k+\alpha_k (v_{k} - x_{k+1}),\\
		v_{k+1} ={}&v_k+ \alpha_k (x_{k+1} - v_{k+1}) - \frac{\alpha_k}{\mu} \nabla f(x_{k+1}),
	\end{split}
	\right.
\end{equation}
which is a Gauss-Seidel type iteration.
We first treat $v$ is known as $v_k$ and solve the first equation to get $x_{k+1}$ and then with known $x_{k+1}$ to solve the second equation to update $v_{k+1}$.  A discrete analogue to \cref{eq:L-HB} is 
\begin{equation}\label{eq:Lk-HB}
\mathcal L_k:=	\mathcal L(\bs x_k):=f(x_k)-f(x^*)+\frac{\mu}{2}\nm{v_k-x^*}^2,
\end{equation}
where $\bm x_k = (x_k,v_k)$.

\begin{lem}\label{lem:explicitEuler-agf-s}
Assume $f\in\mathcal S_{\mu,L}^1$ with $0<\mu\leqslant L<\infty$. Then for the scheme \cref{eq:hbGS} with any $\alpha_k>0$, we have
	\begin{equation}\label{eq:conv-explicitEuler-agf-s}
		\mathcal L_{k+1} - \mathcal L_{k} \leqslant -\alpha_k \mathcal L_{k+1}+\frac{\alpha_k^2}{2\mu} \nm{\nabla f(x_{k+1})}_*^2.
	\end{equation}
%	where $\mathcal L_k = \mathcal L(\bs x_k)=f(x_k)-f(x^*)+\frac{\mu}{2}\nm{v_k-x^*}^2.$	
\end{lem}
\begin{proof}
	We split the difference along the path $(x_k, v_{k})$ to $(x_{k+1}, v_{k})$ and then to $(x_{k+1}, v_{k+1})$:
	\begin{align*}
		\mathcal L_{k+1} - \mathcal L_k
		={}&  \mathcal L(x_{k+1}, v_{k}) - \mathcal L(x_k, v_{k}) \\
		&+ \mathcal L(x_{k+1}, v_{k+1}) - \mathcal L(x_{k+1}, v_{k}) 
		:={}I_1 + I_2.
	\end{align*}
	Again the idea is to express the difference in terms of $\dual{\nabla \mathcal L(\bs x_{k+1}), \mathcal G(\bs x_{k+1})}$ and then use the strong Lyapunov property.  
	
	For item $I_2$, we use the fact $\mathcal L(x_{k+1}, \cdot)$ is $\mu$-convex to get
	\begin{align}
		I_2\leqslant {} &\dual{\nabla_v \mathcal L(x_{k+1}, v_{k+1}), v_{k+1} - v_k} - \frac{\mu}{2}\nm{ v_{k+1} - v_k}^2\nonumber\\
		={}&\alpha_k \dual{\nabla_v \mathcal L(\bs x_{k+1}), \mathcal G^v(\bs x_{k+1})}- \frac{\mu}{2}\nm{ v_{k+1} - v_k}^2.\nonumber
%		\label{eq:I2}
	\end{align}
	%	In the second step, as the parameter $\gamma$ is canceled in the product $\dual{\nabla_v\mathcal L(\bs x), \mathcal G^v(\bs x)}_*$, we can switch the variable $(x_{k+1}, \gamma_k)$ to $(x_{k+1}, \gamma_{k+1})$. 
		We now estimate $I_1$ as follows
	\begin{equation*}
		\begin{split}
			I_1 = {}&f(x_{k+1})-f(x_k) =\dual{\nabla f(x_{k+1}),x_{k+1}-x_k}-D_{f}(x_k,x_{k+1})\\
			\leqslant {}&\dual{\nabla f(x_{k+1}),x_{k+1}-x_k}- \frac{\mu}{2}\nm{x_{k+1} - x_k}^2\\
			={}&\dual{\nabla_x \mathcal L(\bs x_{k+1}), x_{k+1} - x_k} - \frac{\mu}{2}\nm{x_{k+1} - x_k}^2.
		\end{split}
	\end{equation*}
	In the last step, we can switch from point $(x_{k+1}, v_k)$ to $(x_{k+1}, v_{k+1})$ because $\nabla_x \mathcal L = \nabla f(x)$ is independent of $v$. 
	
	Then we use the discretization \cref{eq:hbGS} to replace $x_{k+1} - x_k$ and compare with the flow evaluated at $\bm x_{k+1}=(x_{k+1}, v_{k+1})$:
	\begin{align*}
		&\dual{\nabla_x \mathcal L(\bs x_{k+1}), x_{k+1} - x_k} 
		={}  \alpha_k \dual{\nabla_x \mathcal L(\bs x_{k+1}), \mathcal G^x(\bs x_{k+1})}  +\alpha_k \dual{\nabla f(x_{k+1}), v_k - v_{k+1}}.
	\end{align*}
	Observing the bound for $I_2$, we  use Cauchy\textendash Schwarz inequality to bound the second term by that
	\begin{equation}\label{eq:df-diff-vk}
		\begin{split}
			\alpha_k \| \nabla f(x_{k+1})\|_*\| v_k - v_{k+1}\|\leqslant {}&
			\frac{\alpha_k^2}{2\mu} \| \nabla f(x_{k+1})\|_*^2 
			+ \frac{\mu}{2}\| v_k - v_{k+1}\|^2.
		\end{split}
	\end{equation}
	Adding all together, we get  
	$$
	\mathcal L_{k+1} - \mathcal L_k
	\leqslant \alpha_k \dual{\nabla \mathcal L(\bs x_{k+1}), \mathcal G(\bs x_{k+1})} + \frac{\alpha_k^2}{2\mu} \| \nabla f(x_{k+1})\|_*^2.
	$$	
	Then applying the strong Lyapunov property $\mathcal A(1,1,p)$ at $\bs x_{k+1}$, cf. \eqref{eq:A-hb}, yields that 
\begin{equation}\label{eq:diff-Lk}
			\mathcal L_{k+1} - \mathcal L_k
	\leqslant {} -\alpha_k\mathcal L_{k+1} 
	-\frac{\mu}{2}\nm{x_{k+1} - x_k}^2
	- \frac{\mu\alpha_k}{2}\nm{x_{k+1} - v_{k+1}}^2 +\frac{\alpha_k^2}{2\mu} \nm{\nabla f(x_{k+1})}_*^2.
\end{equation}
	This completes the proof with extra negative terms.
\end{proof}

Although there are additional negative terms in \cref{eq:diff-Lk}, they cannot be used to cancel the positive term involving $\nm{\nabla f(x_{k+1})}_*^2$ as they are not directly related. 

\subsection{Momentum methods}
To cancel the positive term in the right hand side of \cref{eq:conv-explicitEuler-agf-s}, we add one extra gradient descent step to \cref{eq:hbGS}:
\begin{equation}\label{eq:hbGSGD}
	\left\{
	\begin{split}
		y_{k} ={}&x_k+\alpha_k (v_{k} - y_{k}),\\
		v_{k+1} ={}&v_k+ \alpha_k (y_{k} - v_{k+1}) - \frac{\alpha_k}{\mu} \nabla f(y_{k}),\\
		x_{k+1} = {}&y_k - \frac{1}{L} \nabla f(y_{k}).
	\end{split}
	\right.
\end{equation}
Here we use $y_k$ for the intermediate approximation and $x_{k+1}$ is an extra gradient descent step obtained from $y_k$. Note that $\nabla f(y_k)$ appears twice in each iteration but it can be evaluated only once. 

Following the proof of the previous section, we are able to establish the contraction of the Lyapunov function \cref{eq:Lk-HB} for the modified scheme \cref{eq:hbGSGD}, which gives a momentum method and will be summarized later in \cref{algo:momentum}.
%Later on we shall present a simplified version which is more implementation friendly. 

%Let us consider the Gauss-Seidel method supplemented with an extra gradient descent step. As $\mathcal L\in\mathcal S_{\mu}^{1}$ and 
%$\nabla_x  \mathcal L(\bs x_{k+1}) =\nabla f(x_{k+1})$ is Lipschitz continuous with constant $L$, the scheme \cref{eq:implicit-explicit-gradient} can be formulated as
%\begin{equation}\label{eq:semi-extra-gd}
%	\left\{
%	\begin{split}
%		y_{k} ={}&x_k + \alpha_k (v_k - y_k),\\
%		x_{k+1} = {}&y_k - \frac{1}{L} \nabla f(y_{k})\\
%		v_{k+1} ={}&v_k + \alpha_k (y_{k} -v_k) + \frac{L}{\mu}(x_{k+1} - y_k).
%	\end{split}
%	\right.
%\end{equation}
%Here we switch the ordering in \cref{eq:implicit-explicit-gradient} to update $x_{k+1}$ first and then replace the gradient $\nabla f(y_k)$ for $v_{k+1}$ by $x_{k+1} - y_k$ since in most applications, evaluation of gradient dominates the computational cost. 

%By \cref{thm:conv-sys}, we have the following result.
\begin{thm}
%	Let $\{(x_k,v_k)\}$ be the sequence generated by the iterative method \eqref{eq:hbGSGD}.
Assume $f\in\mathcal S_{\mu,L}^1$ with $0<\mu\leqslant L<\infty$. 
If $\alpha_k$ satisfies $	L\alpha_k^2\leqslant \mu (1+ \alpha_k)$,
	then for \cref{eq:hbGSGD} we have the contraction property
	\begin{equation}\label{eq:hblinear}
		\mathcal L_{k+1}  \leqslant \frac{1}{1 + \alpha_k} \mathcal L_{k}.
	\end{equation}
	In particular, choosing 
	$$
	\alpha_k =\sqrt{\frac{\mu}{L}}, \quad \text{ or } \, \alpha_k = \frac{\mu + \sqrt{\mu^2 + 4 L\mu}}{2L},
	$$
	%	$$,
	we have the accelerated linear convergence rate
	$$
	\mathcal L_{k + 1}  \leqslant \frac{1}{1 +\sqrt{\mu/L}} \mathcal L_k.
	$$
\end{thm}
\begin{proof}
	In \cref{lem:explicitEuler-agf-s}, we have already proved that 
	\begin{equation}\label{eq:GSinequality}
		\begin{split}
			{}&(1+\alpha_k)\mathcal L(y_{k}, v_{k+1}) - \mathcal L(x_{k}, v_{k}) 
			\leqslant{} \frac{\alpha_k^2}{2\mu}\| \nabla f(y_{k})\|_*^2.
		\end{split}
	\end{equation}
	Thanks to \cref{eq:DL} and the extra gradient descent step in \cref{eq:hbGSGD}, we have the sufficient decay
	\begin{equation}\label{eq:GDdecay}
		\mathcal L(x_{k+1}, v_{k+1}) - \mathcal L(y_{k}, v_{k+1})  = f(x_{k+1}) -f(y_k)
		\leqslant- \frac{1}{2L}\| \nabla f(y_{k})\|_*^2.
	\end{equation}
	Multiplying $1+\alpha_k$ to \cref{eq:GDdecay} and adding to \cref{eq:GSinequality}, we get
	$$
	(1+\alpha_k)\mathcal L_{k+1} - \mathcal L_k \leqslant \left ( \frac{\alpha_k^2}{2\mu} - \frac{1+\alpha_k}{2L} \right )\| \nabla f(y_{k})\|_*^2\leqslant 0,
	$$
	as $	L\alpha_k^2\leqslant \mu (1+ \alpha_k)$. This implies \cref{eq:hblinear}. The rest part is obvious and therefore we conclude the proof of this theorem.
	% and 
	%\begin{equation}\label{eq:Lk}
	%\mathcal L_k\leqslant \rho_k\mathcal L_0,\quad k\in\mathbb N,
	%\end{equation}
\end{proof}

We now present a simplified version in the following algorithm 
which drops the sequence $\{v_k\}$ from \eqref{eq:hbGSGD}. Verification of the equivalence is straightforward.
\begin{algorithm}[H]
	%	\caption{Momentum method for solving $\min_x f(x)$ with $f\in S_{\mu, L}^1$}
	\caption{Momentum method for minimizing $f\in\mathcal S_{\mu,L}^1$ with $0<\mu\leqslant L<\infty$}
	\label{algo:momentum}
	\begin{algorithmic}[1] 
		\REQUIRE  $x_0,y_0\in V$.
		%		\STATE Set .
		\FOR{$k=0,1,\ldots$}		
		\STATE Update $x_{k+1} = y_k - \frac{1}{L}\nabla f(y_k)$.
		\STATE Update $y_{k+1} = \left\{
		\begin{aligned}
			{}& \frac{\alpha y_k}{1+\alpha} -\frac{x_k}{(1+\alpha)^2} + \frac{2+\alpha}{(1+\alpha)^2}x_{k+1},&&\text{ if } \alpha =\sqrt{\mu/L},\\
						{}&  \frac{\alpha^2y_k}{(1+\alpha)^2} -\frac{x_k}{(1+\alpha)^2} + \frac{2x_{k+1}}{1+\alpha},&&\text{ if } \alpha =\frac{\mu + \sqrt{\mu^2 + 4 L\mu}}{2L}.
		\end{aligned}
		\right. $
		\ENDFOR		
	\end{algorithmic}
\end{algorithm}
\section{Asymptotic Vanishing Damping System}
In this section we study a second order ODE model, the so-called asymptotic vanishing damping (AVD) system~\citep{Su;Boyd;Candes:2016differential}:
\begin{equation}\label{eq:mf}
	x''+\frac{r}{t}x'+\nabla f(x) = 0,\quad t\geqslant t_1>0,
\end{equation}
%\mnote{Here}\LH{Comment: I used $r$ in stead of $\alpha$.}
where $r>0,\,f\in\mathcal S_{0,L}^1$ is smooth convex and initial conditions are $x(t_1)=x_0,\,x'(t_1)=x_1$. 
\subsection{Existing works}
The AVD  model \cref{eq:mf} was firstly derived and analyzed in~\citet{Su;Boyd;Candes:2016differential} for the case $r\geqslant 3$ then further studied in~\citet{aujol_optimal_2017,attouch_rate_2019} for $r>0$. The positive constant $r$ is very crucial for both the designing of Lyapunov function and the convergence rate analysis.

%	In order to understand 
%an accelerated gradient method with 
%the sublinear rate $\mathcal O(1/k^2)$
%proposed by Nesterov~\citep{Nesterov1983},
%Su, Boyd and  Cand\'es~\citep{Su;Boyd;Candes:2016differential} 
%derived the following second-order ODE
%\begin{equation}\label{eq:Su2015-ode}
%	x'' + \frac{r}{t}x' + \nabla f(x) = 0,
%\end{equation}
%with initial conditions $x(0)=x_0$ and $x'(0)=0$, 
%where $r>0$ and $f\in\mathcal S_{0,L}^{1,1}$.
For $r\geqslant 3$,~\citet[Theorem 5]{Su;Boyd;Candes:2016differential}  proved that
\begin{equation}    \label{eq:conv-Su2015-ode}
	f(x(t)) - f(x^*) \leqslant 
	\frac{(r-1)^2}{2t^2}
	\nm{x_0-x^*}^2,
\end{equation}
by using the Lyapunov function
\begin{equation}\label{eq:L-Su-mu0}
	\mathcal L(t): = f(x(t))-f(x^*)+\frac{(r-1)^2}{2t^2}\nm{x+\frac{t}{r-1}x'-x^*}^2.
\end{equation}
Additionally, if $f$ is strongly convex,
then they also obtained a faster decay rate \citep[Theorem 8]{Su;Boyd;Candes:2016differential}
\begin{equation*}
%	\label{eq:conv-Su2015-ode-fast}
	f(x(t)) - f(x^*) \leqslant 
	Ct^{-2r/3},
\end{equation*}
by the Lyapunov function
\begin{equation*}
%	\label{eq:L-Su-mu>0}
	\mathcal L(t): = f(x(t))-f(x^*)+\frac{2r^2}{9t^2}\nm{x+\frac{3t}{2r}x'-x^*}^2.
\end{equation*}

Later on,~\citet{aujol_optimal_2017} introduced a Lyapunov function
\begin{equation}\label{eq:L-AD}
	\mathcal L(t): = f(x(t))-f(x^*)+\frac{\lambda^2}{2t^2}
	\nm{x+\frac{t}{\lambda}x'-x^*}^2+\frac{\xi}{2t^2}\nm{x(t)-x^*}^2,
\end{equation}
with $		\lambda = {}2\beta\min\{1,r/(2\beta+1)\}$ and $		\xi = {}\lambda\snm{\lambda+1-r}$,
%\begin{equation*}
%	\label{eq:para-L-AD}
%	\left\{
%	\begin{aligned}
%		\xi = {}&\lambda\snm{\lambda+1-r},\\
%		\lambda = {}&2\beta\min\{1,r/(2\beta+1)\},
%	\end{aligned}
%	\right.
%\end{equation*}
and also generalized \eqref{eq:conv-Su2015-ode} to 
\begin{equation}\label{eq:conv-Su2015-ode-r}
	f(x(t)) - f(x^*) \leqslant 
	\left\{
	\begin{aligned}
		&Ct^{-2},&&\text{if}~r\geqslant 2\beta+1,\\
		&Ct^{-2r/(2\beta+1)},&&\text{if}~0<r<2\beta+1,
	\end{aligned}
	\right.
\end{equation}
where $(f-f(x^*))^\beta$ is convex with $\beta>0$.
Around the same time,~\citet{attouch_rate_2019} obtained the estimate~\eqref{eq:conv-Su2015-ode-r} 
for $\beta=1$,  
%\begin{equation*}
%%	\label{eq:conv-Attouch}
%	f(x(t)) - f(x^*) \leqslant 
%	\left\{
%	\begin{aligned}
%		&Ct^{-2},&&\text{if}~r\geqslant 3,\\
%		&Ct^{-2r/3},&&\text{if}~0<\alpha<3,
%	\end{aligned}
%	\right.
%\end{equation*}
with the corresponding Lyapunov function \cref{eq:L-AD} taking $\beta=1$.
More importantly, they considered numerical discretizations for~\eqref{eq:mf}
and proved the sublinear convergence rate
\begin{equation}\label{eq:conv-Attouch-dis}
	f(x_k) - f(x^*) \leqslant 
	\left\{
	\begin{aligned}
		&Ck^{-2},&&\text{if}~r\geqslant 3,\\
		&Ck^{-2r/3},&&\text{if}~0<r<3,
	\end{aligned}
	\right.
\end{equation}
which matches the convergence rate 
\eqref{eq:conv-Su2015-ode-r} with $\beta=1$ for the continuous level.
\subsection{Strong Lyapunov condition}
%We shall establish the strong Lyapunov condition. 
Here, we only focus on the case $r= 3$. As discussed in~\citet{luo_chen_differential_2019}, the AVD model \cref{eq:mf} with $r=3$ is equivalent to NAG flow (see \cref{eq:ode-agf} in \cref{sec:NAG}) with suitable time scaling. We shall apply our strong Lyapunov condition to establish the decay rates of continuous problem and its numerical discretizations.

 Let $	v = x+tx'/2$ and introduce an auxiliary function $\gamma = 4t^{-2}$.
Then the AVD model \cref{eq:mf} with $r=3$ can be rewritten as 
\begin{equation}
	\label{eq:sys-Su}
	\left\{
	\begin{aligned}
		x' = {}&\sqrt{\gamma}(v-x),\\
		v' = {}&-\nabla f(x)/		\sqrt{\gamma},\\
		\gamma' = {}&-\gamma^{3/2},
	\end{aligned}
	\right.
\end{equation}
and the Lyapunov function \cref{eq:L-Su-mu0} reads equivalently as follows
\begin{equation}\label{eq:Lt-AVD}
	\mathcal L(\bm x): = f(x(t))-f(x^*)+\frac{\gamma(t)}{2}\nm{v(t)-x^*}^2,
\end{equation}
where $\bs x=(x, v, \gamma)$. Let us write the right hand side of \cref{eq:sys-Su} as $\mathcal G(\bs x)$.
%\[
%\mathcal G(\bs x):=
%\begin{pmatrix}
%	\sqrt{\gamma}(v-x)\\
%	-\nabla f(x)/\sqrt{\gamma}\\
%	-\gamma^{3/2}
%\end{pmatrix}.
%\]
It follows that
\begin{equation}\label{eq:LG-AVD}
	\begin{split}
		-\nabla \mathcal L(\bs x)\cdot \mathcal G(\bs x)={}&\sqrt{\gamma}\dual{\nabla f(x),x-x^*}+
		\frac{\gamma^{3/2}}{2}\nm{v-x^*}^2\\
		\geqslant{}&\sqrt{\gamma}(f(x)-f(x^*))
		+\frac{\gamma^{3/2}}{2}\nm{v-x^*}^2
		\geqslant{}\sqrt{\gamma}\mathcal L(\bs x).
	\end{split}
\end{equation}
Therefore $\mathcal L$ is a strong Lyapunov function of \cref{eq:mf} with $c(\bm x)=\sqrt{\gamma},\,q(\bm x)= 1$ and $p(\bs x)=0$. By \cref{thm:strongLya}, we obtain the decay rate $	\mathcal L(t)=O(t^{-2})$, which coincides with \cref{eq:conv-Su2015-ode}.

\subsection{Gauss-Seidel iteration with extra gradient step}
Given any time step size $\alpha_k>0$, we consider the following semi-implicit scheme for \cref{eq:sys-Su}:
\begin{subnumcases}{}
	\label{eq:AVD-GS-x}
		\frac{x_{k+1}-x_k}{\alpha_k} = {}\sqrt{\gamma_k}(v_k-x_{k+1}),\\
			\label{eq:AVD-GS-v}
	\frac{v_{k+1}-v_k}{\alpha_k}  = {}-\nabla f(x_{k+1})/		\sqrt{\gamma_k},\\
	\label{eq:AVD-GS-g}	
	\frac{\gamma_{k+1}-\gamma_k}{\alpha_k} = {}-\sqrt{\gamma_{k}}\gamma_{k+1},
	\end{subnumcases}
%\begin{equation}
%	\label{eq:semi-sys-Su-non}
%	\left\{
%	\begin{aligned}
%		\frac{x_{k+1}-x_k}{\alpha_k} = {}&\sqrt{\gamma_k}(v_k-x_{k+1}),\\
%		\frac{v_{k+1}-v_k}{\alpha_k}  = {}&-\nabla f(x_{k+1})/		\sqrt{\gamma_k},\\
%		\frac{\gamma_{k+1}-\gamma_k}{\alpha_k} = {}&-\sqrt{\gamma_{k}}\gamma_{k+1},
%	\end{aligned}
%	\right.
%\end{equation}
which is a Gauss-Seidel type discretization.
Mimicking \cref{eq:Lt-AVD}, for $\bm x_k = (x_k,v_k,\gamma_k)$, we introduce the discrete Lyapunov function
\begin{equation}\label{eq:Lk}
	\mathcal L_k
	:= \mathcal L(\bs x_k)
	= {}f(x_k)-f(x^*) +
	\frac{\gamma_k}{2}
	\nm{v_k-x^*}^2.
\end{equation}
A one iteration analysis is given below.
\begin{lem}
	\label{lem:one-step-AVD-semi}
If $f\in\mathcal S_{0,L}^1$, then for the semi-implicit scheme \cref{eq:AVD-GS-x} with any step size $\alpha_k>0$, we have 
	\begin{equation}\label{eq:one-step-AVD-semi}
		\begin{aligned}
			\mathcal L_{k+1} - \mathcal L_k
			\leqslant {}&-\alpha_k\sqrt{\gamma_k}\mathcal L_{k+1}
			- \frac{\gamma_k}{2}\nm{ v_{k+1} - v_k}^2
%			\\
%			{}&\quad	
			+\alpha_k \sqrt{\gamma_k}\dual{\nabla f(x_{k+1}), v_k - v_{k+1}},
		\end{aligned}
	\end{equation}
%	where $\mathcal L_k
%	= \mathcal L(\bs x_k)
%	= {}f(x_k)-f(x^*) +
%	\frac{\gamma_k}{2}
%	\nm{v_k-x^*}^2,
%	$
	which implies 
	\begin{equation}\label{eq:one-step-AVD-semi-1}
		\mathcal L_{k+1} - \mathcal L_k
		\leqslant {}-\alpha_k\sqrt{\gamma_k}\mathcal L_{k+1} +\frac{\alpha_k^2}{2} \nm{\nabla f(x_{k+1})}_*^2.
	\end{equation}
\end{lem}
\begin{proof}
	Let us calculate the difference $	\mathcal L_{k+1} - \mathcal L_k		=I_1 + I_2 + I_3$ where
	\begin{subnumcases}{}
		\label{eq:I1-AVD}
{}I_1:=		\mathcal L(x_{k+1}, v_{k},\gamma_{k}) - \mathcal L(x_k, v_{k},\gamma_{k}),\\		
{}I_2:=	\mathcal L(x_{k+1}, v_{k+1},\gamma_{k}) - \mathcal L(x_{k+1}, v_{k},\gamma_{k}),\\
{}I_3:=	\mathcal L(x_{k+1}, v_{k+1},\gamma_{k+1}) - \mathcal L(x_{k+1}, v_{k+1},\gamma_{k}).
		\end{subnumcases}
	Set $\tau_k = \sqrt{\gamma_k/\gamma_{k+1}}$. Below, we shall estimate the above three terms one by one.
	
	It is evident that 
	\[
	I_3 = \dual{\nabla_\gamma\mathcal L(\bm x_{k+1}),\gamma_{k+1}-\gamma_{k}} =\alpha_k\tau_k\dual{\nabla_\gamma\mathcal L(\bm x_{k+1}),\mathcal G^\gamma(\bm x_{k+1})}.
	\]
	For item $I_2$, we use the fact $\mathcal L(x_{k+1}, \cdot, \gamma_k)$ is $\gamma_k$-convex to get
\[
	\begin{aligned}
	I_2\leqslant {} &\dual{\nabla_v \mathcal L(x_{k+1}, v_{k+1},\gamma_{k}), v_{k+1} - v_k} - \frac{\gamma_k}{2}\nm{ v_{k+1} - v_k}^2\\
	= {} &-\dual{\sqrt{\gamma_{k}}(v_{k+1}-x^*),\nabla f(x_{k+1})} - \frac{\gamma_k}{2}\nm{ v_{k+1} - v_k}^2\\
	= {} &-\tau_k\dual{\sqrt{\gamma_{k+1}}(v_{k+1}-x^*),\nabla f(x_{k+1})} - \frac{\gamma_k}{2}\nm{ v_{k+1} - v_k}^2,
\end{aligned}
\]
and in view of \eqref{eq:AVD-GS-v}, we have
	\begin{equation}\label{eq:I2}
	I_2\leqslant 	\alpha_k\tau_k \dual{\nabla_v \mathcal L(\bs x_{k+1}), \mathcal G^v(\bs x_{k+1})}- \frac{\gamma_k}{2}\nm{ v_{k+1} - v_k}^2.
	\end{equation}
	We then estimate $I_1$ as follows
	\begin{equation*}
		\begin{split}
			I_1 = {}&f(x_{k+1})-f(x_k) 
			\leqslant {}\dual{\nabla f(x_{k+1}),x_{k+1}-x_k}
			={}\dual{\nabla_x \mathcal L(\bs x_{k+1}), x_{k+1} - x_k}.
		\end{split}
	\end{equation*}
	In the last step, we have switched from point $(x_{k+1}, v_k,\gamma_k)$ to $(x_{k+1}, v_{k+1},\gamma_{k+1})$ because $\nabla_x \mathcal L = \nabla f(x)$ is independent of $(v,\gamma)$. Then we use the discretization \eqref{eq:AVD-GS-x} to replace $x_{k+1} - x_k$ and compare with the flow evaluated at $\bm x_{k+1}=(x_{k+1}, v_{k+1}, \gamma_{k+1})$:
	\begin{align*}
		&\dual{\nabla_x \mathcal L(\bs x_{k+1}), x_{k+1} - x_k} =
%		\\
%		={}&
		  \alpha_k \tau_k\dual{\nabla_x \mathcal L(\bs x_{k+1}), \mathcal G^x(\bs x_{k+1})}  +\alpha_k \sqrt{\gamma_k}\dual{\nabla f(x_{k+1}), v_k - v_{k+1}}.
	\end{align*}
	
	Whence, adding all together and applying the strong Lyapunov condition $\mathcal A(\sqrt{\gamma},1,0)$ at $\bs x_{k+1}$ (cf. \cref{eq:LG-AVD}) yield that 
	\[
	\begin{aligned}
		\mathcal L_{k+1} - \mathcal L_k
		\leqslant {}&-\alpha_k\sqrt{\gamma_k}\mathcal L_{k+1}
		- \frac{\gamma_k}{2}\nm{ v_{k+1} - v_k}^2
%		\\
%		{}&\quad
			+\alpha_k \sqrt{\gamma_k}\dual{\nabla f(x_{k+1}), v_k - v_{k+1}}.
	\end{aligned}
	\]
	This proves \cref{eq:one-step-AVD-semi}. Besides, applying Cauchy\textendash Schwarz inequality gives \cref{eq:one-step-AVD-semi-1} and completes the proof of this lemma.
\end{proof}

Obviously, one cannot obtain contraction result of $\mathcal L_k$ from \cref{lem:one-step-AVD-semi}. To cancel the positive term in \cref{eq:one-step-AVD-semi-1}, we then modify \cref{eq:AVD-GS-v} by adding one gradient descent step:
\begin{subnumcases}{}
		\label{eq:semi-sys-Su-non-mod}
			\frac{y_{k}-x_k}{\alpha_k} = {}\sqrt{\gamma_k}(v_k-y_{k}),\\
				\label{eq:semi-sys-Su-non-mod-v}
	\frac{v_{k+1}-v_k}{\alpha_k}  = {}-\nabla f(y_{k})/		\sqrt{\gamma_k},\\
		\label{eq:semi-sys-Su-non-mod-x}
	x_{k+1} -y_k= {}-\frac{1}{L}\nabla f(y_k),\\
	\label{eq:semi-sys-Su-non-mod-g}	
	\frac{\gamma_{k+1}-\gamma_k}{\alpha_k} = {}-\sqrt{\gamma_{k}}\gamma_{k+1}.
	\end{subnumcases}
%\begin{equation}
%%	\label{eq:semi-sys-Su-non-mod}
%	\left\{
%	\begin{aligned}
%		\frac{y_{k}-x_k}{\alpha_k} = {}&\sqrt{\gamma_k}(v_k-y_{k}),\\
%		\frac{v_{k+1}-v_k}{\alpha_k}  = {}&-\nabla f(y_{k})/		\sqrt{\gamma_k},\\
%		x_{k+1} -y_k= {}&-\frac{1}{L}\nabla f(y_k),\\
%		\frac{\gamma_{k+1}-\gamma_k}{\alpha_k} = {}&-\sqrt{\gamma_{k}}\gamma_{k+1}.
%	\end{aligned}
%	\right.
%\end{equation}
Thanks to \cref{lem:one-step-AVD-semi}, we have 
\[
\widehat{	\mathcal L}_{k+1} - \mathcal L_k
\leqslant {} -\alpha_k\sqrt{\gamma_k}\widehat{\mathcal L}_{k+1} +\frac{\alpha_k^2}{2} \nm{\nabla f(y_{k})}_*^2,
\]
where 
\begin{equation}\label{eq:hat-Lk1}
	\widehat{	\mathcal L}_{k+1}:= 
	f(y_k)-f(x^*)+\frac{\gamma_{k+1}}{2}\nm{v_{k+1}-x^*}^2.
\end{equation}
Moreover, by \cref{eq:DL} and the gradient descent step of $x_{k+1}$ in \cref{eq:semi-sys-Su-non-mod}, we see that 
\begin{equation}\label{eq:diff-Lk1-hLk1}
	\mathcal L_{k+1} -\widehat{	\mathcal L}_{k+1}  = f(x_{k+1})-f(y_k) \leqslant -\frac{1}{2L}\nm{\nabla f(y_k)}_*^2.
\end{equation}
This promises the contraction property 
\begin{equation}\label{eq:conv-Lk-AVD}
	\mathcal L_{k+1} - \mathcal L_k
	\leqslant {} -\alpha_k\sqrt{\gamma_k}\mathcal L_{k+1},
\end{equation}
provided that $L\alpha_k^2\leqslant  1+\alpha_k\sqrt{\gamma_k}$. 

Before the convergence analysis, let us simplify \cref{eq:semi-sys-Su-non-mod}. If $L\alpha_k^2=  1+\alpha_k\sqrt{\gamma_k}$, then by \eqref{eq:semi-sys-Su-non-mod} and \eqref{eq:semi-sys-Su-non-mod-v}, we have
\[
v_{k+1} = x_{k+1}+\frac{x_{k+1}-x_k}{\alpha_k\sqrt{\gamma_k}}.
\]
Plugging this into \eqref{eq:semi-sys-Su-non-mod} and using \eqref{eq:semi-sys-Su-non-mod-g} imply that
\[
y_{k+1} = x_{k+1} + \frac{x_{k+1}-x_{k}}{L\alpha_k^2\sqrt{L}\alpha_{k+1}}.
\]
Thus the sequences $\{v_k\}$ and $\{\gamma_k\}$ can totally be dropped. 
%Below, we summarize the above simplification in \cref{algo:AGM}, which is named by accelerated gradient method (AGM). 
%\LH{Please check this algorithm again.}
%\LC{Is this algorithm important? No need to formulate all in Algorithm. Only important ones. I don't think this algorithm is better than NAG.}
%\mnote{Here}
%\begin{algorithm}[H]
%	\caption{AGM for minimizing $f\in\mathcal S_{0,L}^1$}
%	\label{algo:AGM}
%	\begin{algorithmic}[1] 
%		\REQUIRE  $\alpha_0 \geqslant 1/\sqrt{L},\,x_0,\,y_0 \in V$.
%%		\STATE Initialize $\alpha_0=\left (\sqrt{\gamma_0}+ \sqrt{\gamma_0+4L}\right )/(2L)$.
%		\FOR{$k=0,1,\ldots$}		
%%		\STATE Compute $\alpha_k = \left (\sqrt{\gamma_k}+ \sqrt{\gamma_k+4L}\right )/(2L)$.
%%				\smallskip
%		\STATE Compute $\theta_k =\sqrt{L}-\frac{1}{\sqrt{L}\alpha_k^2}$.		
%				\smallskip
%		\STATE Update $\alpha_{k+1} = \left (\sqrt{\theta_k}+ \sqrt{\theta_k+4L}\right )/(2L)$.				
%		\smallskip
%		\STATE Compute $\displaystyle x_{k+1}= y_{k} - \frac{1}{L}\nabla f(y_{k})$.
%		\smallskip
%		\STATE Update $\displaystyle y_{k+1} = x_{k+1} + \frac{x_{k+1}-x_{k}}{L\alpha_k^2\sqrt{L}\alpha_{k+1}}$.
%%		\smallskip
%%		\STATE Update $	\frac{\gamma_{k+1}-\gamma_k}{\alpha_k} = {}-\sqrt{\gamma_{k}}\gamma_{k+1}$.
%		\ENDFOR		
%	\end{algorithmic}
%\end{algorithm}
\begin{thm}
	\label{thm:one-step-AVD-semi-mod}
	For \cref{eq:semi-sys-Su-non-mod}, we have 
	\begin{equation}\label{eq:one-step-AVD-semi-mod}
		\mathcal L_{k+1} - \mathcal L_k
\leqslant {} -\alpha_k\sqrt{\gamma_k}\mathcal L_{k+1}.
	\end{equation}
This implies that
	\[
	\mathcal L_k\leqslant \mathcal L_0\times \prod_{i=0}^{k-1}\frac{1}{1+\alpha_i\sqrt{\gamma_i}}=\frac{\gamma_k}{\gamma_0}\mathcal L_0,
	\]
	where the rate of convergence is given by, with $r = \gamma_0/L$,
	%	\[
	%			\frac{\gamma_k}{\gamma_0}\leqslant \left(1+\frac{r\sqrt{\gamma_0}}{3r\sqrt{\gamma_0}+2}k\right)^{-2}.
	%	\]
	\begin{equation}\label{eq:est-gk}
		%\frac{num}{den}\leqslant 
		%		\frac{\gamma_k}{\gamma_0}\leqslant \left(1+\alpha/(3\alpha+2\gamma_0^{-1/2})k\right)^{-2}.
		%			\frac{\gamma_k}{\gamma_0}\leqslant \left(1+\frac{r\sqrt{\gamma_0}}{3r\sqrt{\gamma_0}+2}k\right)^{-2}.
		\frac{\gamma_{k}}{\gamma_0}\leqslant \left (1+ \frac{\sqrt{r}}{2+\sqrt{r}}\right )^2 			\left (\frac{2}{2 + \sqrt{r} \, k} \right )^2.
	\end{equation}
\end{thm}
\begin{proof}
By the above discussion, the contraction \cref{eq:one-step-AVD-semi-mod} is evident since $L\alpha_k^2= 1+\alpha_k\sqrt{\gamma_k}$. According to the equation of $\{\gamma_k\}$ in \eqref{eq:semi-sys-Su-non-mod-g}, it is clear that 
	\[
	\frac{\gamma_k}{\gamma_0}=\prod_{i=0}^{k-1}\frac{1}{1+\alpha_i\sqrt{\gamma_i}}.
	\]
	It remains to prove the decay rate of $\gamma_k$. We have $\gamma_{k}\geqslant \gamma_{k+1}$ and thus
	\[
	\sqrt{\gamma_{k+1}} -\sqrt{\gamma_{k}}=\frac{\gamma_{k+1} - \gamma_k}{\sqrt{\gamma_{k+1}} + \sqrt{\gamma_{k}}} = -\frac{\alpha_k\sqrt{\gamma_k}\gamma_{k+1}}{\sqrt{\gamma_{k+1}} + \sqrt{\gamma_{k}} }\leqslant -\frac{\alpha_k}{2}\gamma_{k+1}.
	\]
	As $\alpha_k = \sqrt{1+\alpha_k\sqrt{\gamma_k}}/\sqrt{L}\geqslant 1/\sqrt{L}$, we have
		\[
	\sqrt{\gamma_{k+1}} -\sqrt{\gamma_{k}}\leqslant -\frac{1}{2\sqrt{L}}\gamma_{k+1}.
	\]
Applying \cref{thm:ode-ineq-q>1-dis} (4) to the sequence $\{\sqrt{\gamma_k}\}$, we get the decay estimate \cref{eq:est-gk} and finish the proof of this theorem.
\end{proof}
\subsection{Gauss-Seidel iteration with extrapolation}
Instead of \cref{eq:semi-sys-Su-non-mod}, let us consider another modified scheme
\begin{subnumcases}{}
	\label{eq:mod-GS-AVD-y}
			\frac{y_{k}-x_k}{\alpha_k} = {}\sqrt{\gamma_k}(v_k-y_{k}),\\
				\label{eq:mod-GS-AVD-v}
	\frac{v_{k+1}-v_k}{\alpha_k}  = {}-\nabla f(y_{k})/		\sqrt{\gamma_k},\\
		\label{eq:mod-GS-AVD-x}
	\frac{x_{k+1}-x_k}{\alpha_k} = {}\sqrt{\gamma_k}(v_{k+1}-x_{k+1}),\\
		\label{eq:mod-GS-AVD-g}
	\frac{\gamma_{k+1}-\gamma_k}{\alpha_k} = {}-\sqrt{\gamma_{k}}\gamma_{k+1},
	\end{subnumcases}
%\begin{equation}
%	\label{eq:semi-sys-Su-non-mod-2}
%	\left\{
%	\begin{aligned}
%		\frac{y_{k}-x_k}{\alpha_k} = {}&\sqrt{\gamma_k}(v_k-y_{k}),\\
%		\frac{v_{k+1}-v_k}{\alpha_k}  = {}&-\nabla f(y_{k})/		\sqrt{\gamma_k},\\
%		\frac{x_{k+1}-x_k}{\alpha_k} = {}&\sqrt{\gamma_k}(v_{k+1}-x_{k+1}),\\
%		\frac{\gamma_{k+1}-\gamma_k}{\alpha_k} = {}&-\sqrt{\gamma_{k}}\gamma_{k+1},
%	\end{aligned}
%	\right.
%\end{equation}
where we used an extrapolation step \eqref{eq:mod-GS-AVD-x} to update $x_{k+1}$. This is different from the gradient descent step in \cref{eq:semi-sys-Su-non-mod}.
By \cref{lem:one-step-AVD-semi}, we have 
\[
%\widehat{	\mathcal L}_{k+1} - \mathcal L_k
%\leqslant {} -\alpha_k\sqrt{\gamma_k}\widehat{\mathcal L}_{k+1} +\frac{\alpha_k^2}{2} \nm{\nabla f(y_{k})}^2,
\begin{aligned}
	\widehat{	\mathcal L}_{k+1} - \mathcal L_k
	\leqslant {}&-\alpha_k\sqrt{\gamma_k}\, \widehat{	\mathcal L}_{k+1}
	- \frac{\gamma_k}{2}\nm{ v_{k+1} - v_k}^2
%	\\
%	{}&\quad	
	+\alpha_k \sqrt{\gamma_k}\dual{\nabla f(y_{k}), v_k - v_{k+1}},
\end{aligned}
\]
where $\widehat{	\mathcal L}_{k+1}$ is defined by \cref{eq:hat-Lk1}.
%\[
%\widehat{	\mathcal L}_{k+1}:= 
%f(y_k)-f(x^*)+\frac{\gamma_{k+1}}{2}\nm{v_{k+1}-x^*}^2.
%\]
In addition, \cref{eq:diff-Lk1-hLk1} becomes 
\[
\mathcal L_{k+1} -\widehat{	\mathcal L}_{k+1}  = f(x_{k+1})-f(y_k) \leqslant \dual{\nabla f(y_k),x_{k+1}-y_k}+\frac{L}{2}\nm{x_{k+1}-y_k}^2.
\]
Combining \eqref{eq:mod-GS-AVD-y} with \eqref{eq:mod-GS-AVD-x} gives the relation
\[
(1+\alpha_k\sqrt{\gamma_k})(x_{k+1}-y_k) = \alpha_k\sqrt{\gamma_k}(v_{k+1}-v_k),
\]
which implies that
\[
\mathcal L_{k+1} -\widehat{	\mathcal L}_{k+1}\leqslant \frac{\alpha_k \sqrt{\gamma_k}}{1+\alpha_k\sqrt{\gamma_k}}\dual{\nabla f(y_{k}), v_{k+1}-v_k }+\frac{L\alpha_k^2\gamma_k}{(1+\alpha_k\sqrt{\gamma_k})^2}\nm{v_{k+1}-v_k}^2.
\]
Therefore, if $L\alpha_k^2\leqslant 1+\alpha_k\sqrt{\gamma_k}$, then the contraction \cref{eq:conv-Lk-AVD} follows immediately.
%\[
%\mathcal L_{k+1} - \mathcal L_k
%\leqslant {} -\alpha_k\sqrt{\gamma_k}\mathcal L_{k+1}.
%\]

Moreover, if $L\alpha_k^2=  1+\alpha_k\sqrt{\gamma_k}$, then we claim that \cref{eq:mod-GS-AVD-g} coincides with \cref{eq:semi-sys-Su-non-mod}. It is sufficient to verify that \eqref{eq:mod-GS-AVD-x} is identical to \eqref{eq:semi-sys-Su-non-mod-x}. indeed, inserting \eqref{eq:semi-sys-Su-non-mod} and \eqref{eq:semi-sys-Su-non-mod-v} into \eqref{eq:semi-sys-Su-non-mod-x} gives
\[
\begin{aligned}
	x_{k+1} = {}&\frac{x_k + \alpha_{k}\sqrt{\gamma_k}v_{k+1}}{1+ \alpha_{k}\sqrt{\gamma_k}} = \frac{x_k + \alpha_{k}\sqrt{\gamma_k}(v_k-\alpha_k\nabla f(y_k)/\sqrt{\gamma_k})}{1+ \alpha_{k}\sqrt{\gamma_k}} \\
	={}&\frac{x_k + \alpha_{k}\sqrt{\gamma_k}v_k}{1+ \alpha_{k}\sqrt{\gamma_k}} -\frac{\alpha^2_k\nabla f(y_k)}{1+ \alpha_{k}\sqrt{\gamma_k}} =y_k-\frac{1}{L}\nabla f(y_k).
	\end{aligned}
\]
For other choice that violates the relation $L\alpha_k^2=  1+\alpha_k\sqrt{\gamma_k}$, we cannot obtain the equivalence. For simplicity, we will not consider general choices here.
\section{A Family of Nesterov Accelerated Gradient Methods}
\label{sec:NAG}
%\LH{Emphasis the motivation of this part.}
 The last two sections treat the HB model \cref{eq:hb} and the AVD model \cref{eq:mf} for strongly convex case ($\mu>0$) and convex case ($\mu=0$), respectively. Apart from this, we have not considered accelerated methods for the composite case $f = h+g$. 
 
 In this section, we shall propose a novel second order dynamical system called the {\it Hessian-driven Nesterov accelerated gradient} (HNAG) flow that involves a built-in time scaling and unifies the analysis for $\mu\geqslant 0$. We will design several accelerated first order optimization methods based on numerical discretizations of our HNAG flow system. Moreover, we extend this model to the composite setting and propose two accelerated proximal gradient methods. As before, the convergence analysis will be established via the strong Lyapunov condition.  
% 
% In this section, we present an ODE model that involves a built-in scaling and unifies the analysis for $\mu\geqslant 0$. Also, we extend this model to the composite setting and propose two accelerated proximal gradient methods.
\subsection{Nesterov accelerated gradient flow}
\label{sec:HNAG}
In our recent work~\citep{luo_chen_differential_2019}, for $f\in\mathcal S_\mu^1$ with $\mu\geqslant 0$, we have introduced a new ODE model \begin{equation}\label{eq:ode-agf}
	\gamma x''+(\gamma+\mu)x'+\nabla f(x) = 0, \quad \gamma' = \mu - \gamma,
\end{equation}
with initial conditions $x(0) = x_0,\,x'(0) = x_1$ and $\gamma(0)=\gamma_0>0$. For algorithmic designing and convergence analysis, we prefer the alternative formulation as an ODE system
\begin{equation}\label{eq:agf-intro}
	\left\{
	\begin{aligned}
		x' = {}&v-x,\\
		v'={}&\frac{\mu}{\gamma}(x-v)-\frac{1}{\gamma}\nabla f(x),\\
		\gamma'={}&\mu-\gamma.
	\end{aligned}
	\right.
\end{equation}
An appropriate numerical discretization of \cref{eq:agf-intro} recoveries {\it exactly} Nesterov's optimal method constructed from estimate sequence~\citep[Chapter 2]{Nesterov:2013Introductory}. Hence, we call \cref{eq:ode-agf} and \cref{eq:agf-intro} {\it Nesterov accelerated gradient} (NAG) flows. Exponential decay of the Lyapunov function \cref{eq:Lt-AVD} has been established and it was also proved that Gauss-Seidel iteration with one extra gradient descent step lead to a variant of Nesterov accelerated gradient method; see~\citet{luo_chen_differential_2019}. 

Motivated by the dynamical inertial Newton (DIN) system proposed by~\citet{alvarez_second-order_2002} and Hessian-driven damping models ~\citet{AttouchMaingeRedont2012,Attouch;Chbani;Fadili;Riahi:2020First-Order},  we further propose a new second order dynamical system, which is called the {\it Hessian-driven Nesterov accelerated gradient} (HNAG) flow and reads as follows
\begin{equation}\label{eq:NAG-Hessian}
	\gamma x''+(\gamma+\mu) x'+\beta\gamma\nabla^2f(x)x'+(1+\mu\beta+\gamma\beta')\nabla f(x)=0,
\end{equation}
where $\beta>0$ is any positive smooth (continuous differentiable) function on $[0,\infty)$ and $\gamma$ is the same time scaling factor as that in \cref{eq:ode-agf}. 
%Notice that, compared with \cref{eq:ode-agf}, the Hessian data in \cref{eq:NAG-Hessian} is due to the additional term 
%\begin{equation}\label{eq:diff-df}
%	\left(\beta\gamma\nabla f(x)\right)' +\beta\gamma\nabla f(x)= \beta\gamma\nabla^2f(x)x'+(\mu\beta+\gamma\beta')\nabla f(x).
%\end{equation}

%
%Inproposed the so-called 
%\begin{equation}\label{eq:hb-Hessian}
%x''+\gamma x'+r\nabla^2f(x)x'+\nabla f(x)=0,
%\end{equation}
%where $\gamma,\alpha>0$ are constant parameters. This model combines the heavy ball flow \cref{eq:hb} with the continuous Newton method
%\[
%\nabla^2f(x)x' + \nabla f(x) =0,
%\]
%which requires nondegenerate Hessian. Due to the extra Hessian term $r\nabla^2f(x)x'$, the coefficient of $x'$ is enlarged. 
%%According to our investigation of the 1D case in \cref{sec:HB-1d}, this not only makes the model \cref{eq:hb-Hessian} more stable, but also enhances the convergence rate; see numerical illustrations in \cref{sec:DIN-1d}.

Obviously, the HNAG flow model \cref{eq:NAG-Hessian} requires stronger smoothness $f\in \mathcal C^2\cap \mathcal S_\mu^1$ than NAG flow \cref{eq:ode-agf}. Therefore direct discretization based on \cref{eq:NAG-Hessian} is restrictive and might be expensive due to the existence of the Hessian matrix. Fortunately, as observed in~\citep{alvarez_second-order_2002}, 
%we observe from \cref{eq:diff-df} that the Hessian information can be {\it hidden} and \cref{eq:NAG-Hessian} is well defined even if $\nabla^2f(x)$ is degenerate. Indeed, we are allowed to rewrite \cref{eq:NAG-Hessian} as a novel 
if we write \eqref{eq:NAG-Hessian} as the first-order system
\begin{equation}\label{eq:Hagf-intro}
	\left\{
	\begin{aligned}
		x' = {}&v-x-\beta\nabla f(x),\\
		v'={}&\frac{\mu}{\gamma}(x-v)-\frac{1}{\gamma}\nabla f(x),\\
		\gamma'={}&\mu-\gamma,
	\end{aligned}
	\right.
\end{equation}
no $\nabla^2f(x)$ is needed. The formulation \eqref{eq:Hagf-intro} can be also thought of as a modified model of our previous NAG flow \cref{eq:agf-intro} by adding  one more damping term $-\beta\nabla f(x)$ to the system. In the next, we will see this minor modification brings faster decay of the gradient. Under standard assumption $f\in\mathcal S_{\mu,L}^1$ with $0\leqslant \mu\leqslant L<\infty$, existence and uniqueness of classical solution $(x,v)\in\mathcal C^1\times \mathcal C^1$ to \cref{eq:Hagf-intro} can be easily concluded from conventional theory of ODE.

For $\bs x=(x, v, \gamma)$, we still use the Lyapunov function
% \cref{eq:Lt-AVD}:
%\begin{equation}\label{eq:L-HNAG}
$\mathcal L(\bs x):=f(x)-f(x^*)+\frac{\gamma}{2}\nm{v-x^*}^2,$
%\end{equation}
and denote by $\mathcal G(\bm x)$ the right hand side of \cref{eq:Hagf-intro},
%\[
%\mathcal G(\bs x)
%%=
%%\begin{pmatrix}
%% \mathcal G^x\\
%% \mathcal G^v\\
%% \mathcal G^{\gamma}
%%\end{pmatrix}
%:=
%\begin{pmatrix}
%	v-x - \beta \nabla f(x)\\
%	\displaystyle \frac{\mu}{\gamma}(x-v)-\frac{1}{\gamma}\nabla f(x)\\
%	\mu-\gamma
%\end{pmatrix}.
%\]
which then becomes $\bm x' = \mathcal G(\bm x)$.
%A direct computation gives 
%\[
%\begin{split}
%	-\nabla \mathcal L(\bs x) \cdot \mathcal G(\bs x) 
%	={}& - \mu\dual{x-v,v-x^*}+\dual{\nabla f(x),x-x^*} \\
%	{}&\quad+ \beta\nm{\nabla f(x)}_*^2+ \frac{\gamma - \mu}{2}\nm{v-x^*}^2.
%	%\geqslant{} & \mathcal L(x,v,\gamma)+ \beta \nm{\nabla F(x)}_*^2 + \frac{\mu}{2}\nm{x-v}^2.
%\end{split}
%\] 
Observing the identity \cref{eq:squares}
%\[
%2(x-v,v-x^*)=\nm{x-x^*}^2-\nm{x-v}^2-\nm{v-x^*}^2
%\]
and the $\mu$-convexity of $f$ (cf. \cref{eq:Mxstar}),
%$$
%\dual{\nabla f(x),x-x^*} \geqslant f(x) - f(x^*) + \frac{\mu}{2} \nm{x-x^*}^2,
%$$
%we have 
a direct computation gives 
\begin{equation}\label{eq:A-HNAG}
\begin{split}
	-\nabla \mathcal L(\bs x) \cdot \mathcal G(\bs x) 
	={}& - \mu\dual{x-v,v-x^*}+\dual{\nabla f(x),x-x^*} \\
	{}&\quad+ \beta\nm{\nabla f(x)}_*^2+ \frac{\gamma - \mu}{2}\nm{v-x^*}^2\\
	\geqslant{}&  \mathcal L(\bs x)+ \beta \nm{\nabla f(x)}_*^2 + \frac{\mu}{2}\nm{x-v}^2.	
	%\geqslant{} & \mathcal L(x,v,\gamma)+ \beta \nm{\nabla F(x)}_*^2 + \frac{\mu}{2}\nm{x-v}^2.
\end{split}
%\] 
%\begin{equation}\label{eq:A-HNAG}
%	-\nabla \mathcal L(\bs x) \cdot \mathcal G(\bs x) 
%	\geqslant{}  \mathcal L(\bs x)+ \beta \nm{\nabla f(x)}_*^2 + \frac{\mu}{2}\nm{x-v}^2.
\end{equation}
Hence $\mathcal L$ is a strong Lyapunov function of  \cref{eq:Hagf-intro} and satisfies $\mathcal A(q,c,p)$ with $q= 1,\,c(\bs x)=1$, and $p^2(\bs x)=
\beta \nm{\nabla f(x)}_*^2+\frac{\mu}2\nm{x-v}^2$. 
Invoking \cref{thm:strongLya}, one can prove the exponential decay
\begin{equation}\label{eq:conv-hnag}
	\mathcal L(x(t)) +\int_{0}^{t}e^{s-t}\beta(s)\nm{\nabla f(x(s))}_*^2\dd s\leqslant e^{-t}\mathcal L(x(0)),\quad t\geqslant 0.
\end{equation}
%\begin{equation}\label{eq:A-non}
%-\partial \mathcal L(x, d)\cdot \mathcal G(x, d)\geq
%c(x)\mathcal L^q(x)+p^2(x),\quad \forall\,x\in W.
%%,\quad t\geqslant t_0.
%\end{equation}
%%If $c(x) =0$ for $x\in W$, then we call $\mathcal L(x)$ a locally {\it weak Lyapunov function}. 
%When $W=V$, we call $\mathcal L(x)$ a global Lyapunov function. 
%If $c(x)>0$, then we call $\mathcal L$ locally ($W\subset V$) or globally ($W=V$) {\it strong} Lyapunov function. We use $\mathcal A(c,q,p,W)$ to denote the strong Lyapunov condition \cref{eq:A} and use $\mathcal A(c,q,p)$ when $W = V$. 
Thanks to the built-in scaling factor $\gamma$, this holds true for $\mu \geqslant 0$ in a unified and simpler way. Additionally, as one may see from \cref{eq:conv-hnag}, the extra gradient norm square term $\beta \nm{\nabla f(x)}_*^2$ in \cref{eq:A-HNAG} brings faster decay of the gradient. 
%This can be inherit in appropriate numerical discretizations. 
%In the sequel, we shall use it to slightly enlarge the step size and thus speed up the convergence. 
%More importantly, we have faster convergence of the gradient which can be inherit by the corresponding discretizations. 
\subsection{Nesterov accelerated gradient method}\label{sec:HNAG-GS-2}
Let us apply the Gauss-Seidel type discretization to \cref{eq:Hagf-intro} and obtain
%A Gauss-Seidel scheme without extra gradient step is considered below
\begin{equation}\label{eq:ex-HNAG}
	\left\{
	\begin{aligned}
		\frac{x_{k+1}-x_{k}}{\alpha_k}={}& v_{k}-x_{k+1}-\beta_k\nabla f(x_{k}),\\
		\frac{v_{k+1}-v_{k}}{\alpha_k}={}&
		\frac{\mu }{\gamma_k}(x_{k+1}-v_{k+1})
		-\frac{1}{\gamma_k}\nabla f(x_{k+1}),\\
		\frac{\gamma_{k+1} - \gamma_{k} }{\alpha_k}  ={}& 
		\mu -\gamma_{k+1},
	\end{aligned}
	\right.
\end{equation}
where $\alpha_k>0$ is the time step size.
%and chose 
%\begin{equation}\label{eq:ab}
%\alpha_k=\sqrt{\frac{\gamma_k}{L}},\quad \beta_k = \frac{1}{L\alpha_k}.
%\end{equation}
Given the current iterate $\bm x_k=(x_k, v_k, \gamma_k)$, one compute $x_{k+1}$ and $v_{k+1}$ successively from the first and the second equations and then update the parameter $\gamma_{k+1}$ by the last one.
We have three parameters $(\alpha_k, \beta_k, \gamma_k)$ in \eqref{eq:ex-HNAG} and will set
\begin{equation}\label{eq:ab}
	\alpha_k \beta_k = 1/L, \quad L\alpha_k^2 = \gamma_k (2+\alpha_k), 
	\quad \alpha_k>0.
\end{equation}
%Note that a slight larger step size is used in \eqref{eq:ab}. 

Although there are two gradient evaluations in the $k$-th iteration of \cref{eq:ex-HNAG}, the second one $\nabla f(x_{k+1})$ can be reused in the $k+1$-th iteration for updating $x_{k+2}$. Moreover, introduce an extra variable 
\begin{equation}\label{eq:yk}
	y_k = x_k - \frac{1}{L}\nabla f(x_k),
\end{equation}
we can obtain an equivalent form of \cref{eq:ex-HNAG} which requires only one gradient evaluation in each iteration and has been summarized in \cref{algo:HNAG}.
\begin{algorithm}[H]
	\caption{NAG method for minimizing $f\in\mathcal S_{\mu,L}^1$ with $0\leqslant \mu\leqslant L<\infty$}
	\label{algo:HNAG}
	\begin{algorithmic}[1] 
		\REQUIRE  $\gamma_0 > 0,\,x_0,\,v_0 \in V$.
		\STATE Initialization $y_0 = x_{0} - \frac{1}{L}\nabla f(x_{0})$.
		\FOR{$k=0,1,\ldots$}		
		\STATE Compute $\alpha_k = \left (\gamma_k+ \sqrt{\gamma_k^2+8L\gamma_k}\right )/(2L)$.
		\smallskip
		\STATE Update $x_{k+1}= (y_k + \alpha_k v_k)/(1+ \alpha_k)$.
		\smallskip
		\STATE Compute $\displaystyle y_{k+1}= x_{k+1} - \frac{1}{L}\nabla f(x_{k+1})$.
		\smallskip
		\STATE Update $\displaystyle v_{k+1} =  \frac{\gamma_k v_k + \mu\alpha_k  x_{k+1}}{\gamma_k + \mu\alpha_k} + \frac{L\alpha_k}{\gamma_k + \mu\alpha_k } (y_{k+1}-x_{k+1})$.
		\smallskip
		\STATE Update $\gamma_{k+1} = (\mu\alpha_k + \gamma_k)/(1+\alpha_k)$.
		\ENDFOR		
	\end{algorithmic}
\end{algorithm}

\begin{lem}
	\label{lem:conv-ex1-ode-NAG}
	For \cref{algo:HNAG}, we have
	\begin{equation}\label{eq:conv1-ex1-ode-NAG}
		\mathcal L_{k+1} - \frac{1}{2L}\| \nabla f(x_{k+1})\|_*^2
		\leqslant 
		\frac{	1}{1+\alpha_k } \left ( \mathcal L_{k} - \frac{1}{2L}\| \nabla f(x_{k})\|_*^2\right ) \quad\forall\,k\geqslant 0,
	\end{equation}
	where $\mathcal L_k = \mathcal L(\bs x_k)
	= {}f(x_k)-f(x^*) + \frac{\gamma_k}{2} \nm{v_k-x^*}^2.$
\end{lem}
\begin{proof}
Following the proof of \cref{lem:one-step-AVD-semi}, we have the difference
$	\mathcal L_{k+1} - \mathcal L_k		=I_1 + I_2 + I_3$, where $I_1,\,I_2$ and $I_3$ are defined in \cref{eq:I1-AVD}.
%	We split the difference $		\mathcal L_{k+1} - \mathcal L_k$ along the path $\bs x_k=(x_k, v_{k},\gamma_{k})$ to $(x_{k+1}, v_{k},\gamma_{k})$ to $(x_{k+1}, v_{k+1},\gamma_{k})$ and finally to  $\bs x_{k+1}=(x_{k+1}, v_{k+1},\gamma_{k+1})$:
%	\begin{align*}
%{}&		\mathcal L_{k+1} - \mathcal L_k		:=I_1 + I_2 + I_3\\
%%		\\
%		=
%		{}&  
%		\mathcal L(x_{k+1}, v_{k},\gamma_{k}) - \mathcal L(x_k, v_{k},\gamma_{k}) \\
%		&+ \mathcal L(x_{k+1}, v_{k+1},\gamma_{k}) - \mathcal L(x_{k+1}, v_{k},\gamma_{k}) \\
%		&+\mathcal L(x_{k+1}, v_{k+1},\gamma_{k+1}) - \mathcal L(x_{k+1}, v_{k+1},\gamma_{k}).
%	\end{align*}
Below, we shall estimate these three terms one by one.
%	Note that the parameters $(\alpha_k,\beta_k)$ remains the same when estimating the three items. 
	
	As $\mathcal L$ is linear in terms of $\gamma$, we see
	\begin{equation}\label{eq:I3}
		I_3 = {}\dual{\nabla_{\gamma}\mathcal L(\bs x_{k+1}), \gamma_{k+1} - \gamma_{k}} ={} \alpha_k (\nabla_{\gamma}\mathcal L(\bs x_{k+1}), \mathcal G^{\gamma}(\bs x_{k+1})).
	\end{equation}
	For the second item $I_2$, we use the fact $\mathcal L(x_{k+1}, \cdot, \gamma_k)$ is $\gamma_k$-convex to get
\begin{equation}\label{eq:I2-}
		\begin{aligned}
		I_2\leqslant {} &\dual{\nabla_v \mathcal L(x_{k+1}, v_{k+1},\gamma_{k}), v_{k+1} - v_k} - \frac{\gamma_k}{2}\nm{ v_{k+1} - v_k}^2\\
		={}&\alpha_k \dual{\nabla_v \mathcal L(\bs x_{k+1}), \mathcal G^v(\bs x_{k+1})}- \frac{\gamma_k}{2}\nm{ v_{k+1} - v_k}^2.
	\end{aligned}
\end{equation}
	In the last step, as the parameter $\gamma$ is canceled in the product $\dual{\nabla_v\mathcal L(\bs x), \mathcal G^v(\bs x)}_*$, we can switch the variable $(x_{k+1}, \gamma_k)$ to $(x_{k+1}, \gamma_{k+1})$. 
	
	We now focus on the first one $I_1$:
	\begin{equation*}
		\begin{split}
			I_1
			\leqslant {}&\dual{\nabla_x \mathcal L(\bs x_{k+1}), x_{k+1} - x_k} - \frac{1}{2L}\nm{\nabla f(x_{k+1}) - \nabla f(x_k)}_*^2.
		\end{split}
	\end{equation*}
	In the first term, we can switch $(x_{k+1}, v_k,\gamma_k)$ to $\bs x_{k+1}$ because $\nabla_x \mathcal L = \nabla f(x)$ is independent of $(v,\gamma)$. Then we use the discretization \cref{eq:ex-HNAG} to replace $x_{k+1} - x_k$ and compare with the flow evaluated at $\bs x_{k+1}$:
	\begin{align*}
		\dual{\nabla_x \mathcal L(\bs x_{k+1}), x_{k+1} - x_k} 
		={}&  \alpha_k \dual{\nabla_x \mathcal L(\bs x_{k+1}), \mathcal G^x(\bs x_{k+1})} \\
		& +\alpha_k\beta_k ( \nabla f(x_{k+1}) , \nabla f(x_{k+1}) - \nabla f(x_k))\\
		& \quad+\alpha_k \dual{\nabla f(x_{k+1}), v_k - v_{k+1}}.
	\end{align*}
	Observing the bound \cref{eq:I2-} for $I_2$, we  use Cauchy\textendash Schwarz inequality to bound the last term as follows
	\begin{equation*}
		%	\label{eq:df-diff-vk}
		\begin{split}
			\alpha_k \| \nabla f(x_{k+1})\|_*\| v_k - v_{k+1}\|\leqslant {}&
			\frac{\alpha_k^2}{2\gamma_k} \| \nabla f(x_{k+1})\|_*^2 
			+ \frac{\gamma_k}{2}\| v_k - v_{k+1}\|^2.
		\end{split}
	\end{equation*}
	We use the identity \cref{eq:squares} for the cross term
	\begin{align*}
		&\alpha_k\beta_k( \nabla f(x_{k+1}) , \nabla f(x_{k+1}) - \nabla f(x_k)) \\
		= & - \frac{\alpha_k\beta_k}{2}\| \nabla f(x_{k})\|_*^2
		+ \frac{\alpha_k\beta_k}{2}\| \nabla f(x_{k+1})\|_*^2 + \frac{\alpha_k\beta_k}{2}\| \nabla f(x_{k+1}) - \nabla f(x_k) \|_*^2.
	\end{align*}
	
	Adding all together and applying strong Lyapunov property $\mathcal A(1,1,p^2)$ with $p^2 =  \beta \nm{\nabla f(x)}_*^2
	+ \frac{\mu}{2}\nm{x-v}^2$ at $x_{k+1}$ (but with $\beta_k$ not $\beta_{k+1}$) yields that 
	\begin{equation*}
		%	\label{diff-Lk}
		\begin{split}
			\mathcal L_{k+1} - \mathcal L_k
			\leqslant {}& -\alpha_k\mathcal L_{k+1}  -\frac{\alpha_k\beta_k}{2}\nm{\nabla f(x_k)}_*^2  \\
			& \quad +\frac{1}{2}\left ( \frac{\alpha_k^2}{\gamma_k}  - \alpha_k\beta_k \right )\nm{\nabla f(x_{k+1})}_*^2			\\
			& \qquad+\frac{1}{2}\left ( \alpha_k\beta_k - \frac{1}{L}\right )\| \nabla f(x_{k+1}) - \nabla f(x_k) \|^2.
		\end{split}
	\end{equation*}
	By our choice of parameters \cref{eq:ab}:
	\[
	\alpha_k\beta_k-\frac{1}{L}=0,\quad
	\frac{\alpha_k^2}{\gamma_k}  
	-\alpha_k\beta_k = (1+\alpha_k)\frac{1}{L},
	\]
and consequently,
	\begin{align*}
		\mathcal L_{k+1} - \mathcal L_k
		\leqslant {}&- \alpha_k \mathcal L_{k+1}
		-\frac{1}{2L}\nm{\nabla f(x_k)}^2 + (1+\alpha_k) \frac{1}{2L}\nm{\nabla f(x_{k+1})}^2.
	\end{align*}
	Rearranging the above inequality gives the desired estimate \cref{eq:conv1-ex1-ode-NAG}.
\end{proof}	
%	Finally, let us study the asymptotic behavior of $\rho_k$. 
% where $\rho_k$ is defined as
%	\begin{equation}\label{eq:lambdak}
%\rho_0=1,\quad\rho_k = 
%\prod_{i=0}^{k-1}\frac{1}{1+\alpha_i},\quad k\geqslant 1.
%\end{equation}
%
%	The formula of $\gamma_k$ yields
%	\[
%	\frac{1}{1+\alpha_k} = \frac{\gamma_{k+1}}{\gamma_k+\mu \alpha_k}
%	\leqslant \frac{\gamma_{k+1}}{\gamma_k},
%	\]
%	and it follows from \cref{eq:lambdak} that
%\[
%	\rho_k \leqslant 
%\frac{\gamma_k}{\gamma_0}
%=
%\frac{L\alpha_k^2}{\gamma_0}.
%\]

For the extra variable $y_k$ defined by \cref{eq:yk}, we have by \cref{eq:DL}  that
\[
f(y_k)\leqslant f(x_k) - \frac{1}{2L}\|\nabla f(x_k)\|_*^2,
\]
which implies $$\mathcal L_k - \frac{1}{2L}\|\nabla f(x_k)\|_*^2 \geqslant f(y_k) - f(x^*) + \frac{\gamma_k}{2}\| v_k - x^*\|^2 = \mathcal L(y_k,v_k,\gamma_k)\geqslant 0.$$ Therefore it is easy to conclude from \cref{eq:conv1-ex1-ode-NAG} that
\[
\mathcal L(y_k,v_k,\gamma_k)\leqslant \mathcal L_k - \frac{1}{2L}\|\nabla f(x_k)\|_*^2\leqslant \rho_k\left(\mathcal L_0 - \frac{1}{2L}\|\nabla f(x_0)\|_*^2\right),
%\times\prod_{i=0}^{k-1}\frac{1}{1+\alpha_i}.
\]
where $\rho_k$ is defined by \cref{eq:lambdak}.
%of the positive sequence $\mathcal L_k - \frac{1}{2L}\|\nabla f(x_k)\|_*^2$ and consequently $\mathcal L(y_k,v_k,\gamma_k)$. 
%Thanks to \cref{eq:decayrho-B>1/2} and \eqref{eq:ab}, we have the decay estimate 
%\begin{equation}\label{eq:est-rhok}
%	\rho_k 
%	%\prod_{i=0}^{k-1}\frac{1}{1+\alpha_i} 
%	\leqslant 
%	\min\left\{
%\left ( \frac{\sqrt{2}}{\sqrt{2} + k}\right )^2,
%	%\left(\frac{1}{1+\sqrt{2\min\{\gamma_0,\mu\}/L}}\right)^{k}
%	\left(1+\sqrt{\frac{2\mu}{L}}\right)^{-k}
%	\right\},
%\end{equation}
%\begin{equation}\label{eq:est-rhok}
%	\rho_k 
%	%\prod_{i=0}^{k-1}\frac{1}{1+\alpha_i} 
%	\leqslant 
%	\min\left\{
%	2L\left(\sqrt{2L}+\sqrt{\gamma_0}\, k\right)^{-2},\,
%	%\left(\frac{1}{1+\sqrt{2\min\{\gamma_0,\mu\}/L}}\right)^{k}
%	\left(1+\sqrt{\frac{2\gamma_{\min}}{L}}\right)^{-k}
%	\right\},
%\end{equation}
%where $\gamma_{\min} = \min\{\gamma_0,\mu\}$. 

According to the above discussion, we conclude the following result. As we see, due to the slightly larger step size, in the inequality \eqref{eq:alpha2lowerbound}, the constant $1$ is increased to $2$ and thus the rate \eqref{eq:est-rhok} is slightly better. 
\begin{thm}
	For \cref{algo:HNAG} with $\gamma_0 = r L\geqslant \mu$, we have
	\begin{equation*}
		%		\label{eq:conv1-ex1-ode-NAG}
		f(y_k) - f(x^*) + \frac{\gamma_k}{2}\| v_k - x^*\|^2 \leqslant  \rho_k \mathcal L_0 \quad\forall\,k\geqslant 0,
	\end{equation*}
where $\rho_k$ is defined by \cref{eq:lambdak} and satisfies the estimate 
\begin{equation}\label{eq:est-rhok}
	\rho_k 
	%\prod_{i=0}^{k-1}\frac{1}{1+\alpha_i} 
	\leqslant 
	\min\left\{
\left ( \frac{\sqrt{2}}{\sqrt{2} + \sqrt{r}\, k}\right )^2,
	%\left(\frac{1}{1+\sqrt{2\min\{\gamma_0,\mu\}/L}}\right)^{k}
	\left(1+\sqrt{\frac{2\mu}{L}}\right)^{-k}
	\right\}.
\end{equation}
%	$$
%	\rho_k \leqslant 
%	\min\left\{
%	2L\left(\sqrt{2L}+\sqrt{\gamma_0}\, k\right)^{-2},\,
%	\left(1+\sqrt{\frac{2\gamma_{\min}}{L}}\right)^{-k}
%	\right\}.
%	$$
\end{thm}
%Due to the slightly larger step size, the rate is also slightly larger. 
%\begin{rem}
%Explicit methods with extra gradient step can also be analyzed. The optimal convergence rate as that of the method \cref{eq:ex-HNAG} can be obtained as well, but with better convergence bound of the gradient; see~\citep{ChenLuo2019HNAG}.
%\end{rem}
%\LC{Present a simplified version for implementation.}
\subsection{Accelerated proximal gradient methods}
%\subsection{Accelerated proximal gradient method from HNAG flow}
We now move to the nonsmooth case. Let $f =h+g$ and assume that $f \in\mathcal S_{\mu}^{0}$ with $\mu\geqslant 0$, $h\in\mathcal S_L^1$ is the smooth part and the non-smooth part $g$ is convex and lower semicontinuous. 

In this setting, the HNAG flow \cref{eq:Hagf-intro} becomes
\begin{equation}\label{eq:Hagf-intro-nonsmooth}
	\left\{
	\begin{aligned}
		x' \in{}&v-x-\beta\partial f(x),\\
		v'\in{}&\frac{\mu}{\gamma}(x-v)-\frac{1}{\gamma}\partial f(x),\\
		\gamma'={}&\mu-\gamma.
	\end{aligned}
	\right.
\end{equation}
For $\bs x = (x,v,\gamma)$, let the right hand side of \cref{eq:Hagf-intro-nonsmooth} be $\mathcal G(\bm x)$ and we write $\mathcal G(\bm x,d)$ if $\partial f(x)$ is replaced by some $d\in\partial f(x)$. We still use 
%Let $\bs x = (x,v,\gamma)$. Consider the generalized HNAG flow \cref{eq:Hagf-intro}:
%\begin{equation}\label{eq:HNAG}
%	\bs x'(t)\in\mathcal G(\bs x):=
%	\begin{pmatrix}
%		v-x - \beta \partial f(x)\\
%		\displaystyle \frac{\mu}{\gamma}(x-v)-\frac{1}{\gamma}\partial f(x)\\
%		\mu-\gamma
%	\end{pmatrix}.
%\end{equation}
%We still use 
the Lyapunov function \cref{eq:Lt-AVD}.
%\[
%\mathcal L(\bs x):=f(x) - f(x^*)+\frac{\gamma}{2}\nm{v-x^*}^2.
%\]
Similar with \cref{eq:A-HNAG}, one can easily verify the strong Lyapunov property: for any $d\in\partial f(x)$, 
%\[
%\begin{aligned}
%	-\partial \mathcal L(\bs x, d)\cdot \mathcal G(\bs x, d)=
%	\end{aligned}
%\]
\begin{equation}\label{eq:A-non-Lt-HNAG}
	-\partial \mathcal L(\bs x, d)\cdot \mathcal G(\bs x, d)\geqslant
	\mathcal L(\bs x)+ \beta \nm{d}_*^2
	+ \frac{\mu}{2}\nm{x-v}^2.
	%,\quad t\geqslant t_0.
\end{equation}
Yet, unlike the smooth case \cref{eq:Hagf-intro}, it is generally hard to obtain classical $\mathcal C^1$ solution $(x,v)$ of \cref{eq:Hagf-intro-nonsmooth} since the subdifferential $\partial f(x)$ is a set-valued mapping and discontinuity may occur in $x'$ and $v'$. Also it is nontrivial to establish the corresponding nonsmooth version of the exponential decay \cref{eq:conv-hnag}. Here we skip further discussion and investigation on these topics but restrict ourselves to algorithm analysis based on the strong Lyapunov condition \cref{eq:A-non-Lt-HNAG}.
%Hence $\mathcal L(\bs x)$ satisfies 
%$\mathcal A(1,1,p^2)$ with 
%\[
%p^2 =  \beta \nm{\nabla f(x)}_*^2
%+ \frac{\mu}{2}\nm{x-v}^2.
%\]

%In the discretization, we shall always chose $\beta_k$ s.t. $\alpha_k \beta_k = 1/L$ and thus eliminate $\beta_k$ in the scheme. 
%Then we consider discretization. 
To utilize the separable structure of $f = h+g$, given the previous iterate $(x_k, v_k)$, we first find $x_{k+1}$ such that
\begin{equation}\label{eq:APG-1}
	\frac{x_{k+1}-x_{k}}{\alpha_k}\in{} v_{k}-x_{k+1}-\beta_k\nabla h(x_{k})-\beta_k\partial g(x_{k+1}),
\end{equation}
where the operator splitting, also known as forward-backward method, is used.
Let $y_k = x_k - \alpha_k \beta_k \nabla h(x_k)$, then we can write $x_{k+1}={} \proxi_{s_k g}(w_k)$ where
$$
w_k=\frac{y_k+\alpha_k v_k}{1+\alpha_k},\quad s_k=\frac{\alpha_k\beta_k}{1+ \alpha_k}.
$$
Note that we also have
\begin{equation}\label{eq:qk1}
	q_{k+1}:=\frac{w_k-x_{k+1}}{s_k}\in\partial g(x_{k+1}),
\end{equation}
which can be reused to discretize the second equation of \cref{eq:Hagf-intro-nonsmooth}. 

In summary, we obtain
%
%is then reused to update $v_{k+1}$ from 
%\[
%		\frac{v_{k+1}-v_{k}}{\alpha_k}={}
%\frac{\mu }{\gamma_k}(x_{k+1}-v_{k+1})
%-\frac{1}{\gamma_k}\left(\nabla h(x_{k+1})+q_{k+1}\right).
%\]
%%Once we obtained  $x_{k+1}$, a direction in $\partial g(x_{k+1})$ can be computed using \eqref{eq:APG-1}:
%%\[
%%q_{k+1} := \frac{1}{\beta_k}\left(v_{k}-x_{k+1}-\beta_k\nabla h(x_{k})-\frac{x_{k+1}-x_{k}}{\alpha_k}\right)
%%\in\partial g(x_{k+1}).
%%\]
%
%Then we can write 
a semi-implicit scheme for \cref{eq:Hagf-intro-nonsmooth}:
\begin{equation}\label{eq:ex-HODE}
	\left\{
	\begin{aligned}
		\frac{x_{k+1}-x_{k}}{\alpha_k} = {}& v_{k}-x_{k+1}-\beta_k(\nabla h(x_{k}) +q_{k+1}),\\
		\frac{v_{k+1}-v_{k}}{\alpha_k}={}&
		\frac{\mu }{\gamma_k}(x_{k+1}-v_{k+1})
		-\frac{1}{\gamma_k}\left(\nabla h(x_{k+1})+q_{k+1}\right),\\
		\frac{\gamma_{k+1} - \gamma_{k} }{\alpha_k}  ={}& 
		\mu -\gamma_{k+1},
	\end{aligned}
	\right.
\end{equation}
where $q_{k+1}$ is defined by \cref{eq:qk1}.
We chose the parameters $\alpha_k$ and $\beta_k$ by the rule
\begin{equation}\label{eq:split-ab}
	\alpha_k = \sqrt{\frac{\gamma_k}{L}},\quad \alpha_k\beta_k = \frac{1}{L},
\end{equation}
and simplify \cref{eq:ex-HODE} to obtain the following algorithm which is named by accelerated proximal gradient (APG) method.
\begin{algorithm}[H]
%	\caption{Accelerated Proximal Gradient method for solving $\min_x (h(x)+g(x))$}
	\caption{APG method for minimizing $f=h+g,\,h\in\mathcal S_{\mu,L}^1$ with $0\leqslant \mu\leqslant L<\infty$}
	\label{algo:APGHNAG}
	\begin{algorithmic}[1] 
%		\REQUIRE  $y_0, v_0 \in V$, $\alpha_0 > 0$, $0\leqslant \mu < L$.
		\REQUIRE  $\gamma_0 > 0,\,x_0,\,v_0 \in V$.
\STATE Initialization $y_0 = x_{0} - \frac{1}{L}\nabla h(x_{0})$.
		\FOR{$k=0,1,\ldots$}		
		\smallskip
		\STATE Compute $\displaystyle w_k = \frac{y_k+\alpha_k v_k}{1+\alpha_k}$ and $\displaystyle s_k=\frac{1}{L(1+\alpha_k)}$.
		\STATE Update $\displaystyle x_{k+1}= \proxi_{s_k g}\left(w_k\right)$.
		\smallskip
		\STATE Compute $\displaystyle y_{k+1}= x_{k+1} - \frac{1}{L}\nabla h(x_{k+1})$.
		\smallskip
		\STATE Update $v_{k+1} =x_{k+1} + (y_{k+1} - y_k)/(\alpha_k + \mu/L)$.
		\smallskip
		\STATE Update $\alpha_{k+1} = \sqrt{(\alpha_k^2 +  \alpha_k\mu/L)/(1+\alpha_k)}$.
		\ENDFOR		
	\end{algorithmic}
\end{algorithm}

%Again we present a simplified version for the easy of implementation. 
%\begin{prop}\label{prop:APG}
%Given an initial guess $x_0$, set $\alpha_0 = 1, y_0 = x_0 - \nabla h(x_0), v_0 = x_0$. The following scheme is equivalent to \cref{eq:ex-HODE}-\cref{eq:split-ab}: for $k\geqslant 0$, given $(y_k,v_k,\alpha_k)$, update $(y_{k+1}, v_{k+1}, \alpha_{k+1})$ by 
% \begin{equation}\label{eq:PGequivalent}
%\left\{
%\begin{aligned}
%%y_{k}={}& x_{k} - \frac{1}{L}\nabla h(x_{k}),\\
%x_{k+1}={}& \proxi_{s_k g}\left(\frac{y_k+\alpha_k v_k}{1+\alpha_k}\right),\quad s_k=\frac{1}{L(1+\alpha_k)},\\
%y_{k+1}={}& x_{k+1} - \frac{1}{L}\nabla h(x_{k+1}),\\
%v_{k+1} ={}& x_{k+1} + \frac{ y_{k+1} - y_k}{\alpha_k + \mu/L},\\
%\alpha_{k+1} ={}& \sqrt{\frac{\alpha_k^2 +  \alpha_k\mu/L}{1+\alpha_k}}.
%\end{aligned}
%\right.
%\end{equation}
%\end{prop}
%For $\mu = 0$, the update of $v_{k+1}$ and $\alpha_{k+1}$ can be further simplified to:
%\begin{equation}\label{eq:HNAGvsimple2}
%v_{k+1} = x_{k+1} + \frac{1}{\alpha_k} (y_{k+1} - y_{k}), \quad \alpha_{k+1} = \frac{\alpha_k}{\sqrt{1+\alpha_k}}.
%\end{equation}

%Invoking the discrete Lyapunov function \cref{eq:Lk-HNAG}, the following theorem presents the convergence analysis.
\begin{thm}
	\label{thm:conv-split}
	For \cref{algo:APGHNAG},
%	the splitting scheme \cref{eq:ex-HODE} with parameter 
%	setting \cref{eq:split-ab} or equivalently Algorithm \ref{algo:APGHNAG}, 
	we have the contraction property
	\begin{equation}\label{eq:conv-split}
		\mathcal L_{k+1}
		\leqslant 
		\frac{	 \mathcal L_k}{1+\alpha_k } - \frac{\| \nabla h(x_{k}) +q_{k+1} \|_*^2}{1+\alpha_k}\quad\forall\, k\geqslant 0,
	\end{equation} 
	where $\mathcal L_k
	= \mathcal L(\bs x_k)
	= {}f(x_k)-f(x^*) +
	\frac{\gamma_k}{2}
	\nm{v_k-x^*}^2.$
	%	where $\mathcal L_k$ is defined in \cref{eq:Lk},
When $\gamma_0=rL \geqslant \mu$, it holds that
	\begin{equation}\label{eq:conv-algo2}
		\mathcal L_k
		\leqslant\mathcal L_0\times 			\min\left\{
			\left (\frac{\sqrt{r+1}+1}{\sqrt{r+1}+1 + \sqrt{r} k} \right )^2,\,
			\left(1+\sqrt{\frac{\mu}{L}}\right)^{-k}
			\right\}.
	\end{equation}
\end{thm}
\begin{proof}
Following the proof of \cref{lem:conv-ex1-ode-NAG}, we have the difference
	$	\mathcal L_{k+1} - \mathcal L_k		=I_1 + I_2 + I_3$, where $I_1,\,I_2$ and $I_3$ are defined in \cref{eq:I1-AVD}.
%	\begin{align*}
%		\mathcal L_{k+1} - \mathcal L_k
%		={}&  \mathcal L(x_{k+1}, v_{k},\gamma_{k}) - \mathcal L(x_k, v_{k},\gamma_{k}) \\
%		&+ \mathcal L(x_{k+1}, v_{k+1},\gamma_{k}) - \mathcal L(x_{k+1}, v_{k},\gamma_{k}) \\
%		&+\mathcal L(x_{k+1}, v_{k+1},\gamma_{k+1}) - \mathcal L(x_{k+1}, v_{k+1},\gamma_{k}) \\
%		:={}& I_1 + I_2 + I_3,
%	\end{align*}

Clearly, the estimates \cref{eq:I2-,eq:I3} for $I_2$ and $I_3$ keep unchanged here:
\begin{align*}
		I_3 &={} \alpha_k (\nabla_{\gamma}\mathcal L(\bs x_{k+1}), \mathcal G^{\gamma}(\bs x_{k+1})),\\
		\label{eq:I2-HNAG}
	I_2 &\leqslant{}  \alpha_k \dual{\nabla_v \mathcal L(\bs x_{k+1}), \mathcal G^v(\bs x_{k+1}, d_{k+1})} - \frac{\gamma_k}{2}\nm{ v_{k+1} - v_k}^2,
\end{align*}
	where $d_{k+1} := \nabla h(x_{k+1})+ q_{k+1}\in\partial f(x_{k+1})$ and $q_{k+1}$ is defined by \cref{eq:qk1}.
	Observing that 
	\begin{equation*}
		I_1 =  \mathcal L(x_{k+1}, v_{k},\gamma_{k}) - \mathcal L(x_k, v_{k},\gamma_{k}) 
		= g(x_{k+1})- g(x_{k})+h(x_{k+1})- h(x_{k}),
	\end{equation*}
	we use the fact $q_{k+1}\in \partial g(x_{k+1})$ and \cref{eq:philowerL} to estimate $I_1$ as follows
	\begin{equation}\label{eq:I1}
		I_1 \leqslant{}
		\dual{d_{k+1}, x_{k+1} - x_k} 
		- \frac{1}{2L}\nm{\nabla h(x_{k+1}) - \nabla h(x_k)}_*^2.
	\end{equation}
	We use the discretization \cref{eq:ex-HODE} to replace $x_{k+1} - x_k$ and compare with the flow evaluated at $\bs x_{k+1}=(x_{k+1}, v_{k+1}, \gamma_{k+1})$:
	\begin{align*}
		\dual{d_{k+1}, x_{k+1} - x_k} 
		={}&  \alpha_k \dual{d_{k+1}, \mathcal G^x(\bs x_{k+1}, d_{k+1})} \\
		& +\alpha_k\beta_k \dual{d_{k+1}, \nabla h(x_{k+1}) - \nabla h(x_k)}\\
		& \quad+\alpha_k \dual{d_{k+1}, v_k - v_{k+1}}.
	\end{align*}
	Thanks to the negative term in the bound of $I_2$, the last term is bounded by
	\[
	\alpha_k \|d_{k+1}\|_*\| v_k - v_{k+1}\|\leqslant {}
	\frac{\alpha_k^2}{2\gamma_k} \|d_{k+1}\|_*^2 
	+ \frac{\gamma_k}{2}\| v_k - v_{k+1}\|^2.
	\]
	The cross term is expanded by combination of squares (cf. \eqref{eq:squares}), 
	\begin{align*}
		\frac{1}{L}\dual{d_{k+1}, \nabla h(x_{k+1}) - \nabla h(x_k)} ={}& - \frac{1}{2L}\| \nabla h(x_{k}) +q_{k+1} \|_*^2 + \frac{1}{2L}\| d_{k+1}\|_*^2 \\
		&+ \frac{1}{2L}\|  \nabla h(x_{k+1}) - \nabla h(x_k) \|_*^2
	\end{align*}
	%
	%
	%	
	%	
	%	Thanks to the negative term in \cref{eq:I1}, we bound the second term by that
	%	\[
	%	\alpha_k\beta_k \dual{d_{k+1}, \nabla h(x_{k+1}) - \nabla h(x_k)}\leqslant 
	%\frac{L\alpha_k^2\beta_k^2}{2}\nm{d_{k+1}}_*^2 + \frac{1}{2L}\nm{\nabla h(x_{k+1}) - \nabla h(x_k)}^2.
	%	\]
	We now get the estimate for $I_1$ as follows
	\begin{equation*}
		\begin{aligned}
			I_1 \leqslant{}& \alpha_k \dual{d_{k+1}, \mathcal G^x(\bs x_{k+1}, d_{k+1})}  +\frac{\gamma_k}{2}\| v_k - v_{k+1}\|^2 +\left(\frac{1}{2L}
			+\frac{\alpha_k^2}{2\gamma_k}\right)\nm{d_{k+1}}_*^2\\
			&\quad - \frac{1}{2L}\| \nabla h(x_{k}) +q_{k+1} \|_*^2.
		\end{aligned}
	\end{equation*}

	Putting all together and using strong Lyapunov property \cref{eq:A-non-Lt-HNAG} imply that
	\begin{align*}
		\mathcal L_{k+1} - \mathcal L_k
		\leqslant {}&\alpha_k ( \nabla\mathcal L(\bs x_{k+1},d_{k+1}), \mathcal G(\bs x_{k+1}))\\
		{}&\quad+\left(\frac{1}{2L}
		+\frac{\alpha_k^2}{2\gamma_k}\right)\nm{d_{k+1}}_*^2 - \frac{1}{2L}\| \nabla h(x_{k}) +q_{k+1} \|_*^2\\
		\leqslant {}&- \alpha_k \mathcal L_{k+1}+\left(\frac{1}{2L}
		+\frac{\alpha_k^2}{2\gamma_k}-\frac{1}{L}\right)\nm{d_{k+1}}_*^2 - \frac{1}{2L}\| \nabla h(x_{k}) +q_{k+1} \|_*^2\\
		={}& - \alpha_k \mathcal L_{k+1} - \frac{1}{2L}\| \nabla h(x_{k}) +q_{k+1} \|_*^2.
	\end{align*}
This proves \cref{eq:conv-split}. The final decay rate comes from \cref{eq:decayrho-B0} and \eqref{eq:split-ab}.
\end{proof}

If we choose
\begin{equation}\label{eq:ak-bk}
	\alpha_k = \sqrt{\frac{\gamma_k}{4L}},\quad \beta_k = \frac{1}{2L\alpha_k},
\end{equation}
then we have the identity
\begin{equation}\label{eq:key-id}
	\frac{\alpha_k^2\beta_k^2L}{2}
	+\frac{\alpha_k^2}{2\gamma_k}-\alpha_k\beta_k 
	=-\frac{\alpha_k\beta_k}{2}=-\frac{1}{4L}.
\end{equation}
This will keep the negative term $-\nm{d_{k+1}}_*^2$ and implies the faster convergence of gradient; see the theorem below. As the proof is a simple modification of previous one, we only present the result below.
\begin{thm}
	\label{thm:conv-split-2}
	If $f=h+g$ where $h\in\mathcal S_{\mu,L}^1$ with $0\leqslant \mu\leqslant L<\infty$, then for the splitting scheme \cref{eq:ex-HODE} with parameter 
	setting \cref{eq:ak-bk}, we have the contraction property
	\begin{equation}\label{eq:diff-Lk-split-rem-2}
		\mathcal L_{k+1} - \mathcal L_k
		\leqslant - \alpha_k \mathcal L_{k+1}
		-\frac{1}{4L}\nm{d_{k+1}}^2,
	\end{equation}
	%	where $\mathcal L_k$ is defined in \cref{eq:Lk},
	and it follows that
	\begin{equation}\label{eq:conv-split-rem-2}
		\mathcal L_k+\frac{1}{4L}\sum_{i=0}^{k-1}
		\frac{\rho_k}{\rho_i} \nm{d_{i+1}}_*^2
		\leqslant\rho_k\mathcal L_0.
	\end{equation}
Above $\rho_k$ is defined by \cref{eq:lambdak} and satisfies the estimate, when $\gamma_0 = r L\geqslant \mu$, 
	\begin{equation}\label{eq:conv-algo2-2}
		\rho_{k} \leqslant 
		\min\left\{
		\left (\frac{2+\sqrt{r+4}}{2+\sqrt{r+4} +\sqrt{r}\, k} \right )^2,\,
		\left(1+\sqrt{\frac{\mu}{4L}}\right)^{-k}
		\right\}.
	\end{equation}
%	where $\gamma_{\max} = \max\{\gamma_0,\mu\}$ and $\gamma_{\min} = \min\{\gamma_0,\mu\}$. 
	%Particularly, taking $\gamma_0=L$ gives 
	%\[
	%			\rho_{k} \leqslant 
	%\min\left\{
	%\frac{(2\sqrt{L}+\sqrt{4L+\gamma_{\max}})^2}{(\sqrt{\gamma_0}k+\sqrt{4L}+2\sqrt{L+\gamma_{\max}})^2},\,
	%\left(1+\sqrt{\frac{\mu}{4L}}\right)^{-k}
	%\right\}.
	%\]
\end{thm}
The bound of the sub-gradient yields that 
	\[
	\min_{0\leqslant i\leqslant k}
	\nm{d_{i+1}}_*^2\leqslant \frac{4L\mathcal L_0}{\sum_{i=0}^{k}
		1/\rho_i },
	\]
	and asymptotically, we have a faster decay rate of gradient $\nm{d_{i+1}}_*^2 = o(2L\mathcal L_0\rho_k)$.

\subsection{Another accelerated proximal gradient method}
We now apply the operator splitting to the first equation of \cref{eq:Hagf-intro-nonsmooth} by considering $x'=v-x$ first and $x'\in - \beta (\nabla h(x) + \partial g(x))$ next. The later can be further split via the proximal-gradient method. That is
%\LH{Question: Should it be corrected as follows?
%\[
%\frac{x_{k+1}-x_k}{\alpha_k}\in v_k-x_{k+1}-\beta_k\left(\nabla h(y_k)+\partial g(x_{k+1})\right)
%\]
%$\Longrightarrow$
%\[
%\frac{x_{k+1}-(y_k-s_k\nabla h(y_k))}{s_k}\in-\partial g(x_{k+1}),
%\]
%where $y_k=\frac{x_k+\alpha_kv_k}{1+\alpha_k}$ and $s_k = \alpha_k\beta_k/(1+\alpha_k)$. This gives 
%\[
%\left\{
%\begin{aligned}
%{}&			\frac{y_k-x_{k}}{\alpha_k}= v_{k}-y_k,\\
%{}&	x_{k+1} = \proxi_{s_k g}\left(y_k-s_k\nabla h(y_k)\right).
%	\end{aligned}
%\right.
%\]
%}
\begin{equation*}
%	\label{eq:APG0}
	\left\{
	\begin{aligned}
		{}&		\frac{y_k-x_{k}}{\alpha_k}= v_{k}-y_k,\\
		{}&		x_{k+1} = \proxi_{\alpha_k\beta_k g}(y_{k}-\alpha_k\beta_k\nabla h(y_k)).
	\end{aligned}
	\right. 
\end{equation*}
Letting 
\begin{equation}\label{eq:d-yk}
	d_f(y_k):=\frac{y_k - x_{k+1}}{\alpha_k\beta_k}\in\partial g(x_{k+1})+\nabla h(y_k),
\end{equation}
 we have the following `middle' point discretization of HNAG flow \cref{eq:Hagf-intro-nonsmooth}:
\begin{equation}\label{eq:ex1-ode-OAG}
	\left\{
	\begin{aligned}
		%	\frac{			\gamma_{k+1}  -\gamma_{k}  }{\alpha_k}={}&
		%	\mu-\gamma_{k+1},\\
		%	\frac{y_k-x_{k}}{\alpha_k}= {}&	v_{k}-y_k,\\		
		\frac{x_{k+1} - x_{k}}{\alpha_k} ={}&v_{k}-y_k - \beta_k d_f(y_k),\\
		\frac{v_{k+1}-v_{k}}{\alpha_k}={}&
		\frac{\mu}{\gamma_k}(y_k-v_{k+1})
		-\frac{1}{\gamma_k}
		d_f(y_{k}),\\
		\frac{		\gamma_{k+1}  - \gamma_{k}}{\alpha_k}  
		={}&\mu-\gamma_{k+1}.
	\end{aligned}
	\right.
\end{equation}
The point $y_k$ is an intermediate point of $x_k$ and $x_{k+1}$, and in the vector field $\mathcal G(x,v,\gamma)$ the first variable will be evaluated at $y_k$. We note that $d_f(y_k)$ is nothing but the gradient mapping used in \cite{luo_chen_differential_2019}; see also $d_{k+1/2}$ defined in Lemma \ref{lm:pgdecay}.

We use the step size
\begin{equation}\label{eq:apgpara}
	L\alpha_k^2=\gamma_k(1+\alpha_k),\quad \alpha_k\beta_k = 1/L,
\end{equation}
and summarize \cref{eq:ex1-ode-OAG} in \cref{algo:APG}. Note that $\alpha_k$ is slightly larger than that  in \eqref{eq:split-ab}.

\begin{algorithm}[H]
		\caption{New APG method for minimizing $f=h+g,\,h\in\mathcal S_{\mu,L}^1$ with $0\leqslant \mu\leqslant L<\infty$}
	\label{algo:APG}
	\begin{algorithmic}[1] 
		%		\REQUIRE  $y_0, v_0 \in V$, $\alpha_0 > 0$, $0\leqslant \mu < L$.
		\REQUIRE  $\gamma_0 > 0,\,x_0,\,v_0 \in V$ and $s = 1/L$.
%		\STATE Initialization $y_0 = x_{0} - \frac{1}{L}\nabla h(x_{0})$.
%	\caption{APG method for solving $\min_{x\in V} \left [h(x)+g(x) \right ]$}
%	\label{algo:APG}
%	\begin{algorithmic}[1] 
%		\REQUIRE  $x_0,v_0\in V,$ $0\leqslant \mu < L$, $\gamma_0 \geqslant \mu$.
		\FOR{$k=0,1,\ldots$}
		\smallskip
		\STATE Compute $\alpha_k = \left (\gamma_k+ \sqrt{\gamma_k^2+4L\gamma_k}\right )/(2L)$.
		\smallskip				
		\STATE Compute $\displaystyle{y_k =\frac{x_k+\alpha_kv_k}{1+\alpha_k}}$.
		\smallskip
		\STATE Update $\displaystyle{x_{k+1} =\proxi_{s g}(y_k-s\nabla h(y_k))}$.
		\smallskip		
		\STATE Update $\displaystyle{v_{k+1} = \frac{\gamma_kv_k+\mu\alpha_ky_k}{\gamma_k+\mu\alpha_k} +\frac{\gamma_k (1+\alpha_k)}{\gamma_k+\mu\alpha_k} \frac{x_{k+1}-y_k}{\alpha_k}}$.		
		\smallskip
		\STATE Update $\displaystyle{\gamma_{k+1} = \frac{\gamma_k+\mu\alpha_k}{1+\alpha_k}}$.			
		\ENDFOR
	\end{algorithmic}
\end{algorithm}

We will establish the convergence analysis via the strong Lyapunov condition. A key tool is the following estimate at $y$, which allows us to modify \cref{eq:A-non-Lt-HNAG} for later use.
\begin{lem}
%	Let $f =h+g$ and 
	Assume $f =h+g$ and $h\in\mathcal S_{\mu,L}^1$ with $0\leqslant \mu\leqslant L<\infty$.  Let $x = \proxi_{sg}(y - s \nabla h(y))$ with $s=1/L$ and $d_f(y) = (y - x)/s$. We have the following inequality
	\begin{equation}\label{eq:pgconvex}
		\dual{d_f(y), y - x^*} \geqslant f(x) - f(x^*) + \frac{\mu}{2}\| y - x^*\|^2 + \frac{1}{2L}\| d_f(y) \|_*^2.
	\end{equation}
\end{lem}
\begin{proof}
	Let $q(x) = d_f(y) - \nabla h(y)$. Then by definition $q(x)\in \partial g(x)$ and thus 
	\begin{equation*}
%		\label{eq:convexg}
		g(x) - g(x^*) \leqslant \dual{q(x), x - x^*}.
	\end{equation*}
	For $h\in\mathcal S_{\mu,L}^1$, we use Lemma \ref{lm:3pts} to conclude
	\begin{equation*}
%		\label{eq:convexh}
		h(x) - h(x^*) \leqslant \dual{\nabla h(y), x - x^*} - \frac{\mu}{2}\| y - x^*\|^2 + \frac{L}{2}\| y- x\|^2
	\end{equation*}
%	Adding \eqref{eq:convexg} and \eqref{eq:convexh} together yields that
Adding the above two estimates together yields that
	$$
	f(x) - f(x^*) \leqslant \dual{ d_f(y), x - x^*}- \frac{\mu}{2}\| y - x^*\|^2 + \frac{L}{2}\| y- x\|^2.
	$$
	Now split $\dual{ d_f(y), x - x^*} = \dual{ d_f(y), y - x^*} + \dual{ d_f(y), x - y}$ and use the fact $ x - y = - d_f(y)/L$ to get the desired estimate \eqref{eq:pgconvex}. 
\end{proof}
%
%Although the composite gradient mapping $\mathcal G_f(y_k)$ is used twice in both two algorithms, it needs to be evaluated only once in each iteration. Hence, the computational cost of those two methods are the same and we see later they can achieve the same accelerated rate $O(\min(1/k^2,(1-\sqrt{\mu/L})^k))$. More importantly, they are as a proof of the effective and usefulness of our ODE solver approach, by which we can construct accelerated methods without knowing them a priori.
\begin{thm}\label{thm:conv-APG}
	For \cref{algo:APG}, we have
	\begin{equation}\label{eq:conv-APG}
		\mathcal L_{k+1}
		\leqslant 
		\frac{	 \mathcal L_k}{1+\alpha_k }\quad\forall\,k\in\mathbb N,
	\end{equation}
	and this implies, when $\gamma_0 = r L\geqslant \mu$,
	\begin{equation}\label{eq:conv-algo2-}
		\mathcal L_k
		\leqslant\mathcal L_0\times \min\left\{
			\left (\frac{2}{2 + \sqrt{r} \, k} \right )^2,\,
		\left(1+\sqrt{\frac{\mu}{L}}\right)^{-k}
		\right\}.
	\end{equation}
%	where $\gamma_{\min} = \min\{\gamma_0,\mu\}$.
\end{thm}
\begin{proof}
Using the refined convexity lower bound \eqref{eq:pgconvex}, the strong Lyapunov property at $y_k$ reads as
	\begin{equation}\label{eq:strongLy}
		\begin{split}
			&-\partial \mathcal{L}\left(d_{f}\left(y_{k}\right), v_{k+1}, \gamma_{k+1}\right) \cdot \mathcal G\left(d_{f}\left(y_{k}\right), v_{k+1}, \gamma_{k+1}\right) \\
			\geqslant {}&\mathcal{L}\left(x_{k+1}, v_{k+1}, \gamma_{k+1}\right)+\left(\frac{1}{2 L}+\beta_k\right)\left\|d_{f}\left(y_{k}\right)\right\|_*^{2},
		\end{split}
	\end{equation}
	which can proved analogously to \cref{eq:A-non-Lt-HNAG}.

%	Following the proof of \cref{thm:conv-split}, 
	The estimate of 
	$	\mathcal L_{k+1} - \mathcal L_k		=I_1 + I_2 + I_3$ is almost in line with that of \cref{thm:conv-split}, where $I_1,\,I_2$ and $I_3$ are defined in \cref{eq:I1-AVD}. The difference comes from the first item $I_1$.
%
%use the split
%	$$
%	f(x_{k+1}) - f(x_k) = g(x_{k+1}) - g(x_k) + h(x_{k+1}) - h(x_k).
%	$$
%	$h\in\mathcal S_{\mu,L}^1$, 
	Recall \cref{eq:d-yk} and let $q(x_{k+1}) := d_f(y_k) - \nabla h(y_k) \in\partial g(x_{k+1})$. By convexity of $g$, it follows that
	\begin{equation*}
%		\label{eq:convexgk}
		g(x_{k+1}) - g(x_k) \leqslant \dual{q(x_{k+1}), x_{k+1} - x_k},
	\end{equation*}
and we use Lemma \ref{lm:3pts} to conclude
	\begin{equation*}
%		\label{eq:convexhk}
		h(x_{k+1}) - h(x_k) \leqslant \dual{\nabla h(y_k), x_{k+1} - x_k} + \frac{L}{2}\| x_{k+1}- y_k\|^2.
	\end{equation*} 
By \cref{eq:d-yk}, we see 
\begin{equation*}
%	\label{eq:xk1}
	x_{k+1} - y_k = - \frac{1}{L}d_f(y_k),
\end{equation*}
and a routine calculation yields the bound
	\begin{equation*}
		\begin{split}
			\mathcal L_{k+1} - \mathcal L_k \leqslant {}& \alpha_k \partial \mathcal{L}\left(d_{f}\left(y_{k}\right), v_{k+1}, \gamma_{k+1}\right) \cdot \mathcal G\left(d_{f}\left(y_{k}\right), v_{k+1}, \gamma_{k+1}\right)
			+ \left (\frac{1}{2L} + \frac{\alpha_k^2}{2\gamma_k}\right )\|d_f(y_k)\|_*^2.
		\end{split}
	\end{equation*}
	Applying the strong Lyapunov property \eqref{eq:strongLy} and the relation \cref{eq:apgpara}, we get
	$$
	\mathcal L_{k+1} - \mathcal L_k \leqslant {} - \alpha_k \mathcal L_{k+1} + \left (\frac{\alpha_k^2}{2\gamma_k} - \frac{\alpha_{k}}{2 L}-\frac{1}{2L} \right )\|d_f(y_k)\|_*^2= - \alpha_k \mathcal L_{k+1}.
	$$
	This proves \cref{eq:conv-APG}.
%	By the choice of $\alpha_k$, the second term is vanished and the desired result follows.
		To the end, recalling \cref{eq:decayrho-B>1/2} and \eqref{eq:apgpara} proves \cref{eq:conv-algo2-}.
\end{proof}
\section{Concluding Remarks}
\label{sec:conclude}
By using the tool of Lyapunov function and introducing the concept of strong Lyapunov condition, we present a unified self-contained framework for first-order optimization methods including gradient descent method, proximal point algorithm, proximal gradient method, heavy ball (momentum) method, and Nesterov accelerated gradient method.

However, we notice that a systematical way to find an appropriate Lyapunov function satisfying the strong Lyapunov condition is not presented in this work. When $\nabla f$ is linear, it is possible to use control theory to design one; see~\citet{lessard_analysis_2016} for more details. On the other hand, we have not considered non-Euclidean setting that involves the Bregman divergence (or preconditioning effect), which is related to the mirror descent models \citep{wibisono_variational_2016,krichene_accelerated_2015}.

Different from existing works using Lyapunov analysis and involving complicated algebraic calculations, the strong Lyapunov condition can be verified much more systematically by inequalities of convex functions. Besides, by suitable time scaling factor, we can handle the convex case and strong convex case in a unified way. Furthermore, the strong Lyapunov condition can be easily used in the discretization to establish the convergence of the algorithms. This together with continuous dynamical system renders effective tools for designing and analysis of convex optimization algorithms.

\vskip 0.2in
%\bibliography{sample}
%\bibliography{/Users/luohao/Desktop/GitHub/Mymacro/mylibrary}	
%\bibliography{Mylibrary}	
%\bibliography{OptimizationLong}
%\bibliography{Mylibrary 2}

\end{document}